\documentclass[11pt, twoside]{book}
\usepackage{amsmath,amsthm,amssymb}
\usepackage{times}
\usepackage{enumerate}
\usepackage{mathrsfs}
\usepackage{}
\usepackage{txfonts}
\usepackage{amsfonts}
\usepackage{amsfonts}
\usepackage{amssymb,latexsym}
\usepackage{shorttoc}
\def\titlerunning#1{\gdef\titrun{#1}}
\makeatletter
\def\author#1{\gdef\autrun{\def\and{\unskip, }#1}\gdef\@author{#1}}
\def\address#1{{\def\and{\\\hspace*{18pt}}\renewcommand{\thefootnote}{}%
\footnote {#1}}%
\markboth{\autrun}{\titrun}}
\makeatother
\def\email#1{e-mail: #1}
\def\subjclass#1{{\renewcommand{\thefootnote}{}%
\footnote{\emph{Mathematics Subject Classification (2010):} #1}}}
\def\keywords#1{\par\medskip
\noindent\textbf{Keywords.} #1}
\newcommand{\R}{\mathbb R}

\newtheorem{theorem}{Theorem}[section]
\newtheorem{corollary}[theorem]{Corollary}
\newtheorem{lemma}[theorem]{Lemma}
\newtheorem{proposition}[theorem]{Proposition}

\theoremstyle{definition}
\newtheorem{definition}[theorem]{Definition}
\newtheorem{remark}[theorem]{Remark}

\numberwithin{equation}{section}
\def\a {\alpha}
\def\b{\beta}

\def\n{\nabla}

\begin{document}
\let\cleardoublepage

\baselineskip=17pt

\titlerunning{Space-time Quasiconcave Solutions}

\title{On space-time quasiconcave solutions of the heat equation}

\author{Chuanqiang Chen
\and
Xi-Nan Ma
\and
Paolo Salani}

\date{}

\maketitle

\address{C. Chen:  Department of Applied Mathematics, Zhejiang University of Technology, Hangzhou, 310032, Zhejiang Province, CHINA; \email{cqchen@mail.ustc.edu.cn}
\and
X. Ma: School of Mathematical Sciences, University of Science and Technology of China, Hefei, 230026, Anhui Province, CHINA; \email{xinan@ustc.edu.cn}
\and
P. Salani: Dipartimento di Matematica e Informatica `U. Dini', Universit\`a di Firenze, Viale Morgagni 67/A, 50137 Firenze, Italy; \email{paolo.salani@unifi.it}}


\chapter*{ Abstract}
In this paper we first obtain a constant rank theorem for the second fundamental form of the space-time level sets of a space-time quasiconcave solution of the heat equation. Utilizing this constant rank theorem, we can obtain some strictly convexity results of the spatial and space-time level sets of the space-time quasiconcave solution of the heat equation in a convex ring. To explain our ideas and for completeness, we also review the constant rank theorem technique  for the space-time Hessian of space-time convex solution of heat equation  and for the second fundamental form of the convex level sets for harmonic function.

\keywords{Heat equation, quasiconcavity, space-time level set, constant rank theorem, space-time quasiconcave solution}

\newpage

\tableofcontents

\newpage

\chapter{ Introduction}

Throughout the paper, $\Omega=\Omega_0\setminus\overline\Omega_1$ is a $C^{2,\alpha}$ convex ring in $\mathbb{R}^n$ ($n\geq 2$), i.e. $\Omega _0$ and
$\Omega _1$ are bounded  convex open sets in $\mathbb{R}^n$ of class $C^{2,\alpha}$ with $\overline \Omega _1  \subset\Omega _0$, and we consider a classical solution $u$ of the following problem
\begin{equation}\label{1.5}
\left\{\begin{array}{lcl} \dfrac{{\partial u}} {{\partial t}} =
\Delta u   &\text{in}& {\Omega  \times (0,+\infty)}, \\
              u(x,0) = u_0(x)  &\text{in}& \Omega\,,   \\
              u(x,t) =  0   &\text{on}&  \partial \Omega_0 \times [0,+\infty),\\
              u(x,t) =  1   &\text{in}&  \overline\Omega_1 \times [0,+\infty),
\end{array} \right.
\end{equation}
where the initial data $u_0\geq0$ is regular enough and satisfies $u_0=0$ on $\partial\Omega_0$ and $u_0=1$ on $\partial\Omega_1$.
We study the spatial and the space-time {\em quasiconcavity} of $u$ (notice that we set $u\equiv1$ in $\overline\Omega_1$).

We recall that a function $v :
\mathbb{R}^m\longrightarrow \mathbb{R}\cup\{-\infty\}$ is called {\em quasiconcave} in $\mathbb{R}^m$ ($m\in \mathbb N$) if all its superlevel
sets $\{y\in \mathbb{R}^m:  v(y)\ge c\}$  are convex.  If $v$ is defined only in a
proper subset $A\subset \mathbb{R}^m$, we extend it as $-\infty$ outside
$A $ and we say that $v$ is quasiconcave in $A$ if such an extension
is quasiconcave in $\mathbb{R}^m$.
Then we say that $u \in C(\overline\Omega_0\times [0,+\infty))$ is {\em spatially
quasiconcave} if the function $x \longmapsto u(x, t)$ is quasiconcave in $\overline\Omega_0\subset\R^n$ for
every fixed $t \geq 0$, and we say that $u$ is {\em space-time quasiconcave} if it is quasiconcave in $\overline\Omega_0\times[0,\infty)\subset\R^{n+1}$, that is if
all its space-time superlevel sets
$$\Sigma^c_{x,t} =\{(x, t)\in\overline\Omega_0\times[0,\infty)\, :\, u(x, t)\ge c\}$$
are convex in $\R^{n+1}$.
Equivalently (and more explicitly) we can give the following definition.
\begin{definition}
A function $u \in C(\overline\Omega_0\times [0,+\infty))$ is {\em spatially quasiconcave} if
\begin{equation}\label{1.1}
u((1-\lambda)x_0 + \lambda x_1, t) \ge \min \{u(x_0, t), u(x_1, t)\},
\end{equation}
for every $x_0, x_1 \in \overline{\Omega}_0$,  $\lambda \in  (0, 1)$ and every fixed $t\geq 0$.

Analogously, $u$ is {\em space-time quasiconcave} if
\begin{equation}\label{1.2}
u((1-\lambda)x_0 + \lambda x_1, (1-\lambda)t_0 + \lambda t_1) \ge \min \{u(x_0, t_0), u(x_1, t_1)\},
\end{equation}
for every $x_0, x_1 \in \overline{\Omega}_0$, $t_0, t_1 \ge 0$, $\lambda \in (0, 1)$.
\end{definition}

Clearly, if a function is space-time quasiconcave, then it is spatially
quasiconcave at every fixed time: if we fix a time $t\ge 0$,  \eqref{1.2} coincides with \eqref{1.1} if $t_0 =
t_1 = t$.


The quasiconcavity of solutions to elliptic partial
differential equations in convex rings has been extensively studied, starting from
\cite{AH} which contains the well-known result that the level curves
of the Green function of a convex domain in the plane are
convex Jordan curves. In 1956, Shiffman \cite{Sh56} studied the
 minimal annulus in $\mathbb{R}^3$ whose boundary consists of two closed
 convex curves in parallel planes $P_1, P_2$: he proved that the intersection
 of this surface with any parallel plane $P$, between $P_1$ and $P_2$, is a
 convex Jordan curve. In 1957, Gabriel \cite{Ga57} proved that the level sets
 of the Green function of a 3-dimensional bounded convex domain are strictly
 convex. In 1977, Lewis \cite{Le77} extended Gabriel's result to $p$-harmonic
 functions in higher dimensions. Caffarelli-Spruck \cite{CS82} generalized the
 Lewis' result \cite{Le77} to a class of semilinear elliptic partial differential equations.
 Motivated by Caffarelli-Friedman \cite{CF85}, Korevaar \cite{Ko90}
 gave a new proof of the results of Gabriel and Lewis by applying a deformation
 process jointly with a constant rank theorem for the second fundamental form of the level sets of a quasiconcave $p$-harmonic function.
 A survey of this subject was given by Kawohl \cite{Ka85} in 1985. For more recent results and updated references,
 see for instance \cite{BLS, BGMX, GX13}.

For parabolic equations, a natural question is whether the solution of an initial-boundary value problem is able to retain the quasiconcavity of the initial datum. This is in general
not true, as showed in \cite{IS08}. On the other hand, Brascamp and Lieb \cite{BL76}
earlier proved that the log-concavity of the initial datum is preserved by the heat flow and,
as a consequence, they got the $\log$-concavity of the first Dirichlet eigenfunction and the Brunn-Minkowski
inequality for the first Dirichlet eigenvalue of Laplacian operator in convex domains. In
a series of papers \cite{Bo82, Bo96, Bo00}, Borell studied certain
space-time convexities of the solution of heat equation with  Schr\"odinger
potential, obtaining a new proof of the Brascamp-Lieb's
theorem and a Brownian motion proof of the classical Brunn-Minkowski inequality.
Precisely, in relation to the present paper, in \cite{Bo82} Borell considers a solution of
the heat equation,
\begin{equation}\label{1.3}
 \frac{{\partial u}} {{\partial t}} =
\Delta u  \quad  \text{in}\quad {\Omega  \times (0,+\infty)}\,,
\end{equation}
with  the following initial boundary value condition
\begin{equation}\label{1.4}
\left\{\begin{array}{lcl}
              u(x,0) = 0  &\text{in}& \Omega= \Omega _0 \backslash \overline {\Omega _1 },   \\
              u(x,t) =  0   &\text{on}&  \partial \Omega_0 \times [0,+\infty),\\
              u(x,t) =  1   &\text{in}&  \overline\Omega_1 \times [0,+\infty),
\end{array} \right.
\end{equation}
that is problem \eqref{1.5} with $u_0\equiv 0$, and he proved the following theorem.
\begin{theorem}[\cite{Bo82}]\label{th1.1}
Let $u$ be a solution to problem \eqref{1.3}-\eqref{1.4}. Then the
space-time superlevel sets $\Sigma^c_{x,t}$ of $u$ are convex for every $c\in[0,1]$.
\end{theorem}

In 2010 and 2011, Ishige-Salani \cite{IS10,IS11} gave a new proof of the above
theorem of Borell, and they extended it to more general fully
nonlinear parabolic equations, also introducing the notion of parabolic quasiconcavity.
But they still need the initial datum to be identically vanishing, indeed a quite restrictive assumption.
Some results similar to \cite{IS10} are contained in \cite{DK0} too, while an attempt to treat the case of a general (not zero) initial datum was done in
 \cite{DK}. Earlier related results can also be found in \cite{Ke}.

However, until now, it remained a longtime open problem what are suitable conditions on the initial datum $u_0$ that suffice to guarantee a spatially
or (better) a space-time quasiconcave solution $u$ of \eqref{1.5}.

In this paper, we give the following strictly convexity result for the space-time quasiconcave solution of heat equation
\begin{theorem}\label{th1.2}
Let $u$ be a space-time quasiconcave solution of problem \eqref{1.5} (where $\Omega$ is as said at the beginning), with
\begin{equation}\label{1.6}
u_t >0, \quad \text{in}\quad \Omega \times (0,+\infty)\,.
\end{equation}
Then \\
(1) $u$ is spatial strictly quasiconcave, i.e. the spatial superlevel sets  $\Sigma^{c,t}_{x}$ of $u$ are strictly convex for every $c \in (0,1)$ and $t \in (0,+\infty)$.\\
(2) there exists $T_0 \in [0, +\infty)$, such that $u$ is space-time strictly quasiconcave for $t > T_0$. Exactly,
\begin{align*}
&\mathbf{Rank}( II_{\partial\Sigma^{c}_{x,t}} (x,t) ) \equiv n-1, \quad \text{ for any } (x,t) \in \Omega \times (0, T_0];\\
&\mathbf{Rank}( II_{\partial\Sigma^{c}_{x,t}} (x,t) ) \equiv n, \quad \text{ for any } (x,t) \in \Omega \times (T_0, +\infty),
\end{align*}
where $II_{\partial\Sigma^{c}_{x,t}} (x,t)$ is the second fundamental form of the space-time level set $\partial\Sigma^{c}_{x,t}$ at $(x,t)$.
\end{theorem}

The proof of Theorem \ref{th1.2} is given in Section 4.

\begin{remark}\label{remark1}
(1) As showed in \cite{DK},  the initial condition \eqref{1.6} guarantees
\begin{equation}\label{1.7}
 |\nabla u|>0 \quad \text{in}\quad \Omega \times (0,+\infty)\,.
\end{equation}
This is essential for our proof of the main Theorem \ref{th1.2}, as well as for the proofs of Theorem \ref{th1.3} and Theorem \ref{th1.4}.

(2) Here $T_0 \in [0, + \infty)$. If $T_0 =0$, we will get the space-time strictly quasiconcavity for any $t>0$. But it is not easy.

(3) The condition \textbf{that $u$ is space-time quasiconcave} is not easy to verify, even we add some strong conditions on $u_0$ and $\Omega$.
Chau-Weinkove \cite{CW18} give some counterexamples to indicate that $u$ is not space-time quasiconcave even for smooth and subharmonic $u_0$. Here we indicate some necessary conditions on $u_0$.

First, if $u \in C^{4,3}(\overline{\Omega }\times [0, +\infty))$, then $u_0 \in C^4(\Omega)$, $\Delta^2 u_0 = \Delta u_0 =0$ on $\partial \Omega$, and we have the following compatible necessary conditions
\begin{align*}
&u_0(x)=0, \quad x \in \partial \Omega_0;\quad  u_0(x)=1, \quad x \in \partial \Omega_1; \\
&\Delta u_0 (x) \geq 0, \quad |\nabla u_0(x)| >0, \quad \text{ for any } x \in \Omega.
\end{align*}
For any $x \in \Omega$, if we choose the coordinate such that
\begin{align*}
u_{0,n} = |\nabla u_0(x)| >0, \quad \{u_{0,ij}\}_{1\leq i, j \leq n-1}\text{ is diagonal at } x.
\end{align*}
Then we need the following necessary condition
\begin{align}
u_{0,ii} \leq 0,\text{ for any } 1\leq i \leq n-1; \notag
\end{align}
and
\begin{align}
u_{0,n}^2 \Delta^2 u_0 +u_{0,nn} (\Delta u_0)^2 -2 u_{0,n} \Delta u_0 \Delta u_{0,n} - \sum\limits_{i} \frac{1}{u_{0,ii}}[u_{0,n} \Delta u_{0,i}- \Delta u_{0} u_{0,in}]^2 \leq 0, \notag
\end{align}
where the $\sum\limits_{i}$ is the summation for all $i$ such that $u_{0,ii} < 0$.
\end{remark}

We will prove Theorem \ref{th1.2} through the following {\em constant rank
theorem} for the second fundamental form of the space-time level surfaces of a space-time quasiconcave
solution of the heat equation.

\begin{theorem}\label{th1.3}
Suppose $u \in C^{4,3}(\Omega \times (0,T))$ is a space-time quasiconcave solution to the
heat equation (\ref{1.3}) satisfying \eqref{1.6}. Then the second
fundamental form $II_{\partial{\Sigma}^{c}_{x,t}}$ of the space-time level sets
$\partial{\Sigma}^{c}_{x,t}$  has the following constant rank property for $c\in (0,1)$:
if the rank of $II_{\partial{\Sigma}^{c}_{x,t}}$ attains its minimum rank $l_0$
$(0\leq l_0\leq n)$ at some point $(x_0,t_0) \in \Omega \times (0,T)$, then the
rank of $II_{\partial{\Sigma}^{c}_{x,t}}$ is constant  $l_0$ in $\Omega\times (0,t_0]$.
Moreover, let $l(t)$ be the minimal rank of $II_{\partial{\Sigma}^{c}_{x,t}}$ in
$\Omega$, then $l(s) \leqslant l(t)$ for all $s \leqslant t < T$.
\end{theorem}

The proof of Theorem \ref{th1.3} is given in Section 3.2 and Section 3.3. For reader's convenience, the Appendix contains the same proof in dimension $2$.

Constant rank theorems constitute an important tool to study convexity properties of solutions to elliptic and parabolic partial
differential equations. A technique based on the combination of a constant rank theorem and a homotopic deformation process was introduced in dimension 2 by Caffarelli-Friedman \cite{CF85} (a similar result was also discovered by Singer-Wong-Yau-Yau \cite{SYY85} at the same time). The result of \cite{CF85} has been later generalized to $\R^n$ by Korevaar-Lewis \cite{KL87}. Recently constant rank theorems have been obtained for the Hessian of solutions to fully nonlinear elliptic and parabolic equations in \cite{CGM07} and \cite{BG09, BG10, SW15}. Notice that, for parabolic equations, the constant rank theorems in \cite{ BG09, CGM07} regard the space variable only; Hu-Ma \cite{HM12}  obtained instead a constant rank theorem for the space-time Hessian of space-time convex solutions to the heat equation, while Chen-Hu \cite{CH12} were able to reduce the computations of \cite{HM12}, so to get a generalization to fully nonlinear parabolic equations.

About quasiconcave solutions in convex rings, we already mentioned Korevaar \cite{Ko90} who got a constant rank theorem for the second fundamental form of the level sets of quasiconcave $p$-harmonic functions; then Bian-Guan-Ma-Xu \cite{BGMX} and Guan-Xu \cite{GX13} obtained a generalization to fully nonlinear elliptic equations, while Chen-Shi \cite{CS} got a parabolic version of \cite{BGMX, GX13} for the second fundamental form of spatial level sets.

As applications of constant rank theorems, apart from the existence of convex and quasiconcave solutions to partial differential equations, we recall that the Christoffel-Minkowski problem and the related prescribing Weingarten curvature problem were studied in \cite{GLM06, GM03}, the uniqueness of K\"{a}hler-Einstein metric with the related curvature restriction in K\"{a}hler geometry  was studied by \cite{GLZ}. Moreover, the preservation of convexity for the general geometric flows of hypersurfaces has been investigated in \cite{BG09}.

We also recall that constant rank theorems can be often regarded as microscopic versions of some corresponding macroscopic convexity principle; this relationship exists in particular between the results of \cite{BG09} and \cite{All97}, as well as between the results of \cite{BGMX} and \cite{BLS} .

Similarly to the proof of Theorem \ref{th1.2}, we can also get the strict convexity of the space-time level sets of the solution to  (\ref{1.3})-(\ref{1.4}), as a corollary of Theorem \ref{th1.3} and Theorem \ref{th1.1}.
\begin{theorem}\label{th1.4}
Let $u$ be the solution to problem \eqref{1.3}-\eqref{1.4}. Then \\
(1) $u$ is spatial strictly quasiconcave, i.e. the spatial superlevel sets  $\Sigma^{c,t}_{x}$ of $u$ are strictly convex for every $c \in (0,1)$ and $t \in (0,+\infty)$.\\
(2) there exists $T_0 \in [0, +\infty)$, such that $u$ is space-time strictly quasiconcave for $t > T_0$. Exactly,
\begin{align*}
&\mathbf{Rank}( II_{\partial\Sigma^{c}_{x,t}} (x,t) ) \equiv n-1, \quad \text{ for any } (x,t) \in \Omega \times (0, T_0];\\
&\mathbf{Rank}( II_{\partial\Sigma^{c}_{x,t}} (x,t) ) \equiv n, \quad \text{ for any } (x,t) \in \Omega \times (T_0, +\infty).
\end{align*}
\end{theorem}

The rest of the paper is organized as follows.

In Chapter 2, we introduce some basic definitions; in particular Section 2.1 contains some preliminaries and basic curvature formulas for the level sets of a function $u$. To explain our ideas and for completeness, we review the constant rank theorem technique, including the constant rank theorem  on the space-time Hessian for the space-time convex solution of heat equation  in Section 2.2 (see \cite{HM12} and \cite{CH12}) and the strict convexity of the level sets for harmonic functions in convex rings in Section 2.3 via constant rank theorem technique and deformation process (see  \cite{Ko90} and \cite{BGMX}).

In Chapter 3, first we prove Theorem \ref{th3.1}, a constant rank theorem for the second fundamental form of the spatial level sets of a space-time quasiconcave solution to heat equation \eqref{1.3}, then we prove Theorem \ref{th1.3}. Its proof
is split into two cases  (according to Lemma \ref{lem2.8}): CASE 1 is treated in Section 3.2 using the constant rank theorem established in Section 3.1, while CASE 2 is treated in Section 3.3.

In Chapter 4, we study the solution of Borell \cite{Bo82} and prove Theorem \ref{th1.4} in Section 4.1, by utilizing the constant rank theorem of spatial level sets and space-time level sets. Similarly, we prove Theorem \ref{th1.2} in Section 4.2.

Finally, in the appendix we  rewrite the proof of Theorem \ref{th1.3} in the plane. In particular, we rewrite explicitly the computations of Section 3.3 in dimension $2$; we hope this can be helpful to clarify the hard (and long) computations of the general case.

{\bf Acknowledgment:}  Part of the work was done while the second author was visiting the IHES in January 2012 and February 2013, and while he was visiting CMS (Zhejiang University) in August 2012; then he would like to thank these institutes for their warm hospitality. The second author would also like to thank Prof. Fanghua Lin and Prof. Lihe Wang for asking him this question in summer 2009. All the authors thank Prof. Pengfei Guan for the encouragement on this subject and some helpful discussions.  We also thank Prof. Chau and Prof. Weinkove to a communication their
recent works \cite{CW18} on this subject. The first and the second authors were supported by the Wu Wen-Tsun Key Laboratory of Mathematics in USTC and NSFC. The third author has been partially supported by the PRIN 2012 project ``Equazioni alle derivate parziali di tipo ellittico e parabolico: aspetti geometrici, disuguaglianze collegate e applicazioni" of MIUR and by GNAMPA of INdAM.

\newpage

\chapter{Basic definitions and the Constant Rank Theorem technique}

In this chapter, in order to better explain our ideas and for completeness, we review the constant rank theorem technique; in particular, in Section 2.2 we describe the constant rank theorem for the space-time Hessian of space-time convex solutions to heat equation (see \cite{HM12} and \cite{CH12}), while we review the strict convexity of the level sets for harmonic functions in convex rings via the constant rank theorem technique and deformation process (see  \cite{Ko90} and \cite{BGMX}) in Section 2.3.  The technique of Section 2.3 will be generalized to get a constant rank theorem for the second fundamental form of the spatial level sets of a space-time quasiconcave solution to heat equation in Section 3.1. And the technique in Section 2.2 will be generalized to get a constant rank theorem for the  second fundamental form of the space-time level surfaces of a space-time quasiconcave solution of the heat equation in Section 3.2 and Section 3.3.

\section{Preliminaries}
\setcounter{equation}{0} \setcounter{theorem}{0}

Throughout the paper, $\nabla u=(u_1, u_2, \cdots, u_{n-1}, u_n)$
denotes the spatial gradient of $u$ and $D u=(\nabla u, u_t)=(u_1, u_2, \cdots, u_{n-1}, u_n, u_t)$ denotes its
space-time gradient.

In the following four subsections we collect some useful facts about the curvature of level sets
and elementary symmetric functions.

\subsection{The curvature matrix of the  level sets of $u(x)$ }

 In this subsection, we recollect some curvature formulas for the level sets of a $C^2$ function $u(x)$ from the presentation in \cite{BGMX}.
 We first recall some fundamental notations in classical surface theory.

 Assume a surface $\Sigma\subset\R^n$ is
given by the graph of a function $v$ in a domain in $\R^{n-1}$:
$$\Sigma=\{(x',x_n)\,:\,x_n = v(x')\},\,\, x'=( x_1,x_2, \cdots, x_{n-1})\in \R^{n-1}.$$
Then the first  fundamental form of $\Sigma$ is given
by $g_{ij}=\delta_{ij} + v_iv_j.$
The upward normal direction $\vec {\mathfrak{n}}$ and the second fundamental
form  of the graph $x_n = v(x')$ are respectively given by
\begin{align*}
\vec {\mathfrak{n}}=~\frac{1}{W}(-v_1,-v_2, \cdots, -v_{n-1},1),\quad
 b_{ij}=~\frac{v_{ij}}{W},
\end{align*}
where $1\le i,j \le n-1$ and $W = (1+|\n v|^2)^{\frac12}$.
\begin{definition}\label{def1}
We say that the graph of function $v$ is convex with respect to the upward normal $\vec {\mathfrak{n}}$ if the second fundamental form $b_{ij}= \dfrac{v_{ij}}{W}$ of the graph of $v$ is nonnegative definite.
\end{definition}

The principal curvatures $\kappa_1, \cdots, \kappa_{n-1}$ of
the graph of $v$, being the eigenvalues of the second fundamental
form relative to the first fundamental form, satisfy
$$\det(b_{ij} - \kappa_l g_{ij}) = 0\quad\text{for }l=1,\dots,n-1\,.$$
Equivalently, $\kappa_l$ satisfies
$$\det(a_{ij} - \kappa_l \delta_{ij}) = 0\,,
$$
where
$$\quad (a_{ij})= (g^{il})^{\frac12}(b_{lk})(g^{kj})^{\frac12}$$
and $(g^{ij})$ is the inverse matrix of $(g_{ij})$.
Then we have the following well known fact  \cite{CNS85}:
the principal curvature of the graph $x_n = v(x')$ with respect to
the upward normal $\vec {\mathfrak{n}}$ are the eigenvalues of the
symmetric curvature matrix
\begin{equation*}
a_{il} =\frac{1}{W}\bigg\{v_{il} -\frac{v_iv_jv_{jl}}{W(1+W)}
-\frac{v_lv_kv_{ki}}{W(1+W)} + \frac{v_iv_lv_jv_k
v_{jk}}{W^2(1+W)^2}\bigg\},
\end{equation*}
where the summation convention over repeated indices is employed .

Let $\Omega$ be a domain in $\R^n$  and $u\in C^2(\Omega)$.
We denote by $\partial\Sigma^{u(x_o)}$  the level set of
$u$ passing through the point $x_o \in \Omega$, i.e.
$\partial\Sigma^{u(x_o)}=\{x\in\Omega|u(x)=u(x_o)\}$.
Now we shall work near a  point $x_o$ where $|\n
u(x_o)|\neq 0$.  Without loss
of generality we assume $x_o=0$ and $u_n(x_o)\neq 0$ and consider a small
neighborhood of $x_o$. By the implicit function theorem, locally the level set
$\partial\Sigma^{u(x_o)}$ can be represented as a local graph
$$x_n = v(x')\,,\quad x'=( x_1,x_2, \cdots, x_{n-1})\in B(0,\epsilon)\subseteq\R^{n-1},$$
for $\epsilon>0$ sufficiently small and $v(x')$ satisfies the following equation
$$u(x_1,x_2, \cdots, x_{n-1}, v(x_1,x_2, \cdots, x_{n-1})) = u(x_o).$$
The latter yields
\begin{align*}
u_i + u_n v_i = 0\,,
\end{align*}
whence
$$
 v_i = - \frac{u_i}{u_n}\,.
$$
Then the first fundamental form of the level set is
\begin{equation*}
g_{ij}=\delta_{ij} + \frac{u_iu_j}{u_n^2}\,.
\end{equation*}
It follows that the upward
normal direction of the level set is
\begin{equation}\label{1.1aaa}
\vec{\mathfrak{n}}= \dfrac{|u_n|}{|\nabla u|u_n}
(u_1, u_2, \cdots, u_{n-1}, u_n).
\end{equation}
We  also have
\begin{equation*}
u_{ij} +u_{in} v_{j}+ u_{nj}v_{i}+ u_{nn} v_{i}v_j + u_n v_{ij}=0.
\end{equation*}
If we set
\begin{equation}\label{1.1a}
h_{ij} =u_n^2 u_{ij}+u_{nn}u_{i}u_j-u_nu_ju_{in}-u_nu_iu_{jn}\,,
\end{equation}
then
it follows
\begin{equation*}
v_{ij} = - \frac{h_{ij}}{u_n^3}.
\end{equation*}
The second fundamental form  of the
level set of the function $u$ with respect to the upward normal
direction is given by
\begin{equation}\label{1.1aa}
b_{ij}=\frac{v_{ij}}{W}=-\frac{|u_n| h_{ij}}{|\n u|u_n^3}\,,
\end{equation}
where  $ W = (1+|\n v|^2)^{\frac12} = \frac{|\n u|}{|u_n|}$.

\begin{definition}
In the same assumption and notation as above, we say that the level set
$\partial\Sigma^{u(x_o)} = \{x \in \Omega| u(x)=u(x_o)\}$ is locally convex respect
to the upward normal direction $\vec{\mathfrak{n}}$
if the second fundamental form
$b_{ij}= - \dfrac{|u_n| h_{ij}}{|\n u|u_n^3}$ is nonnegative
definite at $x_o$.
\end{definition}

Now we can express the curvature matrix $(a_{ij})$
of the level sets of the function $u$ in terms of the derivatives of $u$.  We can assume $\n u$ is the
upward normal of the level set $\partial\Sigma^{u(x_o)}$ at $x_o$, then
$u_n(x_o)> 0$.

From \cite{BGMX}, it follows that the
symmetric curvature matrix $(a_{ij})$  is given by
\begin{equation}\label{2.5a}
a_{ij} =-\frac{|u_n|}{|\nabla u|{u_n}^3}A_{ij}, \quad  1 \leq i,j \leq n-1,
\end{equation}
where
\begin{equation}\label{2.6a}
A_{ij} = h_{ij}
-\frac{u_iu_lh_{jl}}{W(1+W)u_n^2} -\frac{u_ju_lh_{il}}{W(1+W)u_n^2}
+ \frac{u_iu_ju_ku_l h_{kl}}{W^2(1+W)^2u_n^4}, \quad W=\frac{|\nabla u|}{|u_n|}.
\end{equation}
With the above notations, at a
point $(x_0,t_0)$ where $u_n(x_0,t_0)=|\nabla u(x_0,t_0)|>0$ and $u_i(x_0,t_0)=0$ for
$i=1, \cdots, n-1$, $a_{ij,k}$ is commutative, i.e. it satisfies
the Codazzi property
$$
a_{ij,k}=a_{ik,j} \quad \forall i,j,k \le n-1\,,
$$
where we use the following notation $a_{lm,r}=\frac{\partial}{\partial x_r}a_{lm}$.

\subsection{The curvature matrix of the spatial level sets of $u(x,t)$}
 Throughout this subsection, $\Omega$ is a domain in $\mathbb{R}^n$ and $u\in C^{2,1}(\Omega \times [0,
T))$ satisfies $\nabla u\neq 0$ in $\Omega\times[0,T)$.

We introduce the following notation: for $t\in[0,T)$ and $c\in\mathbb{R}$, $\partial\Sigma_x^{t,c}$ denotes the spatial $c$-level set of the function $u$, at the fixed time $t$, that is
$$
\partial\Sigma_x^{c,t}=\{x\in\Omega\,:\,u(x,t)=c\}\,.
$$
Notice that, thanks to the assumptions on $u$, $\partial\Sigma_x^{c,t}$ is a regular hypersurface in $\mathbb{R}^n$.
Now we fix $(x_0,t_0)\in\Omega\times(0,T)$ and without loss of generality we assume
 $u_n(x_0,t_0) \ne 0$. As in \cite {BGMX, CNS85}, it follows that the upward normal direction of the hypersurface $\partial \Sigma_x^{c,t}$ at $x_0$ is
\begin{eqnarray}\label{2.1}
\vec{\mathfrak{n}}= \frac{|u_n|}{|\nabla u|u_n}\nabla u\,
\end{eqnarray}
and the second fundamental form $II$ of $\partial\Sigma_x^{c,t}$ with respect to $\vec{\mathfrak{n}}$
is given by
\begin{equation}\label{2.4}
b_{ij} = - \frac{|u_n|h_{ij}}{|\nabla u|u_n^3},
\end{equation}
where
\begin{align*}
h_{ij} =u_n^2 u_{ij}+u_{nn}u_{i}u_j-u_nu_ju_{in}-u_nu_iu_{jn}.
\end{align*}
Notice that if $\partial\Sigma_x ^{c,t}$ is
locally convex with respect to the upward normal direction, then $b_{ij}$ is positive semidefinite (and vice versa). Moreover, let $a(x,t)=(a_{ij}(x,t))$ be similarly defined  by
\begin{equation}\label{2.5}
a_{ij} =-\frac{|u_n|}{|\nabla u|{u_n}^3}A_{ij}, \quad  1 \leq i,j \leq n-1,
\end{equation}
where
\begin{equation}\label{2.6}
A_{ij} = h_{ij}
-\frac{u_iu_lh_{jl}}{W(1+W)u_n^2} -\frac{u_ju_lh_{il}}{W(1+W)u_n^2}
+ \frac{u_iu_ju_ku_l h_{kl}}{W^2(1+W)^2u_n^4}, \quad W=\frac{|\nabla u|}{|u_n|}\,;
\end{equation}
then $a_{ij}$ is the symmetric curvature tensor of $\partial\Sigma_x^{c,t}$.

\subsection{The curvature matrix of the space-time level sets of $u(x,t)$}

In this subsection, we assume $u \in C^{3,1}(\Omega \times [0, T))$ and $u_t(x,t) \ne 0$  (whence $|Du(x,t)|\neq 0$)
for every $(x,t)\in \Omega \times [0, T)$.

Similarly to the previous section, we introduce the following notation (in fact already given in the introduction) for the space-time level sets of the function $u$:
$$\partial\Sigma_{x,t}^{c}=\{(x,t) \in
\Omega \times [0, T)\,:\,u(x,t) = c\}\,.$$

Following \cite{BGMX}, we have that the
upward normal direction of $\partial\Sigma_{x,t}^{c}$ is given by
\begin{eqnarray}\label{2.7}
\vec{\hat{\mathfrak{n}}}= \frac{|u_t|}{|D u|u_t}Du\,,
\end{eqnarray}
and the second fundamental form $II$ of $\partial\Sigma_{x,t}^{c}$ with respect to $\vec{\hat{\mathfrak{n}}}$
is
\begin{equation}\label{2.8}
\hat b_{\a \b} =  - \frac{|u_t|(u_t^2 u_{\a \b} + u_{tt} u_{\a} u_{\b} - u_t u_{\b} u_{\a t}
- u_t u_{\a} u_{\b t})}{|D u| u_t^3}.
\end{equation}
Then we set
\begin{equation}\label{2.9}
\hat h_{\a \b} = u_t^2 u_{\a \b} + u_{tt} u_{\a} u_{\b} - u_t u_{\b} u_{\a t}
- u_t u_{\a} u_{\b t}, \quad  1 \leq \a, \b \leq n\,,
\end{equation}
so that we can write
\begin{equation}\label{2.10}
\hat b_{\a \b} = - \frac{|u_t|\hat h_{\a \b}}{|D u|u_t^3}.
\end{equation}
Note that if $\partial\Sigma_{x,t}^{c} = \{(x,t) \in
\Omega \times [0, T]|u(x,t) = c\}$ is
locally convex with respect to the upward normal direction, then $\hat b_{\a\b}$
is positive semidefinite (and vice versa).
Moreover, if  $\hat a(x,t)=(\hat a_{ij}(x,t))$ denotes the symmetric Weingarten
tensor of $\partial\Sigma_{x,t}^{c}$, then $\hat a$ is
positive semidefinite and it holds
\begin{equation}\label{2.11}
\hat a_{\a \b} =-\frac{|u_t|}{|D u|{u_t}^3} \hat A_{\a \b}, \quad  1 \leq \a, \b \leq n,
\end{equation}
where
\begin{equation}\label{2.12}
\hat A_{\a \b} = \hat h_{\a \b}
-\frac{u_\a u_\gamma \hat h_{\b \gamma}}{\hat W(1+\hat W)u_t^2} -\frac{u_\b u_\gamma \hat h_{\a \gamma}}{\hat W(1+ \hat W)u_t^2}
+ \frac{u_\a u_\b u_\gamma u_\eta \hat h_{\gamma \eta}}{\hat W^2(1+\hat W)^2 u_t^4}, \quad \hat W = \frac{|D u|}{|u_t|}.
\end{equation}
With the above notations,
we get
\begin{align}\label{2.13}
1-\frac{u_n^2}{\hat W(1+\hat W) u_t^2} = \frac{\hat W u_t^2+\hat W^2 u_t^2-u_n^2}{\hat W(1+\hat W)u_t^2}
= \frac{1}{\hat W}+\frac{\sum_{l=1}^{n-1} u_l^2}{\hat W(1+\hat W)u_t^2}\,,
\end{align}
and, for $1 \leq i,j \leq n-1$, we have
\begin{align}\label{2.14}
\hat A_{ij} =& \hat h_{ij}-\frac{u_i u_n \hat h_{jn}}{\hat W(1+\hat W)u_t^2} -\frac{u_ju_n \hat h_{in}}{\hat W(1+ \hat W)u_t^2} \notag \\
&-\frac{u_i \sum_{l=1}^{n-1}u_l \hat h_{jl}}{\hat W(1+\hat W)u_t^2} -\frac{u_j \sum_{l=1}^{n-1} u_l \hat h_{il}}{\hat W(1+ \hat W)u_t^2}
+ \frac{u_iu_ju_n^2 \hat h_{nn}}{\hat W^2(1+\hat W)^2u_t^4} +T_{ij},
\end{align}
\begin{align}\label{2.15}
\hat A_{in} =& \hat h_{in} -\frac{u_iu_n \hat h_{nn}}{\hat W(1+\hat W)u_t^2} -\frac{u_n^2\hat h_{in}}{\hat W(1+\hat W)u_t^2}
-\frac{u_n \sum_{l=1}^{n-1}u_l
\hat h_{il}}{\hat W(1+\hat W)u_t^2}+ \frac{u_iu_n^3 \hat h_{nn}}{\hat W^2(1+\hat W)^2u_t^4}  \notag \\
&-\frac{u_i \sum_{l=1}^{n-1}u_l \hat h_{nl}}{\hat W(1+\hat W)u_t^2}
+ 2\frac{u_i  u_n^2 \sum_{l=1}^{n-1}u_l\hat h_{nl}}{\hat W^2(1+\hat W)^2 u_t^4} +T_{in}  \notag \\
=& \hat h_{in}[ \frac{1}{\hat W}+\frac{\sum_{l=1}^{n-1} u_l^2}{\hat W(1+\hat W)u_t^2}]
-\frac{u_iu_n \hat h_{nn}}{\hat W(1+\hat W)u_t^2} [ \frac{1}{\hat W}
+\frac{\sum_{l=1}^{n-1} u_l^2}{\hat W(1+\hat W)u_t^2}]-\frac{u_n \sum_{l=1}^{n-1}u_l
\hat h_{il}}{\hat W(1+\hat W)u_t^2}  \notag \\
&-\frac{u_i \sum_{l=1}^{n-1}u_l \hat h_{nl}}{\hat W(1+\hat W)u_t^2}
+ 2\frac{u_i \sum_{l=1}^{n-1}u_l\hat h_{nl}}{\hat W(1+\hat W)u_t^2}[ 1-\frac{1}{\hat W}
-\frac{\sum_{l=1}^{n-1} u_l^2}{\hat W(1+\hat W)u_t^2}] +T_{in}\notag \\
=& \frac{1}{\hat W}\hat h_{in}-\frac{u_iu_n \hat h_{nn}}{\hat W^2(1+\hat W)u_t^2} -\frac{u_n \sum_{l=1}^{n-1}u_l
\hat h_{il}}{\hat W(1+\hat W)u_t^2} \notag \\
&+\frac{\hat h_{in} \sum_{l=1}^{n-1} u_l^2}{\hat W(1+\hat W)u_t^2}
+ \frac{u_i \sum_{l=1}^{n-1}u_l\hat h_{nl}}{\hat W(1+\hat W)u_t^2}[1- \frac{2}{\hat W}] +T_{in},
\end{align}
and
\begin{align}\label{2.16}
\hat A_{nn} =& \hat h_{nn}-2\frac{u_n u_l \hat h_{nl}}{\hat W(1+\hat W)u_t^2}
+ \frac{u_n^2 u_k u_l \hat h_{kl}}{\hat W^2(1+\hat W)^2 u_t^4}\notag \\
=&  \hat h_{nn}-2\frac{u_n^2 \hat h_{nn}}{\hat W(1+\hat W)u_t^2}
+ \frac{u_n^4 \hat h_{nn}}{\hat W^2(1+\hat W)^2 u_t^4}\notag \\
&-2\frac{u_n\sum_{l=1}^{n-1}u_l\hat h_{nl}}{\hat W(1+\hat W)u_t^2}
+ 2\frac{u_n^3 \sum_{l=1}^{n-1}u_l\hat h_{nl}}{\hat W^2(1+\hat W)^2u_t^4} +
\frac{u_n^2 \sum_{k,l=1}^{n-1} u_k u_l \hat h_{kl}}{\hat W^2(1+\hat W)^2u_t^4}\notag \\
=&  \hat h_{nn}[ \frac{1}{\hat W}+\frac{\sum_{l=1}^{n-1} u_l^2}{\hat W(1+\hat W)u_t^2}]^2 \notag \\
&-2\frac{u_n\sum_{l=1}^{n-1}u_l\hat h_{nl}}{\hat W(1+\hat W)u_t^2}[ \frac{1}{\hat W}
+\frac{\sum_{l=1}^{n-1} u_l^2}{\hat W(1+\hat W)u_t^2}]
+ \frac{ \sum_{k,l=1}^{n-1} u_ku_l \hat h_{kl}}{\hat W(1+\hat W)u_t^2}[1- \frac{1}{\hat W}
-\frac{\sum_{l=1}^{n-1} u_l^2}{\hat W(1+\hat W)u_t^2}]\notag \\
=& \frac{1}{\hat W^2} \hat h_{nn}-2\frac{u_n\sum_{l=1}^{n-1}u_l\hat h_{nl}}{\hat W^2(1+\hat W)u_t^2}\notag \\
&+2\frac{\sum_{l=1}^{n-1} u_l^2\hat h_{nn}}{\hat W^2(1+\hat W)u_t^2}
+ \frac{ \sum_{k,l=1}^{n-1} u_ku_l \hat h_{kl}}{\hat W(1+\hat W)u_t^2}[1- \frac{1}{\hat W}] +T_{nn},
\end{align}
where $T_{\a \b}$ ($1 \leq \a, \b \leq n$) includes all the terms containing at least three $u_i$'s ($1 \leq i \leq n-1$).

Notice that, when we choose a coordinate system such that
$u_n(x_0,t_0)=|\nabla u(x_0,t_0)|>0$ while  $u_i(x_0,t_0)=0$ for $i=1, \cdots, n-1$,
it holds
\begin{align}\label{2.17}
T_{\a \b} =0, D T_{\a \b} =0, D^2 T_{\a \b} =0, \quad 1 \leq \a, \b \leq n.
\end{align}

\subsection{Elementary symmetric functions}

In this subsection, we recall the definition and some basic properties
of elementary symmetric functions. For more details we refer to
\cite{cns, GM03,L96, Reilly}.

\begin{definition}
Let $\lambda=(\lambda_1,\cdots,\lambda_n)\in {\Bbb R}^n$.
For any $k\in\{1, 2,\cdots, n\}$ we denote by $\sigma_k(\lambda)$ the $k$-th elementary symmetric function of $\lambda_1,\dots,\lambda_n$, that is
$$\sigma_k(\lambda) = \sum _{1 \le i_1 < i_2 <\cdots<i_k\leq n}\lambda_{i_1}\lambda_{i_2}\cdots\lambda_{i_k}\,.$$
 We also set $\sigma_0=1$ and $\sigma_k =0$ for $k>n$.
\end{definition}

We denote by $\sigma _k (\lambda \left| i \right.)$ the $k$-th symmetric
function of the vector $\lambda\left| i\right.$, obtained from $\lambda$ by removing the $i$-th component (or equivalently by imposing  $\lambda_i = 0$), and by $\sigma _k (\lambda \left | ij\right.)$ the symmetric function of the vector $\lambda \left | ij\right.$, obtained form $\lambda$ by removing the $i$-th and the $j$-th components (or equivalently by imposing $\lambda_i =\lambda_j = 0$).

We need the following standard formulas for elementary symmetric
functions.
\begin{proposition}\label{prop2.3}
Let $\lambda=(\lambda_1,\dots,\lambda_n)\in\mathbb{R}^n$ and $k
\in\{0, 1, \cdots,n\}$, then
\begin{align*}
&\sigma_k(\lambda)=\sigma_k(\lambda|i)+\lambda_i\sigma_{k-1}(\lambda|i), \quad \forall \,1\leq i\leq n,\\
&\sum_i \lambda_i\sigma_{k-1}(\lambda|i)=k\sigma_{k}(\lambda),\\
&\sum_i\sigma_{k}(\lambda|i)=(n-k)\sigma_{k}(\lambda).
\end{align*}
\end{proposition}

The definition of $\sigma_k$ can be extended to symmetric matrices by letting
$\sigma_k(W) = \sigma_k(\lambda(W))$, where
$$ \lambda(W)= (\lambda
_1(W),\lambda _2 (W), \cdots ,\lambda _{n}(W))$$
is the vector consisting of the eigenvalues
of the $n\times n$ symmetric matrix $W$.

\begin{remark}It is easily seen that $W\geq 0$ if and only if $\sigma_k (W)\geq 0$ for $k=1,\dots,n$ and that, in case $W\geq 0$, then $\mathbf{Rank}(W)=r\in\{0,\dots,n\}$ if and only if
$\sigma_k(W)>0$ for $k=0,\dots,r$ and $\sigma_k(W)=0$ for $k>r$.
\end{remark}

For further use, we denote by $W \left|
i \right.$ the symmetric matrix obtained from $W$ by deleting the $i$-row and
$i$-column and by $W \left| ij\right.$ the symmetric
matrix obtained from $W$ when deleting the $i,j$-rows and $i,j$-columns, and similarly we define $W\left|ijk\right.$. Then
we have the following identities.

\begin{proposition}\label{prop2.4}
If $W=(W_{ij})$ is a diagonal $n\times n$ matrix and $m\in\{1,\dots,n\}$,
then
\begin{align*}
\frac{{\partial \sigma _m (W)}} {{\partial W_{ij} }} = \begin{cases}
\sigma _{m - 1} (W\left| i \right.), &\text{if } i = j, \\
0, &\text{if } i \ne j.
\end{cases}
\end{align*}
and
\begin{align}
\frac{{\partial ^2 \sigma _m (W)}} {{\partial W_{ij} \partial W_{kl}
}} =\begin{cases}
\sigma _{m - 2} (W\left| {ik} \right.), &\text{if } i = j,k = l,i \ne k,\\
- \sigma _{m - 2} (W\left| {ik} \right.), &\text{if } i = l,j = k,i \ne j,\\
0, &\text{otherwise }.
\end{cases}\notag
\end{align}
Here $\sigma _m$ is a function of $A=(A_{ij})$ in the space of symmetric matrices, and we write $\frac{{\partial \sigma _m (W)}} {{\partial W_{ij} }} = \frac{{\partial \sigma _m (A)}} {{\partial A_{ij} }}|_{A=W}$ and $\frac{{\partial ^2 \sigma _m (W)}} {{\partial W_{ij} \partial W_{kl}
}} =\frac{{\partial ^2 \sigma _m (A)}} {{\partial A_{ij} \partial A_{kl}
}}|_{A=W}$ for convenience.
\end{proposition}

Given the $n\times n$ matrix $\hat{a}$, we introduce the following notation:
$$
\hat{a} = \left( {\begin{matrix}
   {M} & { \hat a_{in} }  \\
   {\hat a_{ni}} & {\hat a_{nn} }  \\
 \end{matrix} } \right),
$$
where $M=(\hat a_{ij})_{(n-1)\times(n-1)}$.

\begin{lemma}\label{lem2.5} For $n\geq 3$ and $l\in\{3,\dots,n\}$, we have
\begin{align}\label{2.18}
\sigma_{l+1}(\hat{a} )=&\sigma_{l+1}(M)+\hat a_{nn}\sigma_l(M)
-\sum_{i}\hat a_{ni}\hat a_{in}\sigma_{l-1}(M|i)\notag\\
&+ \sum_{i\ne j}\hat a_{ni}\hat a_{jn}\hat a_{ij}\sigma_{l-2}(M|ij) -\sum_{i\ne
j,i\ne k,j\ne k}\hat a_{ni}\hat a_{jn}\hat a_{ik}\hat a_{kj}\sigma_{l-3}(M|ijk)+T,
\end{align}
where T includes only terms containing at least three of the $\hat a_{ij}$'s with $i\neq j$.
So when $M$ is diagonal, we have
$$
T = 0,\quad D T = 0,\quad D ^2 T = 0.
$$
\end{lemma}

To study the rank of the space-time second fundamental form $\hat a$, we
need the following simple technical lemma.
\begin{lemma} \label{lem2.8}
Suppose $\hat{a}\geq 0$,  $l=\mathbf{Rank} \{\hat a(x_0,t_0)\}$ and
$M=(\hat a_{ij}(x_0,t_0))_{(n-1)\times(n-1)}$ is diagonal with $ \hat a_{11} \geq \hat a_{22} \geq \cdots
\geq \hat a_{n-1n-1} $. Then  there is a positive constant $
C_0$ such that at $(x_0,t_0)$, we have

either CASE 1:
\begin{eqnarray*}
&&\hat a_{11}  \geq \cdots \geq \hat a_{l-1l-1} \geq
C_0 , \quad \hat a_{ll} = \cdots = \hat a_{n-1n-1} =0 , \\
&&\hat a_{nn} -\sum\limits_{i = 1}^{l-1} {\frac{{\hat a_{in} ^2 }} {{\hat a_{ii}
}}} \geq C_0 ,  \quad \hat a_{in} = 0, \quad l \leqslant i \leqslant n-1\,,
\end{eqnarray*}

or CASE 2:
\begin{eqnarray*}
&&\hat a_{11}  \geq \cdots \geq \hat a_{ll} \geq C_0 , \quad \hat a_{l+1l+1} =
\cdots = \hat a_{n-1n-1} =0, \\
&& \hat a_{nn}  = \sum\limits_{i = 1}^{l}{\frac{{\hat a_{in} ^2 }} {{\hat a_{ii} }}}
,  \quad \hat a_{in} = 0, \quad  l+1 \leqslant i \leqslant n-1\,.
\end{eqnarray*}
\end{lemma}

\begin{proof}
Let $\mathbf{Rank} \{ M \}= k$ at $(x_0, t_0)$.
Then either $k=l-1$ or $k =l$. Otherwise, if $k <l-1$, since $ \hat a_{11} \geq \hat a_{22} \geq \cdots
\geq \hat a_{n-1n-1}\geq 0 $, we would have
\begin{align*}
\hat a_{l-1 l-1}= \cdots=\hat a_{n-1n-1} =0  \text{ at }(x_0, t_0),
\end{align*}
and from $\hat a (x_0,t_0) \geq 0$, we would get
\begin{align*}
\hat a_{l-1 n}= \cdots=\hat a_{n-1n} =0  \text{ at }(x_0, t_0).
\end{align*}
So $\mathbf{Rank} \{ \hat a \} \leq l-1$, i.e. a contradiction.

For $k =l-1$, we have at $(x_0, t_0)$
\begin{align*}
\hat a_{11}  \geq \cdots \geq \hat a_{l-1l-1} > 0  , \quad \hat a_{ll} = \cdots = \hat a_{n-1n-1} =0 ,
\end{align*}
and, due to $\hat a (x_0,t_0) \geq 0$, we get
\begin{align*}
\hat a_{ln}= \cdots = \hat a_{n-1 n} = 0.
\end{align*}
Since $\mathbf{Rank} \{ \hat a \}=l$, then $\sigma_l (\hat a) >0$. Direct computation yields
\begin{align*}
\sigma_l (\hat a) = \hat a_{nn}\sigma_{l-1}(M)
-\sum_{i=1}^{l-1}\hat a_{ni}\hat a_{in}\sigma_{l-2}(M|i) = \sigma_{l-1}(M) \left[\hat a_{nn} - \sum_{i=1}^{l-1}\frac{\hat a_{in}^2}{\hat a_{ii}}\right ] >0,
\end{align*}
so we have
\begin{align*}
\hat a_{nn} - \sum_{i=1}^{l-1}\frac{\hat a_{in}^2}{\hat a_{ii}} >0,
\end{align*}
This is CASE 1 with
$$C_0=\min\left\{ \hat a_{11},\cdots,\hat a_{n-1n-1},\hat{a}_{nn}-\sum_{i=1}^{l-1}\frac{\hat a_{in}^2}{\hat a_{ii}}\right\}\,.
$$

For $k =l$, we have at $(x_0, t_0)$
\begin{align*}
\hat a_{11}  \geq \cdots \geq \hat a_{ll} > 0  , \quad \hat a_{l+1 l+1} = \cdots = \hat a_{n-1n-1} =0 ,
\end{align*}
and due to $\hat a (x_0,t_0) \geq 0$, we get
\begin{align*}
\hat a_{l+1 n}= \cdots = \hat a_{n-1 n} = 0.
\end{align*}
Since $\mathbf{Rank }\{ \hat a \}=l$, then $\sigma_{l+1} (\hat a) = 0$. Direct computation yields
\begin{align*}
\sigma_{l+1} (\hat a)= \hat a_{nn}\sigma_{l}(M)
-\sum_{i=1}^{l}\hat a_{ni}\hat a_{in}\sigma_{l-1}(M|i) = \sigma_{l}(M) [\hat a_{nn} - \sum_{i=1}^{l}\frac{\hat a_{in}^2}{\hat a_{ii}} ] =0,
\end{align*}
so we have
\begin{align*}
\hat a_{nn} - \sum_{i=1}^{l}\frac{\hat a_{in}^2}{\hat a_{ii}} =0,
\end{align*}
This is CASE 2.
\end{proof}

Similarly to Lemma 2.5 in \cite{BG09}, we have the following.

\begin{lemma}\label{lem2.9}
Assume $W(x)=(W_{ij}(x))\geq 0$ for every $x \in \Omega \subset
\mathbb{R}^{n}$, and $W_{ij} (x) \in C^{1,1} (\Omega )$. Then for
every $\mathcal {O} \subset \subset \Omega$, there exists a positive
constant $C$, depending only on the Hausdorff distance $dist\{\mathcal {O},
\partial \Omega \}$ of $\mathcal{O}$ from $\partial\Omega$ and $\left\| W \right\|_{C^{1,1} (\Omega )}$,
such that
\begin{equation}\label{2.19}
\left| {\nabla W_{ij} } \right| \leqslant C(W_{ii} W_{jj} )^{\frac{1}{4}},
\end{equation}
for every $x \in \Omega$ and $1 \leq i, j \leq n$.
\end{lemma}

\begin{proof} The same arguments as in the proof of \cite[Lemma 2.5]{BG09} carry through with small modifications since $W$ is a
general matrix instead of the Hessian matrix of a convex function.

It is known that for any nonnegative $C^{1,1}$ function $h$, $|\nabla
h(x)| \leq Ch ^{\frac{1}{2}} (x)$ for all $x \in \mathcal {O}$,
where $C$ depends only on $||h||_{C^{1,1}(\Omega)} $ and
$dist\{\mathcal {O}, \partial \Omega\}$ (see \cite{Tr71}).
Since $W (x) \geq 0$, we can choose $h(x) =W_{ii}(x) \geq 0$. Then we get
\begin{equation*}
\left| {\nabla W_{ii} } \right| \leqslant C_1 (W_{ii})^{\frac{1}{2}} = C_1 (W_{ii} W_{ii} )^{\frac{1}{4}}
\end{equation*}
and \eqref{2.19} holds for $i=j$.

Similarly, for $i\ne j$, we choose $ h=\sqrt{W_{ii}W_{jj}} \geq 0 $, then we get
 \begin{equation}\label{2.20}
\left| {\nabla \sqrt{W_{ii}W_{jj}} } \right| \leqslant C_2 (\sqrt{W_{ii}W_{jj}} )^{\frac{1}{2}} = C_2 (W_{ii} W_{jj} )^{\frac{1}{4}}.
\end{equation}
And for $h=\sqrt{W_{ii}W_{jj}} - W_{ij}$,  we have
 \begin{equation}\label{2.21}
\left| {\nabla (\sqrt{W_{ii}W_{jj}} - W_{ij} )} \right| \leqslant C_3 (\sqrt{W_{ii}W_{jj}} - W_{ij} )^{\frac{1}{2}}
\leq C_3 (W_{ii} W_{jj} )^{\frac{1}{4}}.
\end{equation}
So from \eqref{2.20} and \eqref{2.21}, we get
\begin{align*}
\left| {\nabla W_{ij}} \right|  =& \left| {\nabla \sqrt{W_{ii}W_{jj}} } - {\nabla (\sqrt{W_{ii}W_{jj}} - W_{ij} )} \right|  \\
\leq& \left| {\nabla \sqrt{W_{ii}W_{jj}} } \right| + \left| {\nabla (\sqrt{W_{ii}W_{jj}} - W_{ij} )} \right|  \\
\leq& (C_2 +C_3) (W_{ii} W_{jj} )^{\frac{1}{4}}.
\end{align*}
So \eqref{2.19} holds for $i \ne j$.
\end{proof}

\begin{remark}\label{rem2.10}
If $W(x,t)=(W_{ij}(x,t))_{N \times N}\geq 0$ for every $(x,t) \in \Omega \times (0, T]$ and $W_{ij} (x,t)
 \in C^{1,1} (\Omega  \times (0, T])$, then for
every $\mathcal {O} \times (t_0 - \delta, t_0] \subset \subset \Omega \times (0, T]$ with $t_0 < T$, there exists a positive
constant $C$, depending only on $dist(\mathcal {O} \times (t_0 - \delta, t_0] ,
\partial (\Omega \times (0, T]))$, $t_0$, $\delta$ and $\left\| W \right\|_{C^{1,1} (\Omega \times (0, T] )}$,
such that
\begin{equation}\label{2.29}
\left| {D W_{ij} } \right| \leqslant C(W_{ii} W_{jj} )^{\frac{1}{4}},
\end{equation}
for every $(x, t) \in \mathcal {O} \times (t_0 - \delta, t_0]$ and $1 \leq i, j \leq N$. Notice that $D W_{ij}
= (\nabla_x W_{ij}, \partial_t W_{ij} )$.
In fact, if $t_0 = T$, it only holds
\begin{equation}\label{2.30}
\left| {\nabla_x W_{ij} } \right| \leqslant C(W_{ii} W_{jj} )^{\frac{1}{4}}.
\end{equation}
for every $(x, t) \in \mathcal {O} \times (t_0 - \delta, t_0]$ and $1 \leq i, j \leq N$.
\end{remark}

\section{A constant rank theorem for the space-time convex solution of heat equation}
\setcounter{equation}{0} \setcounter{theorem}{0}

In this section, we consider the space-time convex solutions of the heat equation
\begin{equation}\label{a3.1}
\frac{{\partial u}} {{\partial t }}=\Delta u, \quad (x,t)
\in \Omega \times (0,T] ,
\end{equation}
and establish the corresponding space-time microscopic convexity principle. The result and its proof belong to Hu-Ma \cite{HM12} and  Chen-Hu \cite{CH12}.

First, we give the definition of the space-time convexity of a function $u(x,t)$.
\begin{definition}
Suppose $u \in C^{2,2} (\Omega \times (0,T] )$, where $ \Omega$ is a
domain in $\mathbb{R}^n $; we say that  $u$ is space-time convex if $u$ is
 convex with respect to $(x,t) \in \Omega \times (0,T] $; equivalently
$$
 D^2 u= \left( {\begin{matrix}
   {\nabla^2 u} & {(\nabla u_t )^T }  \\
   {\nabla u_t } & {u_{tt} }  \\
 \end{matrix} } \right) \geq 0\qquad\text{in }\Omega \times (0,T]\,,
$$
where $\nabla u =(u_{x_1}, \cdots, u_{x_n})$ is the spatial gradient and $\nabla^2 u = \{ \frac{\partial^2 u}{\partial x_i \partial x_j}\}_{1 \leq i, j \leq n}$ is the spatial Hessian.
\end{definition}

The following constant rank theorem is obtained in Hu-Ma \cite{HM12}.

\begin{theorem}\label{tha3.2}
Suppose $\Omega$ is a domain in $\mathbb{R}^n $, and $u \in C^{4,3}(\Omega \times
(0,T])$ is a space-time convex solution of $\eqref{a3.1}$.
Then $D^2 u $ has a constant rank in $\Omega$ for each
fixed $t \in (0,T] $. Moreover, let $l(t)$ be the (constant) rank of
$D^2 u$ in $\Omega$ at time $t$, then $l(s) \leqslant l(t)$ for all
$0< s \leqslant t \leqslant T$.
\end{theorem}

In the following three subsections, we give a brief proof of Theorem \ref{tha3.2} based on the ideas of  \cite{HM12} and \cite{CH12}.

\subsection{The constant rank properties of the spatial Hessian $\nabla^2 u$}

Thanks to the assumptions of Theorem \ref{tha3.2}, we know the spatial Hessian $\nabla^2u \geq 0$.
Suppose $\nabla^2 u$ attains its minimal rank $l$ at some point $(x_0, t_0) \in \Omega \times (0, T]$. We pick
a small open neighborhood $\mathcal {O}$ of $x_0$ and $\delta>0$, and for any fixed point $(x, t) \in \mathcal {O}\times
(t_0-\delta, t_0]$, we rotate the $x$ coordinates so that the matrix $\nabla^2 u (x,t)$ is diagonal
and without loss of generality we assume $ u_{11} \geq u_{22}
\geq \cdots \geq u_{nn} $. Then there is a positive
constant $C > 0$ depending only on $\left\| u \right\|_{C^{3,3} }$, such that $ u_{11} \geq \cdots \geq
u_{ll} \geq C > 0 $ for all $(x, t) \in \mathcal {O}\times (t_0-\delta, t_0]$.
For convenience we set $ G = \{1, \cdots ,l\} $ and $ B =
\{ l+1, \cdots ,n\} $ which means good indices and bad
indices respectively. With abuse of notation, but without confusion, we will also simply set $ G = \{ u_{11} , \cdots ,u_{ll} \} $ and $
B = \{ u_{l+1l+1} , \cdots ,u_{nn} \} $.

Set
\begin{equation}\label{a3.2}
\phi = \sigma_{l+1}(\nabla^2u).
\end{equation}
Then
\begin{align}
 \phi  = \sigma _{l + 1} (\nabla^2 u) \ge \sigma _l (G)\sigma _1 (B) \ge 0,  \notag
 \end{align}
so we get
\begin{align}\label{a3.3}
 u_{ii}  = O(\phi ),\quad\text{for } i \in B .
 \end{align}
By Lemma \ref{lem2.9} and \eqref{a3.3}, we can get
\begin{align}\label{a3.4}
 |\nabla u_{ij}|^2  = O(\phi ), \quad i, j \in B .
 \end{align}

Computing the first derivatives of $\phi$, we obtain
\begin{align}
\label{a3.5}&\phi _i  = \frac{{\partial \phi }}{{\partial x_i }} = \sum\limits_{\alpha  = 1}^{n} {\sigma _l (D^2 u|\alpha )u_{\alpha \alpha i} }  = \sigma _l (G)\sum\limits_{\alpha  \in B} {u_{\alpha \alpha i} }  + O(\phi ), \\
\label{a3.6}&\phi _t  = \frac{{\partial \phi }}{{\partial t}} = \sum\limits_{\alpha  = 1}^{n} {\sigma _l (D^2 u|\alpha )u_{\alpha \alpha t} }  = \sigma _l (G)\sum\limits_{\alpha  \in B} {u_{\alpha \alpha t} }  + O(\phi ),
\end{align}
so from \eqref{a3.5}, we get
\begin{align}\label{a3.7}
\sum\limits_{\alpha  \in B} {u_{\alpha \alpha i} } = O(\phi + |\nabla_x \phi| ), \quad i =1, \cdots, n.
\end{align}
Taking the second derivatives of $\phi$ in $x$ coordinates, we have
\begin{align}\label{a3.8}
\phi _{\alpha \alpha }  =& \frac{{\partial ^2 \phi }}{{\partial x_\alpha  \partial x_\alpha  }} \notag \\
=& \sum\limits_{\gamma  = 1}^{n } {\frac{{\partial \sigma _{l + 1} (D^2 u)}}{{\partial u_{\gamma \gamma } }}u_{\gamma \gamma \alpha \alpha } }  + \sum\limits_{\gamma  \ne \eta } {\frac{{\partial ^2 \sigma _{l + 1} }}{{\partial u_{\gamma \gamma } \partial u_{\eta \eta } }}u_{\gamma \gamma \alpha } u_{\eta \eta \alpha} }  + \sum\limits_{\gamma  \ne \eta } {\frac{{\partial ^2 \sigma _{l + 1} }}{{\partial u_{\gamma \eta } \partial u_{\eta \gamma } }}u_{\gamma \eta \alpha } u_{\eta \gamma\alpha } }  \notag \\
=& \sum\limits_{\gamma  = 1}^{n} {\sigma _l (D^2 u|\gamma )u_{\gamma \gamma \alpha \alpha } }  + \sum\limits_{\gamma  \ne \eta } {\sigma _{l - 1} (D^2 u|\gamma \eta )u_{\gamma \gamma \alpha } u_{\eta \eta \alpha} }  \notag\\
&- \sum\limits_{\gamma  \ne \eta } {\sigma _{l - 1} (D^2 u|\gamma \eta )u_{\gamma \eta \alpha } u_{\eta \gamma \alpha} }  ,
\end{align}
where
\begin{align}
\label{a3.9} \sum\limits_{\gamma  = 1}^{n} {\sigma _l (D^2 u|\gamma )u_{\gamma \gamma \alpha \alpha } }  =& \sum\limits_{\gamma  \in B} {\sigma _l (D^2 u|\gamma )u_{\gamma \gamma \alpha \alpha } }  + \sum\limits_{\gamma  \in G} {\sigma _l (D^2 u|\gamma )u_{\gamma \gamma \alpha \alpha } }  \notag \\
=& \sigma _l (G)\sum\limits_{\gamma  \in B} {u_{\gamma \gamma \alpha \alpha } }  + O(\phi ), \\
\label{a3.10} \sum\limits_{\gamma  \ne \eta } {\sigma _{l - 1} (D^2 u|\gamma \eta )u_{\gamma \gamma \alpha } u_{\eta \eta\alpha} }
=& \sum\limits_{\scriptstyle \gamma, \eta  \in B \hfill \atop  \scriptstyle \gamma  \ne \eta  \hfill} {\sigma _{l - 1} (D^2 u|\gamma \eta )u_{\gamma \gamma \alpha } u_{\eta \eta \beta } }  + \sum\limits_{\scriptstyle \gamma  \in B \hfill \atop  \scriptstyle \eta  \in G \hfill} {\sigma _{l - 1} (D^2 u|\gamma \eta )u_{\gamma \gamma \alpha } u_{\eta \eta \alpha } }  \notag \\
&+ \sum\limits_{\scriptstyle \gamma  \in G \hfill \atop  \scriptstyle \eta  \in B \hfill} {\sigma _{l - 1} (D^2 u|\gamma \eta )u_{\gamma \gamma \alpha } u_{\eta \eta \alpha} }  + \sum\limits_{\scriptstyle \gamma, \eta  \in G \hfill \atop  \scriptstyle \gamma  \ne \eta  \hfill} {\sigma _{l - 1} (D^2 u|\gamma \eta )u_{\gamma \gamma \alpha } u_{\eta \eta \alpha } }   \notag\\
=& O(\phi ) + \sum\limits_{\eta  \in G} {\sigma _{l - 1} (G|\eta )u_{\eta \eta \alpha} } \sum\limits_{\gamma  \in B} {u_{\gamma \gamma \alpha } }  + \sum\limits_{\gamma  \in G} {\sigma _{l - 1} (G|\gamma )u_{\gamma \gamma \alpha } } \sum\limits_{\eta  \in B} {u_{\eta \eta \alpha } }   \notag\\
=&  O(\phi + |\nabla_x \phi|),
\end{align}
and
\begin{align}\label{a3.11}
\sum\limits_{\gamma  \ne \eta } {\sigma _{l - 1} (D^2 u|\gamma \eta )u_{\gamma \eta \alpha } u_{\eta \gamma \alpha } } =& \sum\limits_{\scriptstyle \gamma, \eta  \in B \hfill \atop \scriptstyle \gamma  \ne \eta  \hfill} {\sigma _{l - 1} (D^2 u|\gamma \eta )u_{\gamma \eta \alpha } u_{\eta \gamma \alpha } }  + \sum\limits_{\scriptstyle \gamma  \in B \hfill \atop  \scriptstyle \eta  \in G \hfill} {\sigma _{l - 1} (D^2 u|\gamma \eta )u_{\gamma \eta \alpha } u_{\eta \gamma \alpha} }  \notag \\
&+ \sum\limits_{\scriptstyle \gamma  \in G \hfill \atop \scriptstyle \eta  \in B \hfill} {\sigma _{l - 1} (D^2 u|\gamma \eta )u_{\gamma \eta \alpha } u_{\eta \gamma \alpha} }  + \sum\limits_{\scriptstyle \gamma, \eta  \in G \hfill \atop \scriptstyle \gamma  \ne \eta  \hfill} {\sigma _{l - 1} (D^2 u|\gamma \eta )u_{\gamma \eta \alpha } u_{\eta \gamma \alpha } }  \notag \\
=& O(\phi ) + \sum\limits_{\scriptstyle \gamma  \in B \hfill \atop \scriptstyle \eta  \in G \hfill} {\sigma _{l - 1} (G|\eta )u_{\gamma \eta \alpha } u_{\eta \gamma \alpha} }  + \sum\limits_{\scriptstyle \gamma  \in G \hfill \atop \scriptstyle \eta  \in B \hfill} {\sigma _{l - 1} (G|\gamma )u_{\gamma \eta \alpha } u_{\eta \gamma \alpha } }  \notag \\
=& 2\sigma _l (G)\sum\limits_{\scriptstyle \gamma  \in B \hfill \atop  \scriptstyle \eta  \in G \hfill} {\frac{{u_{\gamma \eta \alpha } u_{\eta \gamma \alpha } }}{{u_{\eta \eta } }}}  + O(\phi ).
\end{align}
So from \eqref{a3.8}-\eqref{a3.11}, we get
\begin{align}\label{a3.12}
\phi _{\alpha  \alpha }  = \sigma _l (G)\sum\limits_{\gamma  \in B} {u_{\gamma \gamma \alpha \alpha} }  - 2\sigma _l (G)\sum\limits_{\scriptstyle \gamma  \in B \hfill \atop \scriptstyle \eta  \in G \hfill} {\frac{{u_{\gamma \eta \alpha } u_{\eta \gamma  \alpha} }}{{u_{\eta \eta } }}} + O(\phi + |\nabla_x \phi|).
\end{align}
By \eqref{a3.5}, \eqref{a3.12} and the equation \eqref{a3.1}, we obtain
\begin{align}\label{a3.13}
\Delta _x \phi  - \phi _t  =& \sigma _l (G)\sum\limits_{\gamma  \in B} {\left[ {\left( {\Delta _x u_{\gamma \gamma }  - u_{\gamma \gamma t} } \right) - 2\sum\limits_{\eta  \in G} {\sum\limits_{i = 1}^n {\frac{{u_{\gamma \eta i} ^2 }}{{u_{\eta \eta } }}} } } \right]}  + O(\phi  + |\nabla _x \phi |) \notag \\
=&  - 2\sigma _l (G)\sum\limits_{\gamma  \in B} {\sum\limits_{\eta  \in G} {\sum\limits_{i = 1}^n {\frac{{u_{\gamma \eta i} ^2 }}{{u_{\eta \eta } }}} } }  + O(\phi  + |\nabla _x \phi |)  \notag\\
\le& C_1(\phi  + |\nabla _x \phi |)-C_2\sum\limits_{i \in B} {\left| {\nabla^2 u_{i} }\right|^2}.
\end{align}
where $C_1$, and $C_2$ are two small positive constants. Together with
\begin{equation}\label{a3.14}
\phi(x,t) \geq 0, \quad  (x,t) \in \mathcal {O} \times (t_0-\delta, t_0], \quad
\phi(x_0, t_0) =0,
\end{equation}
we can apply the strong maximum principle for parabolic equations, and we have
\begin{equation}\label{a3.15}
\phi(x,t) = \sigma_{l+1}(\nabla^2 u) =0,
\end{equation}
and
\begin{equation}\label{a3.16}
\sum\limits_{i \in B} {\left| {\nabla^2 u_{i} }\right|^2}\equiv 0.
\end{equation}

Then we get the following constant rank theorem for the spatial Hessian $\nabla^2 u$.

\begin{theorem}\label{tha3.3}
Under the assumption of Theorem \ref{tha3.2},  $\nabla^2 u $ has a constant rank in $\Omega$ for each
fixed $t \in (0,T] $. Moreover, let $l(t)$ be the minimal rank of
$\nabla^2 u$ in $\Omega$, then $l(s) \leqslant l(t)$ for all
$0< s \leqslant t \leqslant T$.
\end{theorem}

Also we get the following useful properties.

\begin{proposition} \label{propa3.4}
Under above assumptions at $(x,t) \in \mathcal {O} \times (t_0-\delta, t_0]$, we have
\begin{eqnarray}
\label{a3.17}&& u_{ij}(x,t) = 0, \quad  i \text{ or } j \in B, \\
\label{a3.18}&& u_{it}(x,t) = 0, \quad  i \in B,
\end{eqnarray}
and
\begin{eqnarray}
\label{a3.19}&&\sum\limits_{i \in B} {\left( | {\nabla^2 u_{i} }| (x,t) +| {\nabla u_{it} }|(x,t)\right)} = 0.
\end{eqnarray}
\end{proposition}

\begin{proof}
From the choice of coordinate at $(x,t)$, we know
\begin{align*}
u_{ij}(x,t) = 0,  \quad i \ne j.
\end{align*}
By the constant rank theorem of $\nabla^2 u$, i.e. \eqref{a3.15}, we obtain
\begin{align*}
u_{ii}(x,t) = 0,  \quad i \in B.
\end{align*}
Hence, $D^2u\geq 0$ yields
\begin{align*}
u_{it}(x,t) = 0,  \quad i \in B.
\end{align*}
So \eqref{a3.17} and \eqref{a3.18} holds.

By Lemma \ref{lem2.9}, we can get
\begin{align}\label{a3.20}
 |\nabla u_{it}|  \leq C (u_{ii} u_{tt})^{\frac{1}{4}} =0 , \quad i \in B\,,
 \end{align}
which, together with \eqref{a3.16}, gives \eqref{a3.19}.
\end{proof}

\subsection{A constant rank theorem for the space-time Hessian: CASE 1}

In this subsection, we will prove Theorem \ref{tha3.2} in CASE 1 (see Lemma \ref{lem2.8}).
Suppose the space-time Hessian $D^2 u$ attains the minimal rank $l$ at some point $(x_0, t_0) \in \Omega \times (0, T]$. We may assume
$l\leq n$, otherwise there is nothing to prove. Then from lemma
\ref{lem2.8}, there is a neighborhood $\mathcal {O}$ of $x_0$ and $\delta>0$, such that  $ u_{11} \geq
\cdots \geq u_{l-1l-1} \geq C > 0 $ and $u_{tt} -\sum\limits_{i
= 1}^{l-1} {\frac{{u_{it} ^2 }} {{u_{ii} }}} \geq C$ for all $(x,t)
\in \mathcal {O}\times (t_0-\delta, t_0]$. For any fixed
point  $(x,t) \in \mathcal {O}\times (t_0-\delta, t_0]$, we can
rotate the $x$ coordinate so that the matrix $\nabla^2
u$ is diagonal, and without loss of generality we assume $ u_{11}
\geq u_{22}\geq \cdots \geq u_{nn} $. We set $ G = \{ 1, \cdots , l-1 \} $ and $ B = \{ l,
\cdots, n\}$ .

In order to prove the theorem, we just need to prove
\begin{equation} \label{a3.21}
\sigma_{l+1}(D^2 u)\equiv 0, \quad \text{for every}\quad
(x,t) \in \mathcal {O} \times (t_0-\delta, t_0].
\end{equation}
In fact, when $\nabla^2 u$ is diagonal at $(x,t)$, we have
\begin{align} \label{a3.22}
\sigma_{l+1}(D^2 u)=&\sigma_{l+1}(\nabla^2 u)+ u_{tt} \sigma_{l}(\nabla^2 u) - \sum_{i=1}^n u_{it}^2 \sigma_{l-1}(\nabla^2 u|i) \notag \\
\leq& \sigma_{l+1}(\nabla^2 u)+ u_{tt} \sigma_{l}(\nabla^2 u).
\end{align}

In CASE 1, the spatial Hessian $\nabla^2u$ attains the minimal rank $l-1$ at $(x_0, t_0)$.
From Theorem \ref{tha3.3}, the constant rank theorem holds for the spatial
Hessian $\nabla^2u$ of the solution $u$ for the heat equation \eqref{a3.1}, so we can get,
\begin{equation}\label{a3.23}
\sigma_{l+1}(\nabla^2u)=\sigma_{l}(\nabla^2u)\equiv 0, \quad \text{for
every}\quad (x,t) \in \mathcal {O} \times (t_0-\delta, t_0].
\end{equation}
Then
\begin{align} \label{a3.24}
0 \leq \sigma_{l+1}(D^2 u)\leq \sigma_{l+1}(\nabla^2 u)+ u_{tt} \sigma_{l}(\nabla^2 u) =0.
\end{align}
Hence \eqref{a3.21} holds.

By the continuity method, Theorem~\ref{tha3.2} holds in CASE 1.

\subsection{A constant rank theorem for the space-time Hessian: CASE 2}

In this subsection, we will prove Theorem \ref{tha3.2} under CASE 2 (see again Lemma \ref{lem2.8}).
Suppose the space-time Hessian $D^2 u$ attains the minimal
rank $l$ at some point $(x_0, t_0) \in \Omega \times (0, T]$. We may assume
$l\leq n$, otherwise there is nothing to prove. Under CASE 2, $l$ is also the minimal rank of $\nabla^2u$ in $\Omega \times (t_0-\delta, t_0]$.
For each fixed $(x,t)\in\mathcal{O}\times(t_0-\delta, t_0]$,
we choose a local orthonormal frame $e_1,\dots,e_n$ so that $\nabla^2 u$
is diagonal and  let $u_{ii}=\lambda_i, \quad i=1, \cdots, n$. We arrange
$\lambda_1\geq\lambda_2 \geq \cdots \geq\lambda_n\ge 0$, where
$\lambda=(\lambda_1, \lambda_2, \cdots, \lambda_n)$ are the eigenvalues
of $\nabla^2u$ at $(x,t)$. As before, we let $G=\{1, \cdots, l\}$ and
 $B=\{l+1, \cdots, n\}$ be the ``good'' set and ``bad''
 set of indices respectively.  Without confusion we will also again denote $ G = \{ u_{11} , \cdots ,u_{ll} \} $ and $
B = \{ u_{l+1l+1} , \cdots ,u_{nn} \} $.

At $(x,t)$, by the constant rank properties Proposition \ref{propa3.4} we have
\begin{eqnarray}
\label{a3.25}&& u_{ij}(x,t) = 0, \quad  i \text{ or } j \in B, \\
\label{a3.26}&& u_{it}(x,t) = 0, \quad  i \in B,
\end{eqnarray}
and
\begin{eqnarray}\label{a3.27}
| {\nabla^2 u_{i} }| (x,t)=0, \quad | {\nabla u_{it} }|(x,t) = 0, \quad i \in B.
\end{eqnarray}

Let's set
\begin{align}\label{a3.28}
\phi=\sigma_{l+1}(D^2u),
\end{align}
then we have at $(x,t)$
\begin{align*}
\phi=&\sigma_{l+1}(D^2u)=\sigma_{l+1}(\nabla^2u)+u_{tt}\sigma_{l}(\nabla^2u)-\sum_{i}u_{ti}^2\sigma_{l-1}(\nabla^2u|i)\notag\\
= &\sigma_{l}(G)(u_{tt}-\sum_{i\in G}\frac{u_{ti}^2}{\lambda_i}),
\end{align*}
so
\begin{align}\label{a3.29}
u_{tt}-\sum_{i\in G}\frac{u_{ti}^2}{\lambda_i}= O(\phi).
\end{align}

Taking the first derivative of $\phi$ with respect to $t$, then using \eqref{a3.25}-\eqref{a3.27}, we have
\begin{align}\label{a3.30}
\phi_t=&\sum_i\sigma_l(D^2u|i)u_{iit}+u_{ttt}\sigma_l(D^2u)+u_{tt}\sum_i\sigma_{l-1}(D^2u|i)u_{iit}\notag\\
&-2\sum_i\sigma_{l-1}(D^2u|i)u_{ti}u_{tit}-\sum_{i\neq j}\sigma_{l-2}(D^2u|ij)u_{tj}^2u_{iit}\notag\\
&+\sum_{i\neq j}\sigma_{l-2}(D^2u|ij)u_{ti}u_{tj}u_{ijt}\notag\\
=&\sigma_{l}(G)u_{ttt}+u_{tt}\sum_{i\in G}\sigma_{l-1}(G|i)u_{iit}
-2\sum_{i\in G}\sigma_{l-1}(G|i)u_{ti}u_{tit} \notag\\
&-\sum_{\stackrel{i,j\in G}{i\neq j}}\sigma_{l-2}(G|ij)u_{tj}^2u_{iit}+\sum_{\stackrel{i,j\in G}{i\neq j}}\sigma_{l-2}(G|ij)u_{ti}u_{tj}u_{ijt} \notag \\
=&\sigma_l(G)(u_{ttt}-2\sum_{i\in G}\frac{u_{ti}}{\lambda
_i}u_{tit}+\sum_{i,j\in G}\frac{u_{ti}}{\lambda
_i}\frac{u_{tj}}{\lambda _j}u_{ijt}) + O(\phi).
\end{align}

Similarly, taking the first derivative of $\phi$ in the direction
$e_\alpha$, it follows that
\begin{align}
\phi_{\alpha}=\sigma_l(G)(u_{tt\alpha}-2\sum_{i\in G}\frac{u_{ti}}{\lambda
_i}u_{ti\alpha} +\sum_{i,j\in G}\frac{u_{ti}}{\lambda
_i}\frac{u_{tj}}{\lambda_j}u_{ij\alpha}) +  O(\phi), \notag
\end{align}
whence
\begin{align}\label{a3.31}
u_{tt\alpha}-2\sum_{i\in G}\frac{u_{ti}}{\lambda_i}u_{ti\alpha} +\sum_{i,j\in G}\frac{u_{ti}}{\lambda
_i}\frac{u_{tj}}{\lambda_j}u_{ij\alpha}=  O(\phi + |\nabla \phi|).
\end{align}

Computing the second derivatives, we have
\begin{align*}
\Delta \phi=& \sum_{i,j}\frac{{\partial \sigma _{l + 1} (D^2 u)}}{{\partial u_{ij} }}\Delta u_{ij }
+ \sum_{i,j,k,l}\frac{{\partial ^2 \sigma_{l + 1} (D^2 u)}} {{\partial u_{ij}\partial u_{kl} }}u_{ij\alpha }u_{kl\alpha}
+ \Delta u_{tt} \sigma _l (D^2 u)\\
&+ 2u_{tt\alpha }\sum_{i,j}\frac{{\partial \sigma _l (D^2u)}} {{\partial u_{ij}}}u_{ij\alpha}
+u_{tt}\sum_{i,j}\frac{{\partial \sigma _l (D^2u)}} {{\partial u_{ij}}} \Delta u_{ij}+u_{tt}\sum_{i,j,k,l}\frac{{\partial ^2 \sigma _l (D^2u)}}{{\partial u_{ij}\partial u_{kl}}}u_{ij\alpha }u_{kl\alpha}\\
&-2\sum\limits_i{\sigma _{l -1} (D^2u|i)u_{ti} \Delta u_{ti }} -2\sum\limits_i{\sigma _{l - 1} (D^2u|i)u_{ti\alpha }u_{ti\alpha}}- 4\sum\limits_{i,j,k}{\frac{{\partial\sigma_{l - 1}(D^2u|i)}}{{\partial u_{jk}}}u_{ti}u_{ti\alpha}u_{jk\alpha }}\notag \\
&-\sum_{i,j,k}\frac{{\partial\sigma_{l-1}(D^2u|i)}}{{\partial u_{jk}}}u_{ti}^2 \Delta u_{jk}
-\sum_{i,j,k,p,q}\frac{{\partial^2\sigma_{l-1}(D^2u|i)}}{{\partial u_{jk}\partial u_{pq}}}u_{ti}^2u_{jk\alpha }u_{pq\alpha}\\
&+\sum_{\stackrel{i,j}{i\ne j}}{\sigma_{l-2}(D^2u|i,j)}u_{ti}u_{tj}\Delta u_{ij}
+2\sum_{\stackrel{i,j}{i\ne j}}{\sigma_{l-2}(D^2u|i,j)u_{tj}u_{ti\alpha}u_{ij\alpha}}\notag \\
&+2\sum_{\stackrel{i,j}{i\ne j}}{\sigma_{l-2}(D^2u|i,j)u_{ti} u_{tj\alpha }u_{ij\alpha}}
+2\sum_{\stackrel{i,j,k,l}{i\ne j}}{\frac{{\partial\sigma_{l-2}(D^2u|i,j)}}{{\partial u_{kl}}}u_{ti}u_{tj}u_{ij\alpha}u_{kl\alpha}}\\
&-2\sum_{\stackrel{i,j,k }{i\ne j,i\ne k,j\ne k}}{\sigma_{l-3}(D^2u|i,j,k)u_{ti}u_{tj}u_{ki\alpha}u_{kj\alpha}},
\end{align*}

For any $i \in B$, we have from \eqref{a3.27}
\begin{align*}
\Delta u_{ii} =  u_{iit} =0.
\end{align*}
Hence
\begin{align*}
\Delta \phi=& \sigma _{l} (G)\sum_{i\in B}\Delta u_{ii}+ \Delta u_{tt} \sigma _l (G)+ 2u_{tt\alpha }\sum_{i \in G}\sigma _{l-1} (G|i)u_{ii\alpha}\\
&+u_{tt}[\sum_{i\in G}\sigma _{l-1} (G|i) \Delta u_{ii} + \sum_{i\in B}\sigma _{l-1} (G) \Delta u_{ii}]+ u_{tt} \sum_{i \ne j \in G} \sigma _{l-2} (G|ij) [u_{ii\alpha }u_{jj\alpha} - u_{ij\alpha }u_{ji\alpha}]\\
&-2\sum\limits_{i \in G} {\sigma _{l -1} (G|i)u_{ti} \Delta u_{ti }} -2\sum\limits_{i \in G}{\sigma _{l - 1} (G|i)u_{ti\alpha }u_{ti\alpha}}- 4\sum\limits_{i \ne j \in G}{\sigma_{l - 2}(G|ij)u_{ti}u_{ti\alpha}u_{jj\alpha }}\notag \\
&-\sum_{i\ne j \in G}\sigma_{l - 2}(G|ij)u_{ti}^2 \Delta u_{jj}-\sum_{i\in G, j \in B}\sigma_{l - 2}(G|i)u_{ti}^2 \Delta u_{jj}
\notag\\
&-\sum_{i\ne j \ne k \in G}\sigma_{l -3}(G|ijk)u_{ti}^2 [u_{jj\alpha }u_{kk\alpha} -u_{jk\alpha }u_{kj\alpha} ]\\
&+\sum_{{i\ne j \in G}}{\sigma_{l-2}(G|i,j)}u_{ti}u_{tj}\Delta u_{ij}
+2\sum_{{i\ne j \in G}}{\sigma_{l-2}(G|i,j)u_{tj}u_{ti\alpha}u_{ij\alpha}}\notag \\
&+2\sum_{{i\ne j \in G}}{\sigma_{l-2}(G|i,j)u_{ti} u_{tj\alpha }u_{ij\alpha}}
+2\sum_{{i\ne j \ne k \in G}}{\sigma_{l-3}(G|i,j,k)u_{ti}u_{tj}u_{ij\alpha}u_{kk\alpha}}\\
&-2\sum_{i\ne j \ne k \in G}{\sigma_{l-3}(G|i,j,k)u_{ti}u_{tj}u_{ki\alpha}u_{kj\alpha}} + O(\phi)\\
\sim &\sigma_l(G)[\Delta u_{tt}-2\sum_{i\in G}\frac{u_{ti}}{\lambda_i}\Delta u_{it}+\sum_{i,j\in G}
\frac{u_{ti}}{\lambda_i}\frac{u_{tj}}{\lambda_j}\Delta u_{ij}]-2\sigma_l(G)\sum_{i\in
G}\frac{1}{\lambda_i}(u_{ti\alpha}-\sum_{j\in G}\frac{u_{tj}}{\lambda_j}u_{ij\alpha})^2 \\
& + O(\phi + |\nabla \phi|).
\end{align*}
So we can write
\begin{align}
\Delta \phi - \phi_t =& -2\sigma_l(G)\sum_{i\in
G}\frac{1}{\lambda_i}(u_{ti\alpha}-\sum_{j\in G}\frac{u_{tj}}{\lambda_j}u_{ij\alpha})^2 + O(\phi + |\nabla \phi|) \notag \\
\leq&  C(\phi + |\nabla \phi|).
\end{align}
Together with
\begin{equation}
\phi(x,t) \geq 0, \quad  (x,t) \in \mathcal {O} \times (t_0-\delta, t_0], \quad
\phi(x_0, t_0) =0,
\end{equation}
we can apply the strong maximum principle of parabolic equations, and we obtain
\begin{equation}\label{a3.34}
\phi(x,t) =\sigma_{l+1}(D^2u)\equiv 0, \quad  (x,t) \in \mathcal {O} \times (t_0-\delta, t_0].
\end{equation}

By the continuity method, Theorem \ref{tha3.2} holds under CASE 2. The proof of Theorem \ref{tha3.2} is complete.

\section{The strict convexity of the level sets of harmonic functions in convex rings}

In this section, we consider
the following initial boundary value problem
\begin{equation}\label{a4.1}
\left\{\begin{array}{lcl}
              \Delta u= 0  &\text{in}& \Omega= \Omega _0 \backslash \overline {\Omega _1 },   \\
              u=  0   &\text{on}&  \partial \Omega_0,\\
              u=  1   &\text{in}&  \partial \Omega_1,
\end{array} \right.
\end{equation}
where $\Omega=\Omega_0\setminus\overline\Omega_1$ is a $C^{2}$ convex ring in $\mathbb{R}^n$ ($n\geq 2$), i.e. $\Omega _0$ and
$\Omega _1$ are bounded convex $C^{2}$ domains in $\mathbb{R}^n$ with $
\overline \Omega _1  \subset\Omega _0$.

Notice that in these assumptions, it holds $|\nabla u| \ne 0$ in $\Omega$ (see Kawohl \cite{Ka85}) and the level sets $\partial\Sigma_x^c =\{x \in \Omega: u =c \} $ are $n-1$ dimensional hypersurfaces for $c \in (0, 1)$.
We ask whether
$\partial\Sigma_x^c$ is convex or strictly convex for every $c \in (0,1)$.

In 1957, Gabriel \cite{Ga57} proved that the level sets of the Green function of a 3-dimension bounded convex domain are strictly
convex. Later, in 1977, Lewis \cite{Le77} extended Gabriel's result to $p$-harmonic functions in
higher dimensions and obtained the following theorem.
\begin{theorem}(Gabriel \cite{Ga57} and Lewis \cite{Le77})\label{tha4.1}
Suppose $\Omega=\Omega_0\setminus\overline\Omega_1$ is a $C^{2}$ convex ring, and $u \in C^2(\overline{\Omega})$ satisfies \eqref{a4.1}.
Then the level set $\partial\Sigma_x^c$ of $u$ is strictly convex  for every $c \in (0,1)$.
\end{theorem}

Motivated by a result of Caffarelli-Friedman \cite{CF85}, Korevaar \cite{Ko90} gave a new proof of Theorem \ref{tha4.1}.
Also Ma-Ou-Zhang \cite{MOZ10} gave a different proof of Theorem \ref{tha4.1}, based on a quantitative Gauss curvature estimate for the curvature of the level sets.

In the following, we give a brief proof of Theorem \ref{tha4.1} in two subsections. The result belongs to \cite{Ko90} and its proof comes from \cite{BGMX, Ko90}.

\subsection{A constant rank theorem for the second fundamental form  of the level sets of harmonic functions}

In this subsection, we use the notation of Subsection 2.1.1.

\begin{theorem}(Korevaar~\cite{Ko90})\label{tha4.2}
Suppose $\Omega=\Omega_0\setminus\overline\Omega_1$ is a $C^{2}$ convex ring, and $u \in C^2(\overline{\Omega})$
is a quasiconcave function satisfying the equation \eqref{a4.1}. Then the second fundamental form
 of level sets $\partial\Sigma_x ^{c}$ has constant rank in $\Omega$.
\end{theorem}

\begin{proof}
By the regularity theory of harmonic functions, $u \in C^\infty (\Omega) \cap C^2(\overline{\Omega})$. And the second
fundamental form $a(x) = \{ a_{ij}\}$ of a level set of $u$ is as in \eqref{2.5a}.

Suppose $a(x)$ attains minimal rank $l$ at some point $x_0 \in
\Omega$. We can assume $l\leqslant n-2$, otherwise there is
nothing to prove. We also assume $u_n>0$, and pick a small neighborhood $\mathcal {O}$ of
$x_0$. For any
fixed point  $x \in \mathcal {O}$, we can choose $e_1,\cdots, e_{n-1},e_n$ such
that
\begin{equation}\label{a4.2}
 |\nabla u(x)|=u_n(x)>0\ \quad \text{and}\quad
\{ u_{ij} \}_{1 \leq i,j \leq n-1} \text{is diagonal at}\ x.
\end{equation}
Without loss of generality we also assume $ u_{11} \leq u_{22}\leq \cdots
\leq u_{n-1n-1} $. So, at $x \in \mathcal {O}$, by \eqref{2.5a}, we have the matrix
$\{ a_{ij} \}$ is also diagonal and $a_{11} \geq a_{22} \geq \cdots \geq a_{n-1
n-1}$. There is a positive constant $\delta>0$ depending only on
$\|u\|_{C^{4}}$ and $\mathcal {O}$,
such that $a_{11} \ge a_{22} \ge \cdots \ge a_{ll} > \delta$. As before, we denote the ``good" and ``bad" sets of indices by $ G = \{ 1, \cdots ,l \} $ and $ B = \{ l+1, \cdots, n-1 \} $ respectively. If there is no confusion, we also denote
\begin{eqnarray}\label{a4.3}
\quad \quad \mbox{\it $G=\{a_{11}, \cdots, a_{ll}\}$ and
$B=\{a_{l+1l+1}, \cdots, a_{n-1n-1}\}$.}
\end{eqnarray}

Set
\begin{eqnarray}\label{3.3}
\phi(x)=\sigma_{l+1}(a_{ij}).
\end{eqnarray}
Following the notations in \cite{CF85} and \cite{KL87}, if $h$ and
$g$ are two functions defined in $\mathcal {O}$, we say $h\lesssim g$ if there exist positive constants
$C_1$ and $C_2$ depending only on $||u||_{C^{4}},n$ (independent
of $x$), such that $(h-g)(x)\leq(C_1\phi+C_2|\nabla\phi|)(x)$, $\forall x\in
\mathcal {O}$. We also write
$$
h\sim g\quad \text{ if }\quad h\lesssim g, \text{ and }\quad g\lesssim
h.
$$

For any fixed point $x \in \mathcal {O}$, we choose a coordinate system as
in \eqref{a4.2} so that $|\nabla u|=  u_n >0$ and the matrix
$\{ a_{ij}(x,t) \}$ is diagonal and nonnegative.
From the definition of $\phi$, we can get
\begin{equation}\label{a4.5}
a_{ii} \sim 0,\quad ~~\forall i \in B,
\end{equation}
and
\begin{equation}\label{a4.6}
h_{ii} \sim 0, \quad u_{ii} \sim 0, ~~\forall i\in B.
\end{equation}

Taking the first derivatives of $\phi$, we get
\begin{eqnarray}
\phi_{\a} &=& \sum_{ij=1}^{n-1}\frac{\partial \sigma_{l+1}(a)}{\partial a_{ij}} a_{ij,\a}  = \sum_{i \in G}  \sigma_l(a | i )
a_{ii,\a}+  \sum_{i \in B}  \sigma_l(a| i ) a_{ii, \a} \notag \\
&\sim&  \sigma_l(G) \sum_{i \in B}  a_{ii, \a} \sim  -u_n^{-3}\sigma_l(G) \sum_{i \in B}  h_{ii,\a}  \notag  \\
&\sim& -u_n^{-3} \sigma_l(G) \sum_{i \in B}  [u_n^2u_{ii\a}-2u_nu_{in}u_{i\a} ], \notag
\end{eqnarray}
whence
\begin{eqnarray}\label{a4.7}
\sum_{i \in B}  a_{ii,\a} \sim  0, \quad  \sum_{i \in B}  h_{ii, \a} \sim 0, \quad  \sum_{i \in B}  [u_n^2u_{ii\a}-2u_nu_{in}u_{i \a} ] \sim 0.
\end{eqnarray}
Then
\begin{eqnarray}\label{a4.8}
 \sum_{i \in B} u_{iij} \sim 0, \quad \forall j \in G.
\end{eqnarray}

Taking the second derivatives of $\phi$, we have from \eqref{a4.5} and \eqref{a4.7}
\begin{eqnarray}\label{a4.9}
\phi_{\alpha \alpha}&=&\sum_{i,j=1}^{n-1}\frac{\partial \sigma_{l+1}(a)}{\partial a_{ij}} a_{ij,\alpha \alpha}
+ \sum_{i,j,k,l=1}^{n-1}\frac{\partial^2 \sigma_{l+1}(a)}{\partial a_{ij}\partial a_{kl}} a_{ij,\alpha}a_{kl,\alpha} \notag \\
&=& \sum_{j=1}^{n-1}\sigma_{l}(a|j) a_{jj,\alpha \alpha} + \sum_{i,j=1, i\ne j}^{n-1}\sigma_{l-1}(a|ij) a_{ii,\alpha}a_{jj,\alpha}
-\sum_{i,j=1, i\ne j}^{n-1}\sigma_{l-1}(a|ij) a_{ij,\alpha}a_{ji,\alpha} \notag \\
&\sim& \sum_{j\in B}\sigma_{l}(G) a_{jj,\alpha \alpha} + [\sum_{i,j\in G, i\ne j}+\sum_{i \in G, j\in B}
+\sum_{i \in B, j\in G}+\sum_{i,j\in B, i\ne j} ]\sigma_{l-1}(a|ij) a_{ii,\alpha}a_{jj,\alpha} \notag\\
&&\qquad-[\sum_{i,j\in G, i\ne j}+\sum_{i \in G, j\in B}+\sum_{i \in B, j\in G}+\sum_{i,j\in B, i\ne j} ]\sigma_{l-1}(a|ij) a_{ij,\alpha}a_{ji,\alpha} \notag \\
&\sim& \sum_{j\in B}\sigma_{l}(G) a_{jj,\alpha \alpha} -2\sum_{i \in G, j\in B} \sigma_{l-1}(G|i) a_{ij,\alpha}a_{ji,\alpha} \notag \\
&=& \sigma_l(G) \sum_{j\in B} \Big[a_{jj,\alpha \alpha}-2\sum_{i\in
G}\frac{a_{ij,\alpha}a_{ij,\alpha}}{a_{ii}}\Big],
\end{eqnarray}
where we have used the following inequalities from Lemma \ref{lem2.9}:
\begin{align*}
&|a_{ii,\a}a_{jj,\b} | \leq C a_{ii}^{\frac{1}{2}}\cdot C a_{jj}^{\frac{1}{2}} \leq C_1 \phi, \quad  i, j\in B, i\ne j;  \\
&|a_{ij,\a}a_{ji,\b}| \leq C [a_{ii}a_{jj} ] ^{\frac{1}{4}} \cdot C [a_{ii}a_{jj} ] ^{\frac{1}{4}}  \leq C_2 \phi, \quad i,j\in B, i\ne j.
\end{align*}

Since $u_k=0$ for $k=1,\cdots, n-1$, from \eqref{2.5a},
\begin{eqnarray}
u_n u_{ij\a}=-u_n^2 a_{ij,\a}+u_{nj}u_{i\a} +u_{ni} u_{j\a}+ u_{n\a} u_{ij}, \quad \forall\; i,j \le n-1, \notag
\end{eqnarray}
and
\begin{align}
\sum_{j\in B}a_{jj,\alpha \alpha} \sim&-\frac{1}{u_n^3} \sum_{j\in B} h_{jj,\alpha \a}- 2 (\frac{|u_n|}{|\nabla u|{u_n}^3})_{\a} \sum_{j\in B} h_{jj,\alpha }  \notag  \\
\sim&-\frac{1}{u_n^3}\sum_{j\in B}[u_n^2u_{jj\alpha\a}-2u_nu_{nj}u_{\a\a
j}+2u_{nn}u_{j\a}^2+4u_{n\a}u_{nj}u_{j\a} -4u_nu_{j\a}u_{nj\a}]. \notag
\end{align}
Hence it yields
\begin{align*}
\sum_{j \in B} \sum_{\alpha=1}^n a_{jj,\alpha \a} \sim
-\frac{1}{u_n^3} \sum_{j \in B}  \sum_{\alpha=1}^n [2u_{nn}u_{j\a}^2+4u_{n\a}u_{nj}u_{j\a} -4u_nu_{j\a}u_{nj\a}],
\end{align*}
where
\begin{align}
\sum_{\alpha=1}^n[2u_{nn}u_{j\a}^2+&4u_{n\a}u_{nj}u_{j\a} -4u_nu_{j\a}u_{nj\a}]\notag \\
\sim& 2u_{nn}u_{jn}^2 + 4u_{nn}u_{nj}u_{jn} -4u_nu_{jn}u_{njn}\notag \\
=& 6u_{nn}u_{nj}u_{jn} -4u_nu_{jn}u_{nnj}\notag \\
=& 6[\Delta u - \sum_{i=1}^{n-1} u_{ii}]u_{nj}u_{jn} -4u_nu_{jn}[\Delta u_j - \sum_{i=1}^{n-1} u_{iij}]
\notag \\
=&  - 6 u_{nj}^2 \sum_{i=1}^{n-1} u_{ii}+ 4u_nu_{jn} \sum_{i=1}^{n-1} u_{iij}\notag \\
\sim&   - 6 u_{nj}^2 \sum_{i\in G} u_{ii}+ 4u_nu_{jn} \sum_{i\in G} u_{iij}. \notag
\end{align}
For $i \in G, j\in B$, we have
\begin{align}
\sum\limits_{\alpha  = 1}^n {\frac{{a_{ij,\alpha } ^2 }}{{a_{ii} }}}  =& - \frac{1}{{u_n ^3 }}\sum\limits_{\alpha  = 1}^n {\frac{{[u_n ^2 a_{ij,\alpha } ]^2 }}{{u_{ii} }}}  \sim  - \frac{1}{{u_n ^3 }}\sum\limits_{\alpha  = 1}^n {\frac{{[u_n u_{ij\alpha }  - u_{i\alpha } u_{nj}  - u_{j\alpha } u_{ni} ]^2 }}{{u_{ii} }}}\notag \\
=&  - \frac{1}{{u_n ^3 }}\sum\limits_{\alpha  = 1}^n {\frac{{[u_n u_{ij\alpha }  - 2u_{i\alpha } u_{nj}  + u_{i\alpha } u_{nj}  - u_{j\alpha } u_{ni} ]^2 }}{{u_{ii} }}}\notag \\
=& - \frac{1}{{u_n ^3 }}\sum\limits_{\alpha  = 1}^n {\frac{{[u_n u_{ij\alpha }  - 2u_{i\alpha } u_{nj} ]^2  + 2[u_n u_{ij\alpha }  - 2u_{i\alpha } u_{nj} ][u_{i\alpha } u_{nj}  - u_{j\alpha } u_{ni} ] + [u_{i\alpha } u_{nj}  - u_{j\alpha } u_{ni} ]^2 }}{{u_{ii} }}}
\notag \\
\sim& - \frac{1}{{u_n ^3 }}\left\{ {\sum\limits_{\alpha  = 1}^n {\frac{{[u_n u_{ij\alpha }  - 2u_{i\alpha } u_{nj} ]^2 }}{{u_{ii} }} + 2u_n u_{iij} u_{nj}  - 3u_{ii} u_{nj} ^2 } } \right\}. \notag
\end{align}
Then
\begin{align}\label{a4.10}
\sum_{j \in B} \sum_{\a=1}^n [a_{jj,\alpha \alpha}-2\sum_{i\in
G}\frac{a_{ij,\alpha}^2}{a_{ii}}] \sim2u_n^{-3}\sum_{j \in B,i \in G}\sum_{\alpha=1}^{n}\frac{[u_nu_{ij\alpha}-2u_{i\a}u_{jn}]^2}{u_{ii}}.
\end{align}

Since $a_{ii} =-\frac{u_{ii}}{u_n}>0$ for $i \in G$, we have $u_{ii} < 0$ for $i \in G$. Hence
\begin{align*}
\Delta \phi(x)=& 2 u_n^{-3}\sigma_l(G) \sum_{j\in B,i\in G} \sum_{\alpha=1}^{n}\frac{[u_nu_{ij\alpha}-2u_{i\a}u_{jn}]^2}{u_{ii}} +O(\phi + |\nabla \phi|) \\
\leq& C(\phi + |\nabla \phi|).
\end{align*}
Together with
\begin{equation}
\phi(x) \geq 0, \quad  x \in \mathcal {O}, \quad \phi(x_0) =0,
\end{equation}
we can apply the strong maximum principle of elliptic equations, and we obtain
\begin{equation}\label{a3.34}
\phi(x) =\sigma_{l+1}(a_{ij})\equiv 0, \quad x \in \mathcal {O}.
\end{equation}

By the continuity method, Theorem \ref{tha4.2} holds.
\end{proof}

\subsection{The strict convexity of the level sets of $u(x)$}

Let $0 \in \Omega_1$. At the initial time we let the domain be the standard ball ring $U =
B_R(0)\setminus \overline{B_r(0)}$ ($0 < r < R$), and for $t\in[0,1]$, we set
\begin{align}
&\Omega_{0,t} = (1 - t)B_R(0) + t \Omega_0,\\
&\Omega_{1,t} = (1 - t)B_r(0) + t \Omega_1,\\
&\Omega_t=\Omega_{0,t}\setminus\overline\Omega_{1,t},
\end{align}
where the sum is the Minkowski vector sum. So the domain $\Omega_t$ is a family
of $C^2$ strictly convex rings ( see Schneider \cite{Sc93}) for $0 \leq t <1$. We denote as $u_t$ the solution of the following Dirichlet problem
\begin{equation}\label{a4.14}
\left\{\begin{array}{lcl}
              \Delta u_t(x) = 0  &\text{in}& \Omega_t\,,   \\
              u_t(x) =  0   &\text{on}&  \partial \Omega_{0,t},\\
              u_t(x) =  1   &\text{in}&  \partial \Omega_{1,t}.
\end{array} \right.
\end{equation}

By the maximum principle $|\nabla u_t| \ne 0$ in $\Omega_t$ (see Kawohl \cite{Ka85}), and by the
standard elliptic theory we have uniform estimates on $|u_t|_{C^3(\Omega_t)}$ only
depending on the geometry of $\Omega$. When $t=0$, $\Omega_0 =U$, and each level set $\partial\Sigma_x^{c,0} = \{ x \in U: u_0 =c \}$
is a ball. Hence $\partial\Sigma_x^{c,0}$ is strictly convex for each $c \in (0,1)$. If $0 < t_0 \leq 1$ is the first time that the level sets of $u_{t_0}$ becomes convex but not strictly convex at some point $x_{t_0} \in \Omega_{t_0}$, we can use the constant rank theorem (that is Theorem \ref{tha4.1}) for $u_{t_0}$. Hence each level set $\partial\Sigma_x^{c, t_0} = \{ x \in \Omega_{t_0}: u_{t_0} =c \}$ is convex but not strictly convex. But $\partial\Sigma_x^{c, t_0}$ is a closed convex hypersurface, and there is at least a strictly convex point on  each $\partial\Sigma_x^{c, t_0}$. This is a contradiction.

\newpage

\chapter{A microscopic space-time Convexity Principle for space-time level sets}

In this chapter, we prove the Theorem \ref{th1.3}. On the proof of constant rank theorem on the space-time convex solutions of the heat equation in Theorem~\ref{tha3.2}, we first get the constant rank properties of spatial Hessian $\nabla^2 u$ in Theorem~\ref{tha3.3}, then we can obtain some useful properties in Proposition~\ref{propa3.4}, at last we complete the proof Theorem~\ref{tha3.2} according {\bf CASE 1} and {\bf CASE 2} in Section 2.2.
Using the similar idea in Section 2.2, in order to prove Theorem \ref{th1.3} on the constant rank theorem for the second
fundamental form of the space-time level sets of a space-time quasiconcave solution of the heat equation, we divide three sections to complete the proof on this theorem. Firstly in Section 3.1 we prove Theorem \ref{th3.1}, a constant rank theorem for the second fundamental form of the spatial level sets of a space-time quasiconcave solution to the heat equation \eqref{1.3}, and obtain some constant rank properties. Then the proof of Theorem \ref{th1.3} is split into two cases  (see Lemma \ref{lem2.8}). CASE 1 is treated in Section 3.2,  using the constant rank theorem established in Section 3.1. Finally we deal with CASE 2 in Section 3.3, completing the proof of Theorem \ref{th1.3}.

\section{A constant rank theorem for the spatial second fundamental form}
\setcounter{equation}{0} \setcounter{theorem}{0}

In this Section we shall consider the spatial level sets of $u$ and obtain the following constant rank theorem for
the spatial second fundamental form.

\begin{theorem}\label{th3.1}
Suppose $u \in C^{4,3}(\Omega \times (0,T])$ is a space-time quasiconcave solution of the heat equation
 \eqref{1.3} with $u_t >0$ and $|\nabla u|>0$ in $\Omega \times (0,T]$. Then the second fundamental form
 of spatial level sets $\partial\Sigma_x ^{c,t}=\{x \in \Omega | u(x,t) =
c\}$ has the constant rank property in $\Omega$ for all $c\in (0,1)$,  i.e.
if the rank of $II_{\partial{\Sigma}_x^{c,t}}$ attains its minimum rank $l_0$ $(0\leq
l\leq n-1)$ at some point $(x_0,t_0) \in \Omega \times (0,T)$, then the rank of
$II_{\partial{\Sigma}^{c,t}_{x}}$ is constant on $\Omega\times (0,t_0]$. Moreover, let
$l(t)$ be the minimal rank of the second fundamental form
$II_{\partial{\Sigma}^{c,t}_{x}}$ in $\Omega$, then $l(s) \leqslant l(t)$ for all $0 <
s \leqslant t \leqslant  T$.
\end{theorem}

The proof is split into two subsections.

\subsection{Some preliminary calculations for a test function}

Since Theorem \ref{th3.1} is of local nature, we can assume that the level surface
$\partial\Sigma_x^{c,t} = \{x \in \Omega| u(x,t)=c\}$ be connected for each $c\in (0,
1)$. Suppose $a(x,t)$ attains minimal rank $l$ at some point $(x_0,t_0) \in
\Omega \times (0, T]$. Let us assume $l\leqslant n-2$, otherwise there is
nothing to prove. Furthermore we assume $u \in C^{4,3}(\Omega\times (0, T])$ and
$u_n>0$. We can pick a parabolic neighborhood $\mathcal {O}\times (t_0-\delta, t_0]$ of $(x_0, t_0)$ for $\delta>0$. For any
fixed point  $(x,t) \in \mathcal {O}\times (t_0-\delta, t_0]$, we can express
$(a_{ij})$ as in \eqref{2.5}, by choosing $e_1,\cdots, e_{n-1},e_n$ such
that
\begin{equation}\label{3.1}
 u_n(x,t)= |\n u(x,t)|>0\ \quad \text{and}\quad
(u_{ij})_{1 \leq i, j \leq n-1} \text{is diagonal at}\ (x,t).
\end{equation}
Without loss of generality we can assume $ u_{11} \leq u_{22}\leq \cdots
\leq u_{n-1n-1} $. So, at $(x,t) \in \mathcal {O}\times (t_0-\delta,
t_0]$, from \eqref{2.5}, we have the matrix
$(a_{ij})_{1 \leq i, j \leq n-1}$ is also diagonal, and without loss of
generality we may assume $a_{11} \geq a_{22} \geq \cdots \geq a_{n-1
n-1}$. There is a positive constant $C>0$ depending only on
$\|u\|_{C^{4}}$ and $\mathcal {O}\times (t_0-\delta, t_0]$,
such that $a_{11} \ge a_{22} \ge \cdots \ge a_{ll} > C$ for all $(x,t)
\in \mathcal {O}\times (t_0-\delta, t_0]$. Let
$ G = \{ 1, \cdots ,l \} $ and $ B = \{
l+1, \cdots, n-1 \} $ be the ``good" and ``bad" sets of indices
respectively, and, if there is no confusion, we also set
\begin{eqnarray}\label{3.2}
\quad \quad \mbox{\it $G=\{a_{11}, \cdots, a_{ll}\}$ and
$B=\{a_{l+1l+1}, \cdots, a_{n-1n-1}\}$.}
\end{eqnarray}
Note that for any $\epsilon>0$, we may choose $\mathcal {O}\times (t_0-\delta,
t_0]$ small enough such that $a_{jj} <\epsilon$ for all $j \in
B$ and $(x,t) \in \mathcal {O}\times (t_0-\delta, t_0]$.

For each $c$, let $a=(a_{ij})$ be the symmetric Weingarten tensor of
$\partial\Sigma_x^{c,t}$. Set
\begin{eqnarray}\label{3.3}
\phi(x,t)=\sigma_{l+1}(a_{ij}),
\end{eqnarray}
Theorem \ref{th3.1} is equivalent to say $ \phi(x)\equiv 0$ in $\mathcal {O}\times (t_0-\delta, t_0]$.

Following the notations in \cite{CF85} and \cite{KL87}, if $h$ and
$g$ are two functions defined in $\mathcal {O}\times (t_0-\delta,
t_0]$, we write $h\lesssim g$ if there exist positive constants
$C_1$ and $C_2$ depending only on $||u||_{C^{3,1}},n$ (independent
of $(x,t)$), such that $(h-g)(x,t)\leq(C_1\phi+C_2|\nabla\phi|)(x,t)$, $\forall(x,t)\in
\mathcal {O}\times (t_0-\delta, t_0]$. We also write
$$
h\sim g\quad \text{if}\quad h\lesssim g, \quad g\lesssim
h.
$$

In the following, we will use $i, j, \cdots$ as indices running from $1$
to $n-1$ and use the Greek indices $\alpha, \beta, \cdots$ as
indices running from $1$ to $n$.

\begin{lemma}\label{lem3.2}
For any fixed $(x,t) \in \mathcal {O}\times (t_0-\delta,
t_0]$, with the coordinate system chosen as in \eqref{3.1}, we have
\begin{equation}\label{3.4}
\phi_t \sim -u_n^{-3}\sigma_l(G)\sum_{j \in B}  [u_n^2 u_{jjt}-2u_n u_{jn}u_{jt} ],
\end{equation}
and
\begin{eqnarray}\label{3.5}
\Delta \phi &\sim& -u_n^{-3}\sigma_l(G)\sum_{j\in B}
[u_n^2\Delta u_{jj}-6u_nu_{nj}\Delta u_{j}
+6u_{nj}^2\Delta u]  \notag  \\
&&+2u_n^{-3}\sigma_l(G) \sum_{j\in B,i\in G} \sum_{\alpha=1}^{n}\frac{[u_n u_{ij\alpha}-2u_{i\a}u_{jn}]^2}{u_{ii}}.
\end{eqnarray}
\end{lemma}

\begin{proof}
This proof is similar as in \cite{CS}; for completeness, we give it with some modifications.

For any fixed point $(x,t) \in \mathcal
{O}\times (t_0-\delta, t_0]$, we choose a coordinate system as
in (\ref{3.1}) so that $|\n u|=  u_n >0$ and the matrix
$(a_{ij}(x,t))$ is diagonal for $1 \le i,j\le n-1$ and nonnegative.
From the definition of $\phi$, we can get
\begin{equation}\label{3.6}
a_{ii} \sim 0, ~~\forall i\in B,
\end{equation}
and
\begin{equation}\label{3.7}
h_{ii} \sim 0, \quad u_{ii} \sim 0, ~~\forall i\in B.
\end{equation}

Taking the first derivatives of $\phi$, we get
\begin{eqnarray}
\phi_{\a} &=& \sum_{ij=1}^{n-1}\frac{\partial \sigma_{l+1}(a)}{\partial a_{ij}} a_{ij,\a}  = \sum_{i \in G}  \sigma_l(a | i )
a_{ii,\a}+  \sum_{i \in B}  \sigma_l(a| i ) a_{ii, \a} \notag \\
&\sim&  \sigma_l(G) \sum_{i \in B}  a_{ii, \a} \sim  -u_n^{-3}\sigma_l(G) \sum_{i \in B}  h_{ii,\a}  \notag  \\
&\sim& -u_n^{-3} \sigma_l(G) \sum_{i \in B}  [u_n^2u_{ii\a}-2u_nu_{in}u_{i\a} ], \notag
\end{eqnarray}
so we get
\begin{eqnarray}\label{3.8}
\sum_{i \in B}  a_{ii,\a} \sim  0, \quad  \sum_{i \in B}  h_{ii, \a} \sim 0, \quad  \sum_{i \in B}  [u_n^2u_{ii\a}-2u_nu_{in}u_{i \a} ] \sim 0.
\end{eqnarray}
Hence
\begin{eqnarray}\label{3.9}
 \sum_{i \in B} u_{iij} \sim 0, \quad \forall j \in G.
\end{eqnarray}
Similarly, we get
\begin{eqnarray}\label{3.10}
\phi_t\sim \sigma_l(G)\sum_{j \in B}
a_{jj,t} \sim -u_n^{-3} \sigma_l(G) \sum_{j \in B}  [u_n^2u_{jjt}-2u_nu_{jn}u_{jt} ].
\end{eqnarray}

Using relationship \eqref{3.6} - \eqref{3.8}, we have
\begin{eqnarray}\label{3.11}
\phi_{\alpha \beta}&=&\sum_{i,j=1}^{n-1}\frac{\partial \sigma_{l+1}(a)}{\partial a_{ij}} a_{ij,\a \beta} + \sum_{i,j,k,l=1}^{n-1}\frac{\partial^2 \sigma_{l+1}(a)}{\partial a_{ij}\partial a_{kl}} a_{ij,\a}a_{kl,\b} \notag \\
&=& \sum_{j=1}^{n-1}\sigma_{l}(a|j) a_{jj,\a \beta} + \sum_{i,j=1, i\ne j}^{n-1}\sigma_{l-1}(a|ij) a_{ii,\a}a_{jj,\b} -\sum_{i,j=1, i\ne j}^{n-1}\sigma_{l-1}(a|ij) a_{ij,\a}a_{ji,\b} \notag \\
&\sim& \sum_{j\in B}\sigma_{l}(G) a_{jj,\a \beta} + [\sum_{i,j\in G, i\ne j}+\sum_{i \in G, j\in B}+\sum_{i \in B, j\in G}+\sum_{i,j\in B, i\ne j} ]\sigma_{l-1}(a|ij) a_{ii,\a}a_{jj,\b} \notag\\
&&\qquad-[\sum_{i,j\in G, i\ne j}+\sum_{i \in G, j\in B}+\sum_{i \in B, j\in G}+\sum_{i,j\in B, i\ne j} ]\sigma_{l-1}(a|ij) a_{ij,\a}a_{ji,\b} \notag \\
&\sim& \sum_{j\in B}\sigma_{l}(G) a_{jj,\a \beta}  + \sum_{i,j\in B, i\ne j} \sigma_{l-1}(a|ij) a_{ii,\a}a_{jj,\b}\notag\\
&& -[\sum_{i \in G, j\in B}+\sum_{i \in B, j\in G}+\sum_{i,j\in B, i\ne j} ]\sigma_{l-1}(a|ij) a_{ij,\a}a_{ji,\b} \notag \\
&\sim& \sum_{j\in B}\sigma_{l}(G) a_{jj,\a \beta} -2\sum_{i \in G, j\in B} \sigma_{l-1}(G|i) a_{ij,\a}a_{ji,\b} \notag \\
&=& \sigma_l(G) \sum_{j\in B} \Big[a_{jj,\alpha \beta}-2\sum_{i\in
G}\frac{a_{ij,\alpha}a_{ij,\beta}}{a_{ii}}\Big].
\end{eqnarray}
where we have used the following fact from Lemma \ref{lem2.9},
\begin{align*}
&|a_{ii,\a}a_{jj,\b} | \leq C a_{ii}^{\frac{1}{2}}\cdot C a_{jj}^{\frac{1}{2}} \leq C_1 \phi, \quad  i, j\in B, i\ne j;  \\
&|a_{ij,\a}a_{ji,\b}| \leq C [a_{ii}a_{jj} ] ^{\frac{1}{4}} \cdot C [a_{ii}a_{jj} ] ^{\frac{1}{4}}  \leq C_2 \phi, \quad i,j\in B, i\ne j.
\end{align*}

Since $u_k=0$ for $k=1,\cdots, n-1$, from \eqref{2.5},
\begin{eqnarray}
u_n u_{ij\a}=-u_n^2 a_{ij,\a}+u_{nj}u_{i\a} +u_{ni} u_{j\a}+ u_{n\a} u_{ij}, \quad \forall\; i,j \le n-1, \notag
\end{eqnarray}
and for each $j\in B$,
\begin{eqnarray}
a_{jj,\alpha \alpha} &\sim&-\frac{1}{u_n^3}h_{jj,\alpha \a}- 2 (\frac{|u_n|}{|\nabla u|{u_n}^3})_{\a} h_{jj,\alpha }  \notag  \\
&=&-\frac{1}{u_n^3}[u_n^2u_{jj\alpha\a}-2u_nu_{nj}u_{\a\a
j}+2u_{nn}u_{j\a}^2+4u_{n\a}u_{nj}u_{j\a} -4u_nu_{j\a}u_{nj\a}] \notag \\
&&\quad - 2 (\frac{|u_n|}{|\nabla u|{u_n}^3})_{\a} h_{jj,\alpha }. \notag
\end{eqnarray}
Hence \eqref{3.8} yields
\begin{align}\label{3.12}
\sum_{j \in B} \sum_{\alpha=1}^n a_{jj,\alpha \a} \sim&
-\frac{1}{u_n^3} \sum_{j \in B}  \Big\{u_n^2 \Delta u_{jj}-2u_nu_{nj} \Delta u_{j}+\sum_{\alpha=1}^n
[2u_{nn}u_{j\a}^2+4u_{n\a}u_{nj}u_{j\a} -4u_nu_{j\a}u_{nj\a}]\Big\} \notag \\
& \quad - 2 \sum_{\alpha=1}^n (\frac{|u_n|}{|\nabla u|{u_n}^3})_{\a} \sum_{j \in B}  h_{jj,\alpha } \notag \\
\sim&
-\frac{1}{u_n^3} \sum_{j \in B}  \Big\{u_n^2 \Delta u_{jj}-2u_nu_{nj} \Delta u_{j}+\sum_{\alpha=1}^n
[2u_{nn}u_{j\a}^2+4u_{n\a}u_{nj}u_{j\a} -4u_nu_{j\a}u_{nj\a}]\Big\},
\end{align}
where
\begin{align}
\sum_{\alpha=1}^n[2u_{nn}u_{j\a}^2+4u_{n\a}u_{nj}u_{j\a} -4u_nu_{j\a}u_{nj\a}]
\sim& 2u_{nn}u_{jn}^2 + 4u_{nn}u_{nj}u_{jn} -4u_nu_{jn}u_{njn}\notag \\
=& 6u_{nn}u_{nj}u_{jn} -4u_nu_{jn}u_{nnj}\notag \\
=& 6[\Delta u - \sum_{i=1}^{n-1} u_{ii}]u_{nj}u_{jn} -4u_nu_{jn}[\Delta u_j - \sum_{i=1}^{n-1} u_{iij}]
\notag \\
=& 6u_{nj}^2 \Delta u  -4u_nu_{jn}\Delta u_j  - 6 u_{nj}^2 \sum_{i=1}^{n-1} u_{ii}+ 4u_nu_{jn} \sum_{i=1}^{n-1} u_{iij}\notag \\
\sim& 6u_{nj}^2 \Delta u  -4u_nu_{jn}\Delta u_j  - 6 u_{nj}^2 \sum_{i\in G} u_{ii}+ 4u_nu_{jn} \sum_{i\in G} u_{iij}. \notag
\end{align}
For $i \in G, j\in B$, we have
\begin{align}
\sum\limits_{\alpha  = 1}^n {\frac{{a_{ij,\alpha } ^2 }}{{a_{ii} }}}  =& - \frac{1}{{u_n ^3 }}\sum\limits_{\alpha  = 1}^n {\frac{{[u_n ^2 a_{ij,\alpha } ]^2 }}{{u_{ii} }}}  \sim  - \frac{1}{{u_n ^3 }}\sum\limits_{\alpha  = 1}^n {\frac{{[u_n u_{ij\alpha }  - u_{i\alpha } u_{nj}  - u_{j\alpha } u_{ni} ]^2 }}{{u_{ii} }}}\notag \\
=&  - \frac{1}{{u_n ^3 }}\sum\limits_{\alpha  = 1}^n {\frac{{[u_n u_{ij\alpha }  - 2u_{i\alpha } u_{nj}  + u_{i\alpha } u_{nj}  - u_{j\alpha } u_{ni} ]^2 }}{{u_{ii} }}}\notag \\
=& - \frac{1}{{u_n ^3 }}\sum\limits_{\alpha  = 1}^n {\frac{{[u_n u_{ij\alpha }  - 2u_{i\alpha } u_{nj} ]^2  + 2[u_n u_{ij\alpha }  - 2u_{i\alpha } u_{nj} ][u_{i\alpha } u_{nj}  - u_{j\alpha } u_{ni} ] + [u_{i\alpha } u_{nj}  - u_{j\alpha } u_{ni} ]^2 }}{{u_{ii} }}}
\notag \\
\sim& - \frac{1}{{u_n ^3 }}\left\{ {\sum\limits_{\alpha  = 1}^n {\frac{{[u_n u_{ij\alpha }  - 2u_{i\alpha } u_{nj} ]^2 }}{{u_{ii} }} + 2u_n u_{iij} u_{nj}  - 3u_{ii} u_{nj} ^2 } } \right\}. \notag
\end{align}
So
\begin{align}\label{3.13}
\sum_{j \in B} \sum_{\a=1}^n [a_{jj,\alpha \alpha}-2\sum_{i\in
G}\frac{a_{ij,\alpha}^2}{a_{ii}}] \sim& -u_n^{-3}
\sum_{j \in B}[u_n^2\Delta u_{jj}-6u_nu_{nj}\Delta u_{j}
+6u_{nj}^2\Delta u]  \notag  \\
&+2u_n^{-3}\sum_{j \in B,i \in G}\sum_{\alpha=1}^{n}\frac{[u_nu_{ij\alpha}-2u_{i\a}u_{jn}]^2}{u_{ii}}.
\end{align}

From \eqref{3.11} and \eqref{3.13}, Lemma \ref{lem3.2} holds.
\end{proof}

\subsection{ Proof of the constant rank theorem for the spatial second fundamental form}

Theorem \ref{th3.1} is a direct consequence of the following
proposition and the strong maximum principle.

\begin{proposition}\label{prop3.3}
Suppose that the function $u$ satisfies the assumptions of Theorem
\ref{th3.1}. If the second fundamental form $II_{\partial\Sigma_x^{c,t}}$ of the spatial level sets
$\partial\Sigma_x^{c,t}=\{x \in \Omega | u(x,t) =
c\}$ attains minimum rank $l$ at a point
$(x_0,t_0)\in \Omega \times (0, T]$, then there exist a neighborhood $\mathcal {O}
\times(t_0-\delta,t_0]$ of $(x_0,t_0)$ and a positive
constant $C$, independent of $\phi$, such
that
\begin{equation}\label{3.14}
 \Delta \phi(x,t)-\phi_t\leq
C(\phi+|\nabla \phi|),~~\forall ~(x,t)\in \mathcal {O}
\times(t_0-\delta,t_0].
\end{equation}
\end{proposition}

\begin{proof}  Let $u\in C^{4,3}(\Omega \times [0,T])$ be a space-time quasiconcave solution of equation
(\ref{1.1}) and $(u_{ij}) \in \mathcal {S}^n.$ Let $l$ be the minimum
rank of the second fundamental forms $II_{\partial\Sigma_x^{c,t}}$ of $\partial\Sigma_x^{c,t}$
($l \in \{0,1,\cdots,n-1\}$) for every $c$, and suppose the minimum rank $l$ is attained at $(x_0, t_0)
\in \Omega \times (0, T]$. We work in $\mathcal
{O} \times(t_0-\delta,t_0]$ of $(x_0,t_0)$, as usual.  Obviously $ \phi(x,t)\geq 0$
and $\phi(x_0,t_0)=0 $. For each fixed $(x,t)$, choose as usual a
local coordinate $e_1,\cdots, e_{n-1}, e_n$ such that (\ref{3.1}) is satisfied. We want to establish differential
inequality \eqref{3.14} for $\phi$.

By Lemma~\ref{lem3.2} and $u_t= \Delta u$,
\begin{eqnarray}\label{3.15}
\Delta \phi(x,t)-\phi_t &=& -u_n^{-3}\sigma_l(G)\sum_{j\in B}
[-4u_nu_{nj} u_{tj}+6u_{nj}^2 u_t]  \notag  \\
&&+2 u_n^{-3}\sigma_l(G) \sum_{j\in B,i\in G} \sum_{\alpha=1}^{n}\frac{[u_nu_{ij\alpha}-2u_{i\a}u_{jn}]^2}{u_{ii}} \\
&& +O(\phi + |\nabla \phi|).\notag
\end{eqnarray}
Since $u_{ii}=-u_n a_{ii} <0$ for $i \in G$, and for $j \in B$
\begin{eqnarray}\label{3.16}
-4u_nu_{nj} u_{tj}+6u_{nj}^2 u_t &=& \frac{1}{u_t}[6 (u_tu_{nj})^2-4 (u_n u_{tj})(u_t u_{nj})] \notag \\
&=& \frac{1}{u_t}[6 (u_tu_{nj})^2-4 (u_t u_{nj}- \frac{\hat h_{jn}}{u_t})(u_t u_{nj})] \notag \\
&=& \frac{1}{u_t}[2 (u_tu_{nj})^2+4\frac{\hat h_{jn}}{u_t}(u_t u_{nj})] \notag \\
&=& \frac{2 }{u_t}[u_t u_{nj}+\frac{\hat h_{jn}}{u_t}]^2 -2\frac{{\hat h_{jn}}^2}{u_t^3}.\notag \\
\end{eqnarray}
By  \eqref{2.11},  \eqref{2.14} and  \eqref{3.7}, for $j \in B$, we know
\begin{eqnarray}\label{3.16a}
\hat a_{jj}=-\frac{\hat h_{jj}}{|Du|u_t^2}=-\frac{u_{jj}}{|Du|}=O(\phi).
\end{eqnarray}
Now we use the key assumption in our  Theorem
\ref{th3.1} that the solution is {\bf space-time quasiconcave},  then for $j \in B$ we get
\begin{eqnarray}\label{3.16aa}
\hat a_{jn}^2\le \hat a_{jj}\hat a_{nn}=O(\phi),\quad \text{and}\quad  {\hat h_{jn}}^2=O(\phi).
\end{eqnarray}
From  \eqref{3.16} and  \eqref{3.16aa}, we obtain
\begin{eqnarray}\label{3.16aaa}
-4u_nu_{nj} u_{tj}+6u_{nj}^2 u_t= \frac{2 }{u_t}[u_t u_{nj}+\frac{\hat h_{jn}}{u_t}]^2 +O(\phi).
\end{eqnarray}

Now we can get
\begin{eqnarray}\label{3.17}
\Delta \phi(x,t)-\phi_t &=& -2u_n^{-3}\sum_{j\in B} \sigma_l(G)
\frac{1}{u_t}[u_t u_{nj}+\frac{\hat h_{jn}}{u_t}]^2  \notag  \\
&&+2u_n^{-3}\sum_{j\in B,i\in G}\sigma_l(G)\sum_{\alpha=1}^{n}\frac{[u_nu_{ij\alpha}-2u_{i\a}u_{jn}]^2}{u_{ii}} \notag \\
&& +O(\phi + |\nabla \phi|)\,.
\end{eqnarray}
This yields \eqref{3.14} and the proof is complete.
\end{proof}

\subsection{Some consequences of  Theorem~\ref{th3.1}}

In the above proof,
for \eqref{3.17} and the strong maximum principle, we get
\begin{align}
\phi =0  \quad \text{ for } (x,t) \in \mathcal {O} \times (t_0-\delta,t_0]. \notag
\end{align}
Then for any $(x,t)\in \mathcal {O}
\times(t_0-\delta_0,t_0]$ it must hold
\begin{align}\label{3.19}
a_{ii}=0, \text{ for } i \in B.
\end{align}

In fact, we can obtain more precise information as follows.
\begin{corollary} \label{cor3.4}
For any $(x,t)\in \mathcal {O} \times(t_0-\delta_0,t_0]$ with the suitable coordinate system \eqref{3.1}, we have
\begin{align}
\label{3.20}&u_{ii}=0, \text{ for } i \in B,  \\
\label{3.21}&u_{ni}=0,   \text{ for } i \in B, \\
\label{3.22}&u_{ij\alpha}=0,  \text{ for } i\in B, j\in G, \alpha =1, \cdots, n.
\end{align}
\end{corollary}

\begin{proof}
By \eqref{2.5} and
\eqref{3.19} we have
\begin{align*}
u_{ii}=0, \text{ for } i \in B,
\end{align*}
Then from \eqref{2.9}, \eqref{2.11} and \eqref{2.12},
\begin{align*}
&\hat h_{ii}=0, \quad \hat a_{ii} =0, \text{ for } i \in B.
\end{align*}
From $\hat a\geq0$,
\begin{align*}
&\hat a_{in} =0, \quad \hat h_{in}=0, \text{ for } i \in B.
\end{align*}
Equation \eqref{3.17} yields
\begin{align*}
&u_t u_{ni}+\frac{\hat h_{in}}{u_t}=0,   \text{ for } i \in B, \\
&u_nu_{ij\alpha}-2u_{j\a}u_{in}=0,  \text{ for } i\in B, j\in G, \alpha =1, \cdots, n,
\end{align*}
so it holds
\begin{align*}
&u_{ni}=0,   \text{ for } i \in B, \\
&u_{ij\alpha}=0,  \text{ for } i\in B, j\in G, \alpha =1, \cdots, n.
\end{align*}
So the proof of Corollary \ref{cor3.4} is complete.
\end{proof}

\begin{remark}
In this section we got a constant rank theorem for the second fundamental form of the spatial level sets of a space-time
quasiconcave solution to heat equation \eqref{1.3}. But if we delete the condition {\bf space-time
quasiconcave solution} in above  Theorem~ \ref{th3.1},  then we couldn't obtain the constant rank theorem for the second fundamental form of the spatial convex level sets of the solution to heat equation \eqref{1.3}. In fact we need another structure condition on the parabolic partial differential equation to get this property, for example it is true for the equation
$$\frac{{\partial u}} {{\partial t}} =\frac{\Delta u}{|\nabla u|^2}. $$
Please see the detail in Chen-Shi \cite{CS} or  Ishige-Salani \cite{IS10}.
\end{remark}

\section{A constant rank theorem for the space-time second fundamental form: CASE 1}

In this section, we begin the proof of Theorem~\ref{th1.3} .

 We start to consider the space-time level sets $\partial\Sigma_{x,t} ^{c} = \{(x,t) \in
\Omega \times [0, T]|u(x,t) = c\}$, and as in Section 2.1.3, the Weingarten tensor of $\partial\Sigma_{x,t} ^{c} $ is
\begin{equation}\label{4.1}
\hat a_{\a \b} =-\frac{|u_t|}{|D u|{u_t}^3} \hat A_{\a \b}, \quad  1 \leq \a, \b \leq n,
\end{equation}
where $A_{\a \b}$ defined as in \eqref{2.12}, and for explicit formulae see from \eqref{2.14} to \eqref{2.17}.

Suppose $\hat{a}(x,t)=(\hat a_{\a \b})_{n \times n}$ attains the minimal
rank $l$ at some point $(x_0, t_0) \in \Omega \times (0, T]$. We may assume
$l\leqslant n-1$, otherwise there is nothing to prove. At $(x_0,t_0)$, we may choose $e_1,\cdots, e_{n-1},e_n$ such that
\begin{equation}\label{4.6}
u_n(x_0,t_0) = |\n u(x_0,t_0)|>0,\ \mbox{ and }
(u_{ij})_{1 \leq i,j \leq n-1} \mbox{is diagonal at}\ (x_0,t_0).
\end{equation}
Without loss of generality we assume $ u_{11} \leq u_{22}\leq \cdots
\leq u_{n-1n-1} $. So, at $(x_0,t_0)$, from \eqref{4.1}, we have the matrix
$(\hat a_{ij})_{1 \leq i,j \leq n-1}$ is also diagonal, and $\hat a_{11} \geq \hat a_{22} \geq \cdots \geq \hat a_{n-1
n-1}$. From Lemma~\ref{lem2.8}, we get at $(x_0,t_0)$, there is a positive constant $
C_0$ such that

{\bf CASE 1:}
\begin{eqnarray*}
&&\hat a_{11}  \geq \cdots \geq \hat a_{l-1l-1} \geq
C_0 , \quad \hat a_{ll} = \cdots = \hat a_{n-1n-1} =0 , \\
&&\hat a_{nn} -\sum\limits_{i = 1}^{l-1} {\frac{{\hat a_{in} ^2 }} {{\hat a_{ii}
}}} \geq C_0 ,  \quad \hat a_{in} = 0, \quad l \leqslant i \leqslant n-1.
\end{eqnarray*}

{\bf CASE 2:}
\begin{eqnarray*}
&&\hat a_{11}  \geq \cdots \geq \hat a_{ll} \geq C_0 , \quad \hat a_{l+1l+1} =
\cdots = \hat a_{n-1n-1} =0, \\
&& \hat a_{nn}  = \sum\limits_{i = 1}^{l}{\frac{{\hat a_{in} ^2 }} {{\hat a_{ii} }}}
,  \quad \hat a_{in} = 0, \quad  l+1 \leqslant i \leqslant n-1.
\end{eqnarray*}

In this section, we consider CASE 1.

There is a neighborhood $\mathcal {O}\times
(t_0-\delta, t_0]$ of $(x_0, t_0)$, such that for any fixed point  $(x,t) \in \mathcal {O}\times (t_0-\delta,
t_0]$, we may choose $e_1,\cdots, e_{n-1},e_n$ such that
\begin{equation}\label{4.7}
u_n(x,t) = |\n u(x,t)|>0\ \mbox{ and }
(u_{ij})_{1 \leq i,j \leq n-1} \mbox{is diagonal at}\ (x,t).
\end{equation}
Similarly we assume $ u_{11} \leq u_{22}\leq \cdots
\leq u_{n-1n-1} $. So, at $(x,t) \in \mathcal {O}\times (t_0-\delta,
t_0]$, from \eqref{4.1}, we have the matrix
$(\hat a_{ij})_{1 \leq i,j \leq n-1}$ is also diagonal, and $\hat a_{11} \geq \hat a_{22} \geq \cdots \geq \hat a_{n-1
n-1}$. There is a positive constant $C_0>0$ depending only on
$\|u\|_{C^{4}}$ and $\mathcal {O}\times (t_0-\delta, t_0]$,
such that
\begin{align*}
\hat a_{11}  \geq \cdots \geq \hat a_{l-1l-1} \geq
C_0 ,  \quad \hat a_{nn} -\sum\limits_{i = 1}^{l-1} {\frac{{\hat a_{in} ^2 }} {{\hat a_{ii}
}}} \geq C_0.
\end{align*}
for $(x,t) \in \mathcal {O}\times (t_0-\delta, t_0]$. For convenience we
denote $ G = \{ 1, \cdots ,l-1 \} $ and $ B = \{
l, \cdots, n-1 \} $ be the ``good" and ``bad" sets of indices
respectively.

Since
\begin{equation}\label{4.8}
\hat a_{ij}= \frac{|\n u|}{|D u|} a_{ij}, \quad 1 \leq i, j \leq n-1,
\end{equation}
there is a positive constant $C>0$ depending only on
$\|u\|_{C^{4}}$ and $\mathcal {O}\times (t_0-\delta, t_0]$,
such that
\begin{equation}\label{4.9}
a_{11}  \geq \cdots \geq  a_{l-1l-1} \geq C,\quad (x,t) \in \mathcal {O}\times (t_0-\delta, t_0],
\end{equation}
and
\begin{equation}\label{4.10}
a_{ll}(x_0,t_0) = \cdots = a_{n-1n-1}(x_0,t_0) = 0.
\end{equation}
So the spatial
second fundamental form $a=(a_{ij})_{n-1 \times n-1}$ attains the minimal rank $l-1$ at $(x_0,t_0)$.
From Theorem $\ref{th3.1}$, the constant rank theorem holds for the spatial
second fundamental form $a$. So we can get $a_{ii}=0, \forall i \in B$ for any $(x,t) \in \mathcal {O} \times (t_0-\delta, t_0]$.
Furthermore,
\begin{equation}\label{4.11}
\hat a_{ii}=0, ~\forall i \in B.
\end{equation}
We denote $M=(\hat a_{ij})_{n-1 \times n-1}$,
so
\begin{equation} \label{4.12}
\sigma_{l+1}(M)=\sigma_{l}(M) \equiv 0, \quad \text{for
every}\quad (x,t) \in \mathcal {O} \times (t_0-\delta, t_0].
\end{equation}
Then we have from \eqref{2.18}
\begin{equation} \label{4.13}
0 \leq \sigma_{l+1}(\hat a)\leq \sigma_{l+1}(M)+\hat a_{nn}\sigma_l(M) = 0.
\end{equation}
So \begin{equation} \label{4.14}
\sigma_{l+1}(\hat a)\equiv 0, \quad \text{for every}\quad
(x,t) \in \mathcal {O} \times (t_0-\delta, t_0].
\end{equation}

By the continuity method, Theorem \ref{th1.3} holds under the CASE 1.

\section{A constant rank theorem for the space-time second fundamental form: CASE 2 }
\setcounter{equation}{0} \setcounter{theorem}{0}

Now we consider {\bf CASE 2}. For completeness we write out the Weingarten
tensor of the space-time level sets $\partial\Sigma_{x,t} ^{c} = \{(x,t) \in
\Omega \times [0, T]| u(x,t) = c\}$ as
\begin{equation}\label{6.1}
\hat a_{\a \b} =-\frac{|u_t|}{|D u|{u_t}^3} \hat A_{\a \b}, \quad  1 \leq \a, \b \leq n,
\end{equation}
where
\begin{align}
\label{6.2} \hat A_{ij} =& \hat h_{ij}-\frac{u_iu_n \hat h_{jn}}{\hat W(1+\hat W)u_t^2} -\frac{u_ju_n \hat h_{in}}{\hat W(1+ \hat W)u_t^2}  \\
&-\frac{u_i \sum_{l=1}^{n-1}u_l \hat h_{jl}}{\hat W(1+\hat W)u_t^2} -\frac{u_j \sum_{l=1}^{n-1} u_l \hat h_{il}}{\hat W(1+ \hat W)u_t^2}
+ \frac{u_iu_ju_n^2 \hat h_{nn}}{\hat W^2(1+\hat W)^2u_t^4} +T_{ij},  \quad 1 \leq i,j \leq n-1, \notag \\
\label{6.3} \hat A_{in} =&  \frac{1}{\hat W}\hat h_{in}-\frac{u_iu_n \hat h_{nn}}{\hat W^2(1+\hat W)u_t^2} -\frac{u_n \sum_{l=1}^{n-1}u_l
\hat h_{il}}{\hat W(1+\hat W)u_t^2}  \\
&+\frac{\hat h_{in} \sum_{l=1}^{n-1} u_l^2}{\hat W(1+\hat W)u_t^2}
+ \frac{u_i \sum_{l=1}^{n-1}u_l\hat h_{nl}}{\hat W(1+\hat W)u_t^2}[1- \frac{2}{\hat W}] +T_{in},  \quad 1 \leq i \leq n-1, \notag \\
\label{6.4} \hat A_{nn} =&  \frac{1}{\hat W^2} \hat h_{nn}-2\frac{u_n\sum_{l=1}^{n-1}u_l\hat h_{nl}}{\hat W^2(1+\hat W)u_t^2}  \\
&+2\frac{\sum_{l=1}^{n-1} u_l^2\hat h_{nn}}{\hat W^2(1+\hat W)u_t^2}
+ \frac{ \sum_{k,l=1}^{n-1} u_ku_l \hat h_{kl}}{\hat W(1+\hat W)u_t^2}[1- \frac{1}{\hat W}] +T_{nn},\notag
\end{align}
and
\begin{equation}\label{6.5}
\hat h_{\a \b} =u_t^2 u_{\a \b} + u_{tt}u_{\a}u_\b - u_t u_\b u_{\a t} - u_t u_\a u_{\b t}, \quad  1 \leq \a,\b \leq n.
\end{equation}
When we choose the coordinates such that  $u_i=0$ ($1 \leq i \leq n-1$) at some point $(x_0, t_0)$,
$T_{\a \b}$ ($1 \leq \a, \b \leq n$) satisfies
\begin{align}\label{6.6}
 T_{\a \b} =0, \quad D T_{\a \b} =0, \quad D^2 T_{\a \b} =0, \quad 1 \leq \a, \b \leq n.
\end{align}

From Lemma \ref{lem2.8}, $\hat{a}(x,t)=(\hat a_{\a \b})_{n \times n}$ attains the minimal
rank $l$ at some point $(x_0, t_0) \in \Omega \times (0, T]$ and at $(x_0,t_0)$, we may choose $e_1,\cdots, e_{n-1},e_n$ such that
\begin{equation}
u_n = |\n u|>0,\ \mbox{ and }
(u_{ij})_{1 \leq i,j \leq n-1} \mbox{is diagonal at}\ (x_0,t_0). \notag
\end{equation}
Then we have
\begin{eqnarray*}
&&\hat a_{11}  \geq \cdots \geq \hat a_{ll} \geq C_0 , \quad \hat a_{l+1l+1} =
\cdots = \hat a_{n-1n-1} =0, \\
&& \hat a_{nn}  = \sum\limits_{i = 1}^{l}{\frac{{\hat a_{in} ^2 }} {{\hat a_{ii} }}}
,  \quad \hat a_{in} = 0, \quad  l+1 \leqslant i \leqslant n-1.
\end{eqnarray*}
Then there is a neighborhood $\mathcal {O}\times
(t_0-\delta, t_0]$ of $(x_0, t_0)$, such that for any fixed point  $(x,t) \in \mathcal {O}\times (t_0-\delta,
t_0]$, we may choose $e_1,\cdots, e_{n-1},e_n$ such that
\begin{equation}\label{6.7}
u_n(x,t)= |\n u(x,t)|>0,\ \mbox{ and }
(u_{ij})_{1 \leq i,j \leq n-1} \mbox{is diagonal at}\ (x,t).
\end{equation}
Similarly we assume $ u_{11} \leq u_{22}\leq \cdots
\leq u_{n-1n-1} $. So, at $(x,t) \in \mathcal {O}\times (t_0-\delta,
t_0]$, from \eqref{6.1} and \eqref{6.2}, we have the matrix
$(\hat a_{ij})_{1 \leq i,j \leq n-1}$ is also diagonal, and $\hat a_{11} \geq \hat a_{22} \geq \cdots \geq \hat a_{n-1
n-1}$. There is a positive constant $C>0$ depending only on
$\|u\|_{C^{4}}$ and $\mathcal {O}\times (t_0-\delta, t_0]$,
such that
\begin{align*}
\hat a_{11}  \geq \cdots \geq \hat a_{ll} \geq C,
\end{align*}
for all $(x,t) \in \mathcal {O}\times (t_0-\delta, t_0]$. For convenience we
denote $ G = \{ 1, \cdots ,l \} $ and $ B = \{
l+1, \cdots, n-1 \} $ be the ``good" and ``bad" sets of indices
respectively. Without confusion we will also simply denote $
G = \{ \hat a_{11} , \cdots , \hat a_{ll} \} $ and $ B= \{\hat a_{l+1 l+1} , \cdots ,\hat a_{n-1n-1} \} $.

We shall divide this part into three steps. In step 1 we use Theorem~\ref{th3.1} to perform a reduction of the proof. Step 2 starts from Lemma \ref{lem6.5}, which is the reduction for the second derivative of $\phi$ via step 1. In step 3, we shall complete the proof of Theorem~\ref{th1.3}
with the help of Theorem \ref{th6.7}.

\subsection{Step 1: reduction using Theorem \ref{th3.1}}
From Theorem $\ref{th3.1}$, the constant rank theorem holds for the spatial
second fundamental form $a$ with the minimal rank $l$. So for any $(x,t) \in \mathcal {O} \times (t_0-\delta, t_0]$ with the coordinate \eqref{6.7}, we can get $a_{ii}=0, \forall i \in B$. Furthermore, $\hat a_{ii}=0, \forall i \in B$.

Under above assumptions, we can get
\begin{proposition}\label{prop6.1}
For any $(x,t) \in \mathcal {O} \times (t_0-\delta, t_0]$ with the coordinate \eqref{6.7}, we have
\begin{align}\label{6.8}
\hat a_{ii}(x,t) \equiv 0, \quad i \in B.
\end{align}
Furthermore, we have at $(x,t)$
\begin{eqnarray}
\label{6.9} &&\hat a_{ij}(x,t) \equiv 0, \quad  \text{ i or j } \in B, \\
\label{6.10} && \hat a_{in}(x,t)=\hat a_{ni}(x,t) \equiv 0, \quad i \in B, \\
\label{6.11} && {D \hat a_{ij} }(x,t)= {(\nabla \hat a_{ij}, \hat a_{ij,t}) }(x,t) \equiv 0, \quad  \text{ i or j } \in B,\\
\label{6.12} && {D \hat a_{in} }(x,t)= {D \hat a_{ni} }(x,t) \equiv 0, \quad  i \in B.
\end{eqnarray}
\end{proposition}

\begin{proof}
The proof is directly from the constant rank theorem of $a$ and Lemma \ref{lem2.9} (i.e. Remark \ref{rem2.10}).

For any $(x,t) \in \mathcal {O} \times (t_0-\delta, t_0]$ with the coordinate \eqref{6.7}, we can get
\begin{equation*}
\hat a_{ij}= \frac{|\n u|}{|D u|} a_{ij}, \quad 1 \leq i, j \leq n-1,
\end{equation*}
then there is a positive constant $C>0$ depending only on
$\|u\|_{C^{4}}$ and $\mathcal {O}\times (t_0-\delta, t_0]$,
such that
\begin{equation*}
a_{11}  \geq \cdots \geq  a_{ll} \geq C,\quad (x,t) \in \mathcal {O}\times (t_0-\delta, t_0],
\end{equation*}
and
\begin{equation*}
a_{l+1l+1}(x_0,t_0) = \cdots = a_{n-1n-1}(x_0,t_0) = 0.
\end{equation*}
So the spatial second fundamental form $a=(a_{ij})_{n-1 \times n-1}$ attains the minimal rank $l$ at $(x_0,t_0)$.
From Theorem $\ref{th3.1}$, the constant rank theorem holds for the spatial
second fundamental form $a$. Then we can get for any $(x,t) \in \mathcal {O} \times (t_0-\delta, t_0]$ with the coordinate  \eqref{6.7} (that is \eqref{3.1}),
\begin{equation*}
a_{ii}(x,t)=0, \quad \forall i \in B .
\end{equation*}
Furthermore, we have
\begin{equation*}
\hat a_{ii}(x,t)=\frac{|\n u|}{|D u|} a_{ii}=0, \quad \forall i \in B.
\end{equation*}
From the positive definite of $\hat a$ at $(x,t)$, we get
\begin{equation*}
\hat a_{in}(x,t)=0, \quad \forall i \in B.
\end{equation*}
And from Remark \ref{rem2.10}, we get
\begin{equation*}
 |D \hat a_{ij}|(x,t)=|D \hat a_{in}|(x,t)=0, \quad \forall i \in B.
\end{equation*}
So the lemma holds.
\end{proof}

We denote $M=(\hat a_{ij})_{n-1 \times n-1}$, and set
\begin{eqnarray}\label{6.13}
\phi = \sigma_{l+1}(\hat a).
\end{eqnarray}
Following the notation in \cite{CF85} and \cite{KL87}, let $h$ and
$g$ be two functions defined in $\mathcal {O}\times (t_0-\delta,
t_0]$, we say $h\lesssim g$ if there exist positive constants
$C_1$ and $C_2$ depending only on $||u||_{C^{3,1}},n$ (independent
of $(x,t)$), such that $(h-g)(x,t)\leq(C_1\phi+C_2|\nabla\phi|)(x,t)$, $\forall(x,t)\in
\mathcal {O}\times (t_0-\delta, t_0]$. We also write
$$
h\sim g\quad \text{if}\quad h\lesssim g, \quad g\lesssim
h.
$$

\begin{lemma} \label{lem6.2}
Under the above assumptions and notations, for any $(x,t) \in \mathcal {O} \times (t_0-\delta, t_0]$ with the coordinate \eqref{6.7}, we have
\begin{align}
\label{6.14} &u_{ii}  = 0,\quad i \in B; \qquad  u_{ii}  = \frac{{\hat h_{ii} }}{{u_t ^2 }}, \quad i \in G;  \\
\label{6.15} &u_{ij}  = 0,\quad i \in B \cup G, j \in B \cup G, i \ne j;  \\
\label{6.16} &u_{in}  =0, \quad i \in B; \qquad u_{it}=0, \quad i \in B; \\
\label{6.17}  &u_t ^2 u_{in}  = \hat h_{in}  + u_n u_t u_{it},\quad i \in G;  \\
\label{6.18} &u_{nn} = u_t - \frac{{ 1}}{{u_t ^2 }}\sum\limits_{i \in G} {\hat h_{ii}};  \\
\label{6.19} &u_n ^2 u_{tt} \sim \sum\limits_{i \in G} {\frac{{\hat h_{in} ^2 }}{{\hat h_{ii} }}}
 + \sum\limits_{i \in G} {\hat h_{ii} }  - u_t ^3  + 2u_n u_t u_{nt}.
\end{align}

\end{lemma}

\begin{proof} Under the above assumptions, we need to do some routine
calculations for the derivatives of $\phi$. In the following, \eqref{6.7} can be used all the time.

Since $M$ is diagonal and $\hat a_{ii}=0$ for $i \in B$, we can get from Lemma \ref{lem2.5},
\begin{align}
\phi=&\sigma_{l+1}(M)+ \hat a_{nn}\sigma_l(M)
-\sum_{i}\hat a_{ni} \hat a_{in}\sigma_{l-1}(M|i) \notag \\
=&\sigma_l(G)[\hat a_{nn} - \sum\limits_{i \in G}{\frac{{\hat a_{in} ^2 }} {{\hat a_{ii} }}}], \notag
\end{align}
so
\begin{equation}\label{6.20}
\hat a_{nn} - \sum\limits_{i \in G}{\frac{{\hat a_{in} ^2 }} {{\hat a_{ii}}}}\sim 0.
\end{equation}

By \eqref{6.1}, we have
\begin{align}\label{6.21}
\hat A_{nn} - \sum\limits_{i \in G}{\frac{{\hat A_{in} ^2 }} {{\hat A_{ii} }}}
\sim 0.
\end{align}
Since $u_n= |\nabla u| > 0$ by \eqref{6.7}, $u_i=0, i=1, \cdots, n-1$, then we get
\begin{align}\label{6.22}
\hat A_{ij} =\hat h_{ij}, \quad \hat A_{in} = \frac{1}{\hat W}\hat h_{in}, \quad
\hat A_{nn} = \frac{1}{\hat W^2} \hat h_{nn},
\end{align}
so from \eqref{6.21}, we can get
\begin{align} \label{6.23}
\hat h_{nn} - \sum\limits_{i \in G}{\frac{{\hat h_{in} ^2 }} {{\hat h_{ii} }}}
\sim 0.
\end{align}

From \eqref{6.1}, \eqref{6.21} and Proposition \ref{prop6.1}, we have
\begin{align*}
&0=a_{ij} = -\frac{|u_t|}{|D u|{u_t}^3}\hat A_{ij}= -\frac{|u_t|}{|D u|{u_t}^3}\hat h_{ij}, \quad  i \in B, j\in B \cup G, \\
&0=a_{in} = -\frac{|u_t|}{|D u|{u_t}^3}\hat A_{in}= -\frac{|u_t|}{|D u|{u_t}^3}\frac{1}{\hat W}\hat h_{in}, \quad i \in B,
\end{align*}
so we get
\begin{align} \label{6.24}
\hat h_{ij}= u_t ^2 u_{ij}=0,\quad i \in B, j\in B \cup G; \qquad \hat h_{in}=u_t ^2 u_{in}-u_n u_t u_{it}=0, \quad i \in B.
\end{align}
From \eqref{6.5}, \eqref{6.7} and  \eqref{6.24}, we get
\begin{align*}
&u_{ii}  = 0,\quad i \in B; \qquad \qquad\quad\quad u_{ii}  = \frac{{\hat h_{ii} }}{{u_t ^2 }}, \quad i \in G;  \\
&u_{ij}  = 0,\quad i \in B \cup G, j \in B \cup G, i \ne j; \\
&u_{in}  =\frac{ u_n} { u_t} u_{it}, \quad i \in B; \qquad \qquad \quad u_t ^2 u_{in}  = \hat h_{in}  + u_n u_t u_{it},\quad i \in G;  \\
&u_{nn} =u_t  - \sum\limits_{i = 1}^{n - 1} {u_{ii} }  = u_t  - \sum\limits_{i \in G} {\frac{{\hat h_{ii} }}{{u_t ^2 }}};  \\
&u_n ^2 u_{tt} = \hat h_{nn}  - u_t ^2 u_{nn}  + 2u_n u_t u_{nt} \sim \sum\limits_{i \in G} {\frac{{\hat h_{in} ^2 }}{{\hat h_{ii} }}}
 + \sum\limits_{i \in G} {\hat h_{ii} }  - u_t ^3 + 2u_n u_t u_{nt}.
\end{align*}
By Corollary \ref{cor3.4}, $u_{in}=u_{it}=0$ for $i \in B$, so we can get \eqref{6.14} - \eqref{6.19}.
The lemma holds.
\end{proof}

\begin{lemma}\label{lem6.3}
Under the above assumptions and notations, for any $(x,t) \in \mathcal {O} \times (t_0-\delta, t_0]$ with the coordinate \eqref{6.7}, we have
\begin{eqnarray}
 \label{6.25} &\hat a_{nn,\alpha }  - 2\sum\limits_{i \in G}
{\frac{{\hat a_{in} }} {{\hat a_{ii} }}\hat a_{in,\alpha } + \sum\limits_{i,j \in G}
{\frac{{\hat a_{in} }} {{\hat a_{ii} }}\frac{{\hat a_{jn} }} {{\hat a_{jj}
}}\hat a_{ij,\alpha } } }   \sim 0, \quad  \a =1, \cdots, n, \\
 \label{6.26} &\hat A_{nn,\alpha }  - 2\sum\limits_{i \in G}
{\frac{{\hat A_{in} }} {{\hat A_{ii} }}\hat A_{in,\alpha } + \sum\limits_{i,j \in G}
{\frac{{\hat A_{in} }} {{\hat A_{ii} }}\frac{{\hat A_{jn} }} {{\hat A_{jj}
}}\hat A_{ij,\alpha } } }   \sim 0, \quad  \a =1, \cdots, n, \\
 \label{6.27} &\hat h_{nn,\alpha }  - 2\sum\limits_{i \in G}
{\frac{{\hat h_{in} }} {{\hat h_{ii} }}\hat h_{in,\alpha } + \sum\limits_{i,j \in G}
{\frac{{\hat h_{in} }} {{\hat h_{ii} }}\frac{{\hat h_{jn} }} {{\hat h_{jj}
}}\hat h_{ij,\alpha } } }   \sim 0, \quad  \a =1, \cdots, n,
\end{eqnarray}
and
\begin{align}  \label{6.28}
\phi_t \sim& \sigma_l(G)[\hat a_{nn,t }  - 2\sum\limits_{i \in G}
{\frac{{\hat a_{in} }} {{\hat a_{ii} }}\hat a_{in,t } + \sum\limits_{i,j \in G}
{\frac{{\hat a_{in} }} {{\hat a_{ii} }}\frac{{\hat a_{jn} }} {{\hat a_{jj}
}}\hat a_{ij,t } } } ]  \notag \\
\sim& \sigma_l(G)(-\frac{|u_t|}{|D u|{u_t}^3}) [\hat A_{nn,t}  - 2\sum\limits_{i \in G}
{\frac{{\hat A_{in} }} {{\hat A_{ii} }}\hat A_{in,t } + \sum\limits_{i,j \in G}
{\frac{{\hat A_{in} }} {{\hat A_{ii} }}\frac{{\hat A_{jn} }} {{\hat A_{jj}
}}\hat A_{ij,t} } } ]\notag \\
\sim& \sigma_l(G) (-\frac{|u_t|}{|D u|{u_t}^3})\frac{1}{\hat W^2} [\hat h_{nn,t }  - 2\sum\limits_{i \in G}
{\frac{{\hat h_{in} }} {{\hat h_{ii} }}\hat h_{in,t} + \sum\limits_{i,j \in G}
{\frac{{\hat h_{in} }} {{\hat h_{ii} }}\frac{{\hat h_{jn} }} {{\hat h_{jj}
}}\hat h_{ij,t} } }].
\end{align}

\end{lemma}

\begin{proof} Taking the first
derivatives of $\phi$ with respect to $t$,  we have from Lemma \ref{lem2.5}
\begin{align}
\phi_t=&\sum_i\sigma_l(M|i)\hat a_{ii,t}+\hat a_{nn,t}\sigma_l(M)+ \hat a_{nn}\sum_i\sigma_{l-1}(M|i)\hat a_{ii,t}\notag\\
&-2\sum_i\sigma_{l-1}(M|i)\hat a_{ni}\hat a_{ni,t}-\sum_{i\neq j}\sigma_{l-2}(M|ij)\hat a_{ni}^2\hat a_{jj,t}+\sum_{i\neq j}\sigma_{l-2}(M|ij)\hat a_{ni}\hat a_{nj}\hat a_{ij,t}\notag\\
=&\sigma_{l}(G)\hat a_{nn,t}+\hat a_{nn}\sum_{i\in
G}\sigma_{l-1}(G|i)\hat a_{ii,t}\notag\\
&-2\sum_{i\in G}\sigma_{l-1}(G|i)\hat a_{ni}\hat a_{ni,t}-\sum_{\stackrel{i,j\in G}{i\neq j}}\sigma_{l-2}(G|ij)\hat a_{ni}^2\hat a_{jj,t}+\sum_{\stackrel{i,j\in G}{i\neq j}}\sigma_{l-2}(G|ij)\hat a_{ni}\hat a_{nj}\hat a_{ij,t},\notag
\end{align}
by \eqref{6.9}-\eqref{6.12} and \eqref{6.20}, we get
\begin{align}
&\hat a_{nn}\sum_{i\in G}\sigma_{l-1}(G|i)\hat a_{ii,t}-\sum_{\stackrel{i,j\in G}{i\neq j}}\sigma_{l-2}(G|ij)\hat a_{ni}^2\hat a_{jj,t}+\sum_{\stackrel{i,j\in G}{i\neq j}}\sigma_{l-2}(G|ij)\hat a_{ni}\hat a_{nj}\hat a_{ij,t}\notag\\
=&\hat a_{nn}\sigma_{l}(G)\sum_{i\in G}\frac{1}{\hat a_{ii}}\hat a_{ii,t}-\sigma_{l}(G)\sum_{\stackrel{i,j\in G}{i\neq j}}\frac{\hat a_{ni}^2}{\hat a_{ii}}\frac{1}{\hat a_{jj}}\hat a_{jj,t}+\sigma_{l}(G)\sum_{\stackrel{i,j\in G}{i\neq j}}\frac{\hat a_{ni}}{\hat a_{ii}}\frac{\hat a_{nj}}{\hat a_{jj}}\hat a_{ij,t}\notag\\
\thicksim&\sigma_{l}(G)\sum\limits_{i,j \in G}
{\frac{{\hat a_{in} }} {{\hat a_{ii} }}\frac{{\hat a_{jn} }} {{\hat a_{jj}
}}\hat a_{ij,t } }, \notag
\end{align}
so we can get that
\begin{align} \label{6.29}
\phi_t\sim& \sigma_l(G)[\hat a_{nn,t }  - 2\sum\limits_{i \in G}
{\frac{{\hat a_{in} }} {{\hat a_{ii} }}\hat a_{in,t } + \sum\limits_{i,j \in G}
{\frac{{\hat a_{in} }} {{\hat a_{ii} }}\frac{{\hat a_{jn} }} {{\hat a_{jj}
}}\hat a_{ij,t } } } ].
\end{align}
By \eqref{6.1} and \eqref{6.20}, we have
\begin{align} \label{6.30}
\phi_t \sim \sigma_l(G)(-\frac{|u_t|}{|D u|{u_t}^3}) [\hat A_{nn,t}  - 2\sum\limits_{i \in G}
{\frac{{\hat A_{in} }} {{\hat A_{ii} }}\hat A_{in,t } + \sum\limits_{i,j \in G}
{\frac{{\hat A_{in} }} {{\hat A_{ii} }}\frac{{\hat A_{jn} }} {{\hat A_{jj}
}}\hat A_{ij,t} } } ],
\end{align}
and from \eqref{6.2}-\eqref{6.4}, we get
\begin{align*}
\hat A_{nn,t} =& \Big(\frac{1}{\hat W^2}\Big)_t \hat h_{nn}+\frac{1}{\hat W^2} \hat h_{nn,t}
-2\frac{u_n\sum_{l=1}^{n-1}u_{lt}\hat h_{nl}}{\hat W^2(1+\hat W)u_t^2},\\
- 2\sum\limits_{i \in G}{\frac{{\hat A_{in} }} {{\hat A_{ii} }}\hat A_{in,t}}
=&- 2\sum\limits_{i \in G}{\Big(\frac{1}{\hat W}\Big)\frac{{\hat h_{in} }} {{\hat h_{ii} }}\Big[ \Big(\frac{1}{\hat W}\Big)_t\hat h_{in}+\frac{1}{\hat W}\hat h_{in,t}
-\frac{u_{it}u_n \hat h_{nn}}{\hat W^2(1+\hat W)u_t^2} -\frac{u_n \sum_{l=1}^{n-1}u_{lt}\hat h_{il}}{\hat W(1+\hat W)u_t^2}\Big]} \notag \\
=&- 2\sum\limits_{i \in G}{\Big(\frac{1}{\hat W}\Big)\frac{{\hat h_{in} }} {{\hat h_{ii} }}\Big[ \Big(\frac{1}{\hat W}\Big)_t\hat h_{in}+\frac{1}{\hat W}\hat h_{in,t}
-\frac{u_{it}u_n \hat h_{nn}}{\hat W^2(1+\hat W)u_t^2} -\frac{u_n u_{it}\hat h_{ii}}{\hat W(1+\hat W)u_t^2}\Big]} \notag \\
\sim&- 2\Big(\frac{1}{\hat W}\Big)\Big(\frac{1}{\hat W}\Big)_t\hat h_{nn}
-2\frac{1}{\hat W^2}\sum\limits_{i \in G}{\frac{{\hat h_{in} }} {{\hat h_{ii} }}\hat h_{in,t}}
+\frac{2u_n \hat h_{nn}}{\hat W^3(1+\hat W)u_t^2}\sum\limits_{i \in G}{\frac{{u_{it}\hat h_{in} }} {{\hat h_{ii} }}} \notag\\
&+\frac{2u_n\sum\limits_{i \in G}{u_{it}\hat h_{in}}}{\hat W^2(1+\hat W)u_t^2}, \\
\sum\limits_{i,j \in G}{\frac{{\hat A_{in} }} {{\hat A_{ii} }}\frac{{\hat A_{jn} }} {{\hat A_{jj}
}} \hat A_{ij,t}} =&\sum\limits_{i,j \in G}{\Big(\frac{1}{\hat W^2}\Big)\frac{{\hat h_{in} }} {{\hat h_{ii} }}\frac{{\hat h_{jn} }} {{\hat h_{jj}
}}\Big[ \hat h_{ij,t}-\frac{u_{it}u_n \hat h_{jn}}{\hat W(1+\hat W)u_t^2} -\frac{u_{jt}u_n \hat h_{in}}{\hat W(1+ \hat W)u_t^2} \Big]} \notag \\
\sim&\Big(\frac{1}{\hat W^2}\Big)\sum\limits_{i,j \in G}{\frac{{\hat h_{in} }} {{\hat h_{ii} }}\frac{{\hat h_{jn} }} {{\hat h_{jj}
}}\hat h_{ij,t}}-\frac{2u_n \hat h_{nn}}{\hat W^3(1+\hat W)u_t^2}\sum\limits_{i \in G}{\frac{{u_{it}\hat h_{in} }} {{\hat h_{ii} }}},
\end{align*}
then
\begin{align} \label{6.31}
&\hat A_{nn,t}  - 2\sum\limits_{i \in G}
{\frac{{\hat A_{in} }} {{\hat A_{ii} }}\hat A_{in,t } + \sum\limits_{i,j \in G}
{\frac{{\hat A_{in} }} {{\hat A_{ii} }}\frac{{\hat A_{jn} }} {{\hat A_{jj}
}}\hat A_{ij,t} } } \notag\\
\sim& \frac{1}{\hat W^2} [\hat h_{nn,t }  - 2\sum\limits_{i \in G}
{\frac{{\hat h_{in} }} {{\hat h_{ii} }}\hat h_{in,t} + \sum\limits_{i,j \in G}
{\frac{{\hat h_{in} }} {{\hat h_{ii} }}\frac{{\hat h_{jn} }} {{\hat h_{jj}
}}\hat h_{ij,t} } }].
\end{align}
So by \eqref{6.30} and \eqref{6.31}, \eqref{6.28} holds.

Similarly, taking the first derivative of $\phi$ in the direction
$e_\alpha$, it follows that
\begin{align}
\phi_\alpha\sim& \sigma_l(G)[\hat a_{nn,\alpha }  - 2\sum\limits_{i \in G}
{\frac{{\hat a_{in} }} {{\hat a_{ii} }}\hat a_{in,\alpha } + \sum\limits_{i,j \in G}
{\frac{{\hat a_{in} }} {{\hat a_{ii} }}\frac{{\hat a_{jn} }} {{\hat a_{jj}
}}\hat a_{ij,\alpha } } } ], \notag
\end{align}
so
\begin{align}\label{6.31}
\hat a_{nn,\alpha }  - 2\sum\limits_{i \in G}
{\frac{{\hat a_{in} }} {{\hat a_{ii} }}\hat a_{in,\alpha } + \sum\limits_{i,j \in G}
{\frac{{\hat a_{in} }} {{\hat a_{ii} }}\frac{{\hat a_{jn} }} {{\hat a_{jj}
}}\hat a_{ij,\alpha } } } \sim 0.\notag
\end{align}
Similarly, we can get
\begin{align}
\hat A_{nn,\alpha }  - 2\sum\limits_{i \in G}
{\frac{{\hat A_{in} }} {{\hat A_{ii} }}\hat A_{in,\alpha } + \sum\limits_{i,j \in G}
{\frac{{\hat A_{in} }} {{\hat A_{ii} }}\frac{{\hat A_{jn} }} {{\hat A_{jj}
}}\hat A_{ij,\alpha } } }   \sim 0, \quad  \a =1, \cdots, n, \notag
\end{align}
and
\begin{equation}
\hat h_{nn,\alpha }  - 2\sum\limits_{i \in G}
{\frac{{\hat h_{in} }} {{\hat h_{ii} }}\hat h_{in,\alpha } + \sum\limits_{i,j \in G}
{\frac{{\hat h_{in} }} {{\hat h_{ii} }}\frac{{\hat h_{jn} }} {{\hat h_{jj}
}}\hat h_{ij,\alpha } } }   \sim 0, \quad  \a =1, \cdots, n.\notag
\end{equation}
The lemma holds.
\end{proof}

\begin{lemma}\label{lem6.4}
Under the above assumptions and notations, for any $(x,t) \in \mathcal {O} \times (t_0-\delta, t_0]$ with the coordinate \eqref{6.7}, we have
\begin{align}
\label{6.32} & u_{ij\a}=0, \quad i \in B, j \in B \cup G, \a \in B \cup G; \\
\label{6.33} & u_{ijn} =0, \quad u_{ijt} = 0, \quad  i \in B, j \in B\cup G; \\
\label{6.34} & u_{inn}  = 0, \quad u_{int} = 0, \quad u_{itt} = 0,\quad i \in B;
\end{align}

\begin{align}
\label{6.35} & u_t^2 u_{iii} = \hat h_{ii,i},  \quad i \in  G; \\
\label{6.36} & u_t^2 u_{iij} = \hat h_{ii,j}-2\frac{u_{j t}}{u_t}  \hat h_{ii}, ~u_t^2 u_{iji} = \hat h_{ij,i}+\frac{u_{j t}}{u_t}  \hat h_{ii}, ~u_t^2 u_{ijj} = \hat h_{ij,j}+\frac{u_{i t}}{u_t}  \hat h_{jj},~~i \in G, j \in G, i \ne j; \\
\label{6.37} & u_t^2 u_{iin} = \hat h_{ii,n}-2\frac{u_{n t}}{u_t}  \hat h_{ii}+2\frac{u_{i t}}{u_t} \hat h_{in}+2u_n u_{i t}^2,  \quad i \in G; \\
\label{6.38} & u_t^2 u_{ijk}  = \hat h_{ij,k}, \quad i \in G, j \in G, k \in G, i \ne j\ne k; \\
\label{6.39} &u_t^2 u_{ijn} = \hat h_{ij,n}+\frac{u_{i t}}{u_t}  \hat h_{jn}+\frac{u_{j t}}{u_t} \hat h_{in}+ 2 u_n u_{i t}u_{j t},  \quad i \in G, j \in G, i \ne j;   \\
\label{6.40} &u_t^2u_{nni} =-\sum\limits_{k \in G} {\hat h_{kk,i}} +u_t^2 u_{it} +2\frac{{u_{it} }}{u_t}\sum\limits_{k \in G, k\ne i} {\hat h_{kk}}, \quad i \in G\\
\label{6.41} &u_t^2u_{nnn} =-\sum\limits_{k \in G} {\hat h_{kk,n}} +u_t^2 u_{nt}
  +2\frac{{u_{nt} }}{u_t}\sum\limits_{k \in G} {\hat h_{kk}}-2\sum\limits_{k \in G} {\frac{{u_{kt} }}{u_t}\hat h_{kn}}-2\sum\limits_{k \in G} {u_n u_{kt} ^2};
\end{align}

\begin{align}
\label{6.42} &u_{n}u_t u_{ii t} =-\hat h_{in,i}+\hat h_{ii,n}+3\frac{{u_{it} }}{u_t}\hat h_{in} -3\frac{{u_{nt} }}{u_t}\hat h_{ii}+2u_n u_{it} ^2+ u_{ii}u_n u_{tt}, \quad  i \in G;  \\
\label{6.43} &u_{n}u_t u_{ij t} =-\hat h_{in,j}+\hat h_{ij,n}+3\frac{{u_{jt} }}{u_t}\hat h_{in} +2u_n u_{it} u_{jt}, \quad  i \in G, j \in G, i \ne j;\\
\label{6.44} &u_{n}u_t u_{nn t} =\sum\limits_{i \in G} {[\hat h_{in,i}-\hat h_{ii,n}]}+u_{n}u_t u_{tt}-\sum\limits_{i \in G} [3\frac{{u_{it} }}{u_t}\hat h_{in} -3\frac{{u_{nt} }}{u_t}\hat h_{ii}+2u_n u_{it} ^2+ u_{ii}u_n u_{tt}];\\
\label{6.45} &u_{n}u_t u_{in t} =-\hat h_{in,n}-\sum\limits_{k \in G} \hat h_{kk,i}+3\frac{{u_{it} }}{u_t}\sum\limits_{k \in G,k \ne i} \hat h_{kk} +\frac{{u_{it} }}{u_t}\hat h_{ii}+\frac{{u_{nt} }}{u_t}\hat h_{in}+ u_{in}u_n u_{tt}, \quad  i \in G;
\end{align}
and
\begin{align}
\label{6.46} u_{n}^2 u_{tti} =&\hat h_{nn,i}-2\hat h_{in,n}-\sum\limits_{k \in G} \hat h_{kk,i} -u_t^2 u_{it} +4\frac{{u_{it} }}{u_t}\sum\limits_{k \in G,k \ne i} \hat h_{kk} \notag \\
&+2\frac{{u_{it} }}{u_t}\hat h_{ii} +2\frac{{u_{nt} }}{u_t}\hat h_{in}- 2 u_t u_{it} u_{nn} + 2u_{t}u_{ni} u_{tn}+ 2u_{n}u_{ti} u_{tn}, ~ i \in G;  \\
\label{6.47} u_{n}^2 u_{ttn} =&\hat h_{nn,n}+2\sum\limits_{i \in G}\hat h_{in,i}-\sum\limits_{k \in G} \hat h_{kk,n}\notag \\
&-u_t^2 u_{nt}-4\sum\limits_{i \in G} \frac{{u_{it} }}{u_t}\hat h_{in}+ 4 \frac{{u_{nt} }}{u_t}\sum\limits_{i \in G}\hat h_{ii} - 2\sum\limits_{i \in G} u_n  u_{it}^2+ 2u_{n}u_{tn}^2,
\end{align}
\end{lemma}
\begin{proof}
By \eqref{6.1}, \eqref{6.11} and \eqref{6.12}, we get for $i \in B$ or $j \in B$,
\begin{align*}
0={D \hat a_{ij} } &=-\frac{|u_t|}{|D u|{u_t}^3} {D \hat A_{ij} } \\
&=-\frac{|u_t|}{|D u|{u_t}^3}
[{D \hat h_{ij} }-\frac{D u_i u_n \hat h_{jn}}{\hat W(1+\hat W)u_t^2}
 -\frac{Du_ju_n \hat h_{in}}{\hat W(1+ \hat W)u_t^2} ]\\
 &=-\frac{|u_t|}{|D u|{u_t}^3}
{D \hat h_{ij} },  \\
0=D \hat a_{in}  &= -\frac{|u_t|}{|D u|{u_t}^3} D \hat A_{in} \\
 &=-\frac{|u_t|}{|D u|{u_t}^3} [ \frac{1}{\hat W}D \hat h_{in}
-\frac{D u_iu_n \hat h_{nn}}{\hat W^2(1+\hat W)u_t^2} -\frac{u_n \sum_{l=1}^{n-1} Du_l
\hat h_{il}}{\hat W(1+\hat W)u_t^2}]\\
 &=-\frac{|u_t|}{|D u|{u_t}^3} \frac{1}{\hat W}D \hat h_{in},
\end{align*}
so we get
\begin{align*}
& 0= \hat h_{ij, \a}=u_t^2 u_{ij\a}, \quad i \in B, j \in B, \a \in B \cup G; \\
& 0=\hat h_{ij, n}=u_t^2 u_{ijn} - u_t u_{in} u_{j t} - u_t u_{jn} u_{i t}=u_t^2 u_{ijn}, \quad  i \in B, j \in B; \\
& 0=\hat h_{ij,t}=u_t^2 u_{ijt} - u_t u_{it} u_{j t} - u_t u_{jt} u_{i t}=u_t^2 u_{ijt}, \quad  i \in B, j \in B; \\
&0=\hat h_{ij,\a}=u_t^2 u_{ij\a} - u_t u_{j\a} u_{i t}=u_t^2 u_{ij\a} , \quad  i \in B, j \in G, \a \in B \cup G; \\
&0=\hat h_{ij,n}=u_t^2 u_{ijn} - u_t u_{in} u_{j t} - u_t u_{jn} u_{i t} =u_t^2 u_{ijn} ,  \quad i \in B, j \in G; \\
&0=\hat h_{ij,t}=u_t^2 u_{ijt} - 2u_t u_{it} u_{j t}=u_t^2 u_{ijt},  \quad i \in B, j \in G; \\
&0=\hat h_{in, \a}=u_t^2 u_{in\a} - u_t u_{n } u_{i t\a} =u_t^2 u_{in\a}, \quad  i \in B, \a \in B \cup G; \\
&0=\hat h_{in,n}=u_t^2 u_{inn}  - u_t u_{n} u_{i n t}, \quad  i \in B; \\
&0=\hat h_{in,t}=u_t^2 u_{int} - u_t u_{n} u_{i t t}, \quad  i \in B;
\end{align*}
and by the equation \eqref{1.3}, we have
\begin{align*}
u_{nni} =\Delta u_{i}- \sum\limits_{k=1}^{n-1} u_{kki}=u_{it}- \sum\limits_{k \in G} u_{kki}=0, \quad i \in B.
\end{align*}
So
\begin{align*}
& u_{ij\a}=0, \quad i \in B, j \in B \cup G, \a \in B \cup G; \\
& u_{ijn} =0, \quad u_{ijt} = 0, \quad  i \in B, j \in B\cup G; \\
& u_{inn}  = 0, \quad u_{int} = 0, \quad u_{itt} = 0,\quad i \in B;
\end{align*}

For $i,j \in G$, we can get
\begin{align}\label{6.48}
\hat h_{ij,\alpha}= u_t^2 u_{ij\a}+ 2u_t u_{t \a} u_{ij} -u_{i \a} u_t u_{tj}- u_{j \a} u_t u_{ti},
\end{align}
so
\begin{align*}
u_t^2 u_{iii} =&  \hat h_{ii,i}-[2u_t u_{ti} u_{ii} -u_{ii} u_t u_{ti}- u_{ii} u_t u_{ti}]  = \hat h_{ii,i},  \quad i \in  G;
\end{align*}
and
\begin{align*}
 u_t^2 u_{iij} =&\hat h_{ii,j}- [2u_t u_{tj} u_{ii} -u_{ij} u_t u_{ti}- u_{i j} u_t u_{ti}]  \\
=& \hat h_{ii,j}- 2u_t u_{tj} u_{ii}   \\
 =& \hat h_{ii,j}-2\frac{u_{j t}}{u_t}  \hat h_{ii},  \quad i \in G, j \in G, i \ne j.
\end{align*}
Similarly, we have from \eqref{6.48}
\begin{align*}
 u_t^2 u_{iji} =&\hat h_{ij,i}- [2u_t u_{ti} u_{ij} -u_{ii} u_t u_{tj}- u_{ji} u_t u_{ti}]  \\
 =&\hat  h_{ij,i}+u_t u_{tj} u_{ii}   \\
 =& \hat  h_{ij,i}+\frac{u_{j t}}{u_t}  \hat h_{ii},  \quad i \in G, j \in G, i \ne j;  \\
 u_t^2 u_{ijj} =&\hat h_{ij,j}- [2u_t u_{tj} u_{ij} -u_{ij} u_t u_{tj}- u_{jj} u_t u_{ti}]  \\
 =&\hat h_{ij,j}+ u_{jj} u_t u_{ti}   \\
 =& \hat h_{ij,j}+\frac{u_{i t}}{u_t}  \hat h_{jj},  \quad i \in G, j \in G, i \ne j.
\end{align*}
From \eqref{6.48}, we also have
\begin{align*}
u_t^2 u_{iin} =&\hat h_{ii,n}- [2u_t u_{tn} u_{ii} -u_{in} u_t u_{ti}- u_{in} u_t u_{ti}]  \\
 =&\hat h_{ii,n}- 2u_t u_{tn} u_{ii} +2 u_{in} u_t u_{ti}  \\
=&\hat h_{ii,n}- 2\frac{u_{n t}}{u_t} u_{t}^2 u_{ii} +2\frac{u_{it}}{u_t}[ \hat h_{in}+u_{n} u_t u_{ti}]   \\
=& \hat h_{ii,n}-2\frac{u_{n t}}{u_t} \hat h_{ii}+2\frac{u_{i t}}{u_t} \hat h_{in}+2u_n u_{i t}^2,  \quad i \in G;  \\
u_t^2 u_{ijk}  =& \hat h_{ij,k} - [2u_t u_{tk} u_{ij} -u_{ik} u_t u_{tj}- u_{jk} u_t u_{ti}] = \hat h_{ij,k} , \quad i \in G, j \in G, k \in G, i \ne j\ne k;\\
u_t^2 u_{ijn} =& \hat h_{ij,n} - [2u_t u_{tn} u_{ij} -u_{in} u_t u_{tj}- u_{jn} u_t u_{ti}]  \\
=& \hat h_{ij,n} +u_{in} u_t u_{tj} + u_{jn} u_t u_{ti}  \\
=& \hat h_{ij,n} + \frac{u_{j t}}{u_t} [ \hat h_{in} + u_n u_{t}u_{i t}]+\frac{u_{i t}}{u_t} [ \hat h_{jn} + u_n u_{t}u_{j t}]  \\
=& \hat h_{ij,n}+\frac{u_{i t}}{u_t}  \hat h_{jn}+\frac{u_{j t}}{u_t} \hat h_{in}+ 2 u_n u_{i t}u_{j t},  \quad i \in G, j \in G, i \ne j.
\end{align*}
And
\begin{align*}
u_t^2u_{nni} =&u_t^2[\Delta u_{i}- \sum\limits_{k \in G} u_{kki}]=u_t^2 u_{it}-u_t^2u_{iii} -\sum\limits_{k \in G, k \ne i} u_t^2u_{kki}  \\
=&u_t^2 u_{it}- \hat h_{ii,i}-\sum\limits_{k \in G, k \ne i} [\hat h_{kk,i}-2\frac{u_{i t}}{u_t}  \hat h_{kk}]  \\
=& u_t^2 u_{it}-\sum\limits_{k \in G} {\hat h_{kk,i}} +2\frac{{u_{it} }}{u_t}\sum\limits_{k \in G, k\ne i} {\hat h_{kk}}, \quad i \in G,  \\
u_t^2u_{nnn} =& u_t^2[\Delta u_{n}- \sum\limits_{k \in G} u_{kkn}]  \\
=& u_t^2 u_{nt}- \sum\limits_{k \in G}[ \hat h_{kk,n}-2\frac{u_{n t}}{u_t} \hat h_{kk}+2\frac{u_{k t}}{u_t} \hat h_{kn}+2u_n u_{k t}^2]  \\
=&u_t^2  u_{nt}-\sum\limits_{k \in G} {\hat h_{kk,n}}
+2\frac{{u_{nt} }}{u_t}\sum\limits_{k \in G} {\hat h_{kk}}-2\sum\limits_{k \in G} {\frac{{u_{kt} }}{u_t}\hat h_{kn}}-2\sum\limits_{k \in G} {u_n u_{kt} ^2}.
\end{align*}

For $i \in G$, we can get from \eqref{6.5}
\begin{align}\label{6.49}
\hat h_{in,\alpha}= u_t^2 u_{in\a}+ 2u_t u_{t \a} u_{in} +u_{i\a}u_n u_{tt}-u_{i \a} u_t u_{tn}- u_{n \a} u_t u_{ti}- u_nu_{t \a}  u_{ti}- u_{n} u_t u_{ti\a},
\end{align}
then
\begin{align*}
u_{n}u_t u_{ii t} =& -\hat h_{in,i}+ u_t^2 u_{iin}+ 2u_t u_{ti} u_{in} +u_{ii}u_n u_{tt}-u_{i i} u_t u_{tn}- u_{n i} u_t u_{ti}- u_nu_{ti} u_{ti}  \\
=&-\hat h_{in,i}+ [\hat h_{ii,n}-2\frac{u_{n t}}{u_t} \hat h_{ii}+2\frac{u_{i t}}{u_t} \hat h_{in}+2u_n u_{i t}^2]  \\
&+ 2u_t u_{ti} u_{in} +u_{ii}u_n u_{tt}-u_{i i} u_t u_{tn}- u_{n i} u_t u_{ti}- u_nu_{ti} u_{ti}  \\
=&-\hat h_{in,i}+\hat h_{ii,n}+3\frac{{u_{it} }}{u_t}\hat h_{in} -3\frac{{u_{nt} }}{u_t}\hat h_{ii}+2u_n u_{it} ^2+ u_{ii}u_n u_{tt}, \quad  i \in G;  \\
u_{n}u_t u_{ij t} =&-\hat h_{in,j}+ u_t^2 u_{ijn}+ 2u_t u_{tj} u_{in} +u_{ij}u_n u_{tt}-u_{i j} u_t u_{tn}- u_{n j} u_t u_{ti}- u_nu_{tj}  u_{ti}\\
=&-\hat h_{in,j}+[\hat h_{ij,n}+\frac{u_{i t}}{u_t}  \hat h_{jn}+\frac{u_{j t}}{u_t} \hat h_{in}+ 2 u_n u_{i t}u_{j t}]   \\
&+ 2u_t u_{tj} u_{in} - u_{n j} u_t u_{ti}- u_n u_{tj}  u_{ti}\\
=&-\hat h_{in,j}+[\hat h_{ij,n}+\frac{u_{i t}}{u_t}  \hat h_{jn}+\frac{u_{j t}}{u_t} \hat h_{in}+ 2 u_n u_{i t}u_{j t}]   \\
&+ 2\frac{u_{j t}}{u_t} [\hat h_{in}+u_n u_{t} u_{ti} ]- \frac{u_{i t}}{u_t} [ \hat h_{jn}+u_n u_{t} u_{tj} ]- u_n u_{tj}  u_{ti}\\
=&-\hat h_{in,j}+\hat h_{ij,n}+3\frac{{u_{jt} }}{u_t}\hat h_{in} +2u_n u_{it} u_{jt}, \qquad  i \in G, j \in G, i \ne j;
\end{align*}

\begin{align*}
u_{n}u_t u_{in t} =&-\hat h_{in,n}+ u_t^2 u_{inn}+ 2u_t u_{tn} u_{in} +u_{in}u_n u_{tt}-u_{in} u_t u_{tn}- u_{nn} u_t u_{ti}- u_nu_{tn}  u_{ti} \\
=&-\hat h_{in,n}+ [u_t^2 u_{it}-\sum\limits_{k \in G} {\hat h_{kk,i}} +2\frac{{u_{it} }}{u_t}\sum\limits_{k \in G, k\ne i} {\hat h_{kk}}]  \\
&+ u_t u_{tn} u_{in} +u_{in}u_n u_{tt}- u_{nn} u_t u_{ti}- u_nu_{tn}  u_{ti} \\
=&-\hat h_{in,n}-\sum\limits_{k \in G} {\hat h_{kk,i}}+2\frac{{u_{it} }}{u_t}\sum\limits_{k \in G, k\ne i} {\hat h_{kk}}+ [u_t u_{it} \Delta u- u_{nn} u_t u_{ti} ]  \\
&+\frac{{u_{nt} }}{u_t}[u_t^2u_{in} -u_nu_{t}  u_{ti}]+u_{in}u_n u_{tt} \\
=&-\hat h_{in,n}-\sum\limits_{k \in G} \hat h_{kk,i}+3\frac{{u_{it} }}{u_t}\sum\limits_{k \in G,k \ne i} \hat h_{kk} +\frac{{u_{it} }}{u_t}\hat h_{ii}+\frac{{u_{nt} }}{u_t}\hat h_{in}+ u_{in}u_n u_{tt}, \quad  i \in G.
\end{align*}
So by the equation, we have
\begin{align*}
u_{n}u_t u_{nn t} =& u_{n}u_t [\Delta u_{t}- \sum\limits_{k \in G} u_{kkt}]\\
=&u_{n}u_t u_{tt}- \sum\limits_{i \in G}[-\hat h_{in,i}+\hat h_{ii,n}+3\frac{{u_{it} }}{u_t}\hat h_{in} -3\frac{{u_{nt} }}{u_t}\hat h_{ii}+2u_n u_{it} ^2+ u_{ii}u_n u_{tt}] \notag \\
=&u_{n}u_t u_{tt}+\sum\limits_{i \in G} {[\hat h_{in,i}-\hat h_{ii,n}]}-\sum\limits_{i \in G} [3\frac{{u_{it} }}{u_t}\hat h_{in} -3\frac{{u_{nt} }}{u_t}\hat h_{ii}+2u_n u_{it} ^2+ u_{ii}u_n u_{tt}].
\end{align*}

At last, we can get
\begin{align}\label{6.50}
\hat h_{nn,\alpha}=& u_t^2 u_{nn\a}+ 2u_t u_{t \a} u_{nn} + u_n^2 u_{tt\a}+ 2u_n u_{n\a}u_{tt}-2u_{n \a} u_t u_{tn} \notag \\
&-2 u_nu_{t \a}  u_{tn}- 2u_{n} u_t u_{tn\a},
\end{align}
so
\begin{align*}
u_{n}^2 u_{tti} =&\hat h_{nn,i}-[ u_t^2 u_{nni}+ 2u_t u_{ti} u_{nn}+ 2u_n u_{ni}u_{tt}-2u_{ni} u_t u_{tn} -2 u_nu_{ti}  u_{tn}- 2u_{n} u_t u_{tni} ] \\
=&\hat h_{nn,i}-[u_t^2 u_{it}-\sum\limits_{k \in G} {\hat h_{kk,i}} +2\frac{{u_{it} }}{u_t}\sum\limits_{k \in G, k\ne i} {\hat h_{kk}}]  \\
&+2[-\hat h_{in,n}-\sum\limits_{k \in G} \hat h_{kk,i}+3\frac{{u_{it} }}{u_t}\sum\limits_{k \in G,k \ne i} \hat h_{kk} +\frac{{u_{it} }}{u_t}\hat h_{ii}+\frac{{u_{nt} }}{u_t}\hat h_{in}+ u_{in}u_n u_{tt}]  \\
& -[2u_t u_{ti} u_{nn}+ 2u_n u_{ni}u_{tt}-2u_{ni} u_t u_{tn} -2 u_nu_{ti}  u_{tn}] \\
=&\hat h_{nn,i}-2\hat h_{in,n}-\sum\limits_{k \in G} \hat h_{kk,i} -u_t^2 u_{it} +4\frac{{u_{it} }}{u_t}\sum\limits_{k \in G,k \ne i} \hat h_{kk} \\
&+2\frac{{u_{it} }}{u_t}\hat h_{ii} +2\frac{{u_{nt} }}{u_t}\hat h_{in}- 2 u_t u_{it} u_{nn} + 2u_{t}u_{ni} u_{tn}+ 2u_{n}u_{ti} u_{tn}, ~  i \in G;  \\
\end{align*}
and
\begin{align*}
u_{n}^2 u_{ttn} =&  \hat h_{nn,n}-[ u_t^2 u_{nnn}+ 2u_t u_{tn} u_{nn} + 2u_n u_{nn}u_{tt}-2u_{nn} u_t u_{tn} -2 u_nu_{tn}  u_{tn}- 2u_{n} u_t u_{tnn}] \\
=&  \hat h_{nn,n}-[ u_t^2 u_{nt}-\sum\limits_{k \in G} {\hat h_{kk,n}}
+2\frac{{u_{nt} }}{u_t}\sum\limits_{k \in G} {\hat h_{kk}}-2\sum\limits_{k \in G} {\frac{{u_{kt} }}{u_t}\hat h_{kn}}-2\sum\limits_{k \in G} {u_n u_{kt} ^2}] \\
&+2[u_{n}u_t u_{tt}+\sum\limits_{i \in G} {(\hat h_{in,i}-\hat h_{ii,n})}-\sum\limits_{i \in G} (3\frac{{u_{it} }}{u_t}\hat h_{in} -3\frac{{u_{nt} }}{u_t}\hat h_{ii}+2u_n u_{it} ^2+ u_{ii}u_n u_{tt})]   \\
&-[  2u_n u_{nn}u_{tt} -2 u_nu_{tn}  u_{tn}] \\
=&\hat h_{nn,n}+2\sum\limits_{i \in G}\hat h_{in,i}-\sum\limits_{k \in G} \hat h_{kk,n}+[2u_{n}u_t u_{tt}-2\sum\limits_{i \in G}(u_{ii}u_n u_{tt}) -  2u_n u_{nn}u_{tt}] \\
&-u_t^2 u_{nt}-4\sum\limits_{i \in G} \frac{{u_{it} }}{u_t}\hat h_{in}+ 4 \frac{{u_{nt} }}{u_t}\sum\limits_{i \in G}\hat h_{ii} - 2\sum\limits_{i \in G} u_n  u_{it}^2+ 2u_{n}u_{tn}^2 \\
=&\hat h_{nn,n}+2\sum\limits_{i \in G}\hat h_{in,i}-\sum\limits_{k \in G} \hat h_{kk,n} \\
&-u_t^2u_{nt}-4\sum\limits_{i \in G} \frac{{u_{it} }}{u_t}\hat h_{in}+ 4 \frac{{u_{nt} }}{u_t}\sum\limits_{i \in G}\hat h_{ii} - 2\sum\limits_{i \in G} u_n  u_{it}^2+ 2u_{n}u_{tn}^2.
\end{align*}

The lemma holds.
\end{proof}

\subsection{Step 2: reduction for the second derivatives of the test function $\phi$ }

\begin{lemma} \label{lem6.5}
Under the above assumptions and notations, for any $(x,t) \in \mathcal {O} \times (t_0-\delta, t_0]$ with the coordinate \eqref{6.7}, we have
\begin{eqnarray}\label{6.51}
\phi_{\alpha \alpha} &\sim &\left[ {\sigma _l (G) + \hat a_{nn} \sigma
_{l-1} (G) - \sum\limits_{i \in G} {\hat a_{in} ^2 \sigma _{l - 2}
(G\left| i\right.)} }\right]\sum\limits_{m \in B} {\hat a_{mm, \alpha \alpha} }  \notag \\
&&+ \sigma_l(G)\left[ {\hat a_{nn,\alpha \alpha}  - 2\sum\limits_{i \in
G} {\frac{{\hat a_{in} }} {{\hat a_{ii}}}\hat a_{in,\alpha \alpha }+\sum\limits_{i,j
\in G} {\frac{{\hat a_{in} }} {{\hat a_{ii} }}\frac{{\hat a_{jn} }} {{\hat a_{jj}
}}\hat a_{ij,\alpha \alpha} } } } \right]\\
&&- 2\sigma _l (G)\sum\limits_{i \in G} {\frac{1} {{\hat a _{ii}
}}\left[ {\hat a_{in,\alpha }  - \sum\limits_{j \in G} {\frac{{\hat a_{jn} }}
{{\hat a_{jj} }}\hat a_{ij,\alpha } } } \right]^2},   \notag
\end{eqnarray}
where
\begin{eqnarray} \label{6.52}
\sigma _l (G) + \hat a_{nn} \sigma _{l - 1} (G) - \sum\limits_{i \in G}
{\hat a_{in} ^2 \sigma _{l - 2} (G\left| i \right.)}  \sim \sigma _l
(G)\left( {1 + \sum\limits_{i \in G} {\frac{{\hat a_{in} ^2 }} {{\hat a
_{ii}^2 }}} } \right).
\end{eqnarray}
\end{lemma}

\begin{proof} The proof is similar as in \cite{CH12}. For completeness, we give the details of the proof.

Computing the second derivatives directly, we have from Lemma \ref{lem2.5}
\begin{align*}
\phi _{\alpha \alpha}=& \sum_{i,j}\frac{{\partial \sigma _{l + 1} (M)}}{{\partial \hat a_{ij} }}\hat a_{ij,\alpha \a }
+ \sum_{i,j,k,l}\frac{{\partial ^2 \sigma_{l + 1} (M)}} {{\partial \hat a_{ij}\partial \hat a_{kl} }}\hat a_{ij,\alpha }\hat a_{kl,\a}+ \hat a_{nn,\alpha \a } \sigma _l (M)\\
&+ 2\hat a_{nn,\alpha }\sum_{i,j}\frac{{\partial \sigma _l (M)}} {{\partial \hat a_{ij}}}\hat a_{ij,\a}
+\hat a_{nn}\sum_{i,j}\frac{{\partial \sigma _l (M)}} {{\partial \hat a_{ij}}}\hat a_{ij,\alpha\a}
+\hat a_{nn}\sum_{i,j,k,l}\frac{{\partial ^2 \sigma _l (M)}}{{\partial \hat a_{ij}\partial \hat a_{kl}}}\hat a_{ij,\alpha }\hat a_{kl,\a}
\\
&-2\sum\limits_i{\sigma _{l -1} (M|i)\hat a_{ni} \hat a_{ni,\alpha \a }}
-2\sum\limits_i{\sigma _{l - 1} (M|i)\hat a_{ni,\alpha }\hat a_{ni,\a}}
-4\sum\limits_{i,j,k}{\frac{{\partial\sigma_{l - 1}(M|i)}}{{\partial \hat a_{jk}}}\hat a_{ni}\hat a_{ni,\alpha}\hat a_{jk,\a }}\notag \\
&-\sum_{i,j,k}\frac{{\partial\sigma_{l-1}(M|i)}}{{\partial \hat a_{jk}}}\hat a_{ni}^2\hat a_{jk,\alpha \a}
-\sum_{i,j,k,p,q}\frac{{\partial^2\sigma_{l-1}(M|i)}}{{\partial \hat a_{jk}\partial \hat a_{pq}}}\hat a_{ni}^2\hat a_{jk,\alpha }\hat a_{pq,\a}\\
&+\sum_{\stackrel{i,j}{i\ne j}}{\sigma_{l-2}(M|ij)}\hat a_{ni}\hat a_{nj}\hat a_{ij,\alpha \a }
+4\sum_{\stackrel{i,j}{i\ne j}}{\sigma_{l-2}(M|ij)\hat a_{nj}\hat a_{ni,\alpha}\hat a_{ij,\a}}\\
&+2\sum_{\stackrel{i,j,k,l}{i\ne j}}{\frac{{\partial\sigma_{l-2}(M|ij)}}{{\partial \hat a_{kl}}}\hat a_{ni}\hat a_{nj}\hat a_{ij,\a}\hat a_{kl,\alpha}}-2\sum_{\stackrel{i,j,k }{i\ne j,i\ne k,j\ne k}}{\sigma_{l-3}(M|ijk)\hat a_{ni}\hat a_{nj}\hat a_{ki,\alpha}\hat a_{kj,\a}}\\
=&I_{\alpha}+II_{\alpha}+III_{\alpha}+IV_{\alpha},
\end{align*}

where
\begin{align*}
I_{\alpha}=& \sum_{i,j}\frac{{\partial \sigma _{l + 1} (M)}}{{\partial \hat a_{ij} }}\hat a_{ij,\alpha \a }
+ \hat a_{nn,\alpha \a } \sigma _l (M)
+\hat a_{nn}\sum_{i,j}\frac{{\partial \sigma _l (M)}} {{\partial \hat a_{ij}}}\hat a_{ij,\alpha\a}\\
&-2\sum\limits_i{\sigma _{l -1} (M|i)\hat a_{ni} \hat a_{ni,\alpha \a }}
-\sum_{i,j,k}\frac{{\partial\sigma_{l-1}(M|i)}}{{\partial \hat a_{jk}}}\hat a_{ni}^2\hat a_{jk,\alpha \a}
+\sum_{\stackrel{i,j}{i\ne j}}{\sigma_{l-2}(M|ij)}\hat a_{ni}\hat a_{nj}\hat a_{ij,\alpha \a }, \\
II_{\alpha}=& \sum_{i,j,k,l}\frac{{\partial ^2 \sigma_{l + 1} (M)}} {{\partial \hat a_{ij}\partial \hat a_{kl} }}\hat a_{ij,\alpha }\hat a_{kl,\a}
+\hat a_{nn}\sum_{i,j,k,l}\frac{{\partial ^2 \sigma _l (M)}}{{\partial \hat a_{ij}\partial \hat a_{kl}}}\hat a_{ij,\alpha }\hat a_{kl,\a}
-\sum_{i,j,k,p,q}\frac{{\partial^2\sigma_{l-1}(M|i)}}{{\partial \hat a_{jk}\partial \hat a_{pq}}}\hat a_{ni}^2\hat a_{jk,\alpha }\hat a_{pq,\a},\\
III_{\alpha}=& 2\hat a_{nn,\alpha }\sum_{i,j}\frac{{\partial \sigma _l (M)}} {{\partial \hat a_{ij}}}\hat a_{ij,\a}
-4\sum\limits_{i,j,k}{\frac{{\partial\sigma_{l - 1}(M|i)}}{{\partial \hat a_{jk}}}\hat a_{ni}\hat a_{ni,\alpha}\hat a_{jk,\a }}
+2\sum_{\stackrel{i,j,k,l}{i\ne j}}{\frac{{\partial\sigma_{l-2}(M|ij)}}{{\partial \hat a_{kl}}}\hat a_{ni}\hat a_{nj}\hat a_{ij,\a}\hat a_{kl,\alpha}},\\
IV_{\alpha}=& -2\sum\limits_i{\sigma _{l - 1} (M|i)\hat a_{ni,\alpha }\hat a_{ni,\a}}
+4\sum_{\stackrel{i,j}{i\ne j}}{\sigma_{l-2}(M|ij)\hat a_{nj}\hat a_{ni,\alpha}\hat a_{ij,\a}}\notag\\
&-2\sum_{\stackrel{i,j,k }{i\ne j,i\ne k,j\ne k}}{\sigma_{l-3}(M|ijk)\hat a_{ni}\hat a_{nj}\hat a_{ki,\alpha}\hat a_{kj,\a}}.
\end{align*}

Now we use the formulas (\ref{6.9})-(\ref{6.12}), (\ref{6.20}) and (\ref{6.25}) to treat the terms in
$I_{\alpha}, II_{\alpha}, III_{\alpha}$ and $IV_{\alpha}$.

First, we will deal with $I_{\alpha}$.
\begin{eqnarray} \label{6.53}
I_{\alpha}&=& \sigma _l (G)\sum\limits_{m \in B}
{\hat a_{mm,\alpha \a} }  + \hat a_{nn,\alpha \a } \sigma _l (G) + \hat a_{nn}
[ {\sum\limits_{m \in G} {\sigma _{l - 1} (G\left| m
\right.)\hat a_{mm,\alpha \a } }  + \sigma _{l - 1} (G)\sum\limits_{m
\in B} {\hat a_{mm,\alpha \a } } } ]  \notag \\
&& - 2\sum\limits_{i \in G} {\hat a_{in} \hat a_{in,\alpha \a } \sigma _{l -
1} (G\left| i \right.)}  - \sum\limits_{i \in G} {\hat a_{in} ^2 [
{\sum\limits_{\scriptstyle m \in G \hfill \atop \scriptstyle m \ne i
\hfill}  {\sigma _{l - 2} (G\left| {im} \right.)\hat a_{mm,\alpha \a }
}  + \sigma _{l - 2} (G\left| i \right.)\sum\limits_{m \in B}
{\hat a_{mm,\alpha \a } } } ]}  \notag \\
&& + \sum\limits_{\scriptstyle ij \in G \hfill \atop \scriptstyle i
\ne j \hfill}  {\hat a_{in} \hat a_{jn} \hat a_{ij,\alpha \a } \sigma _{l - 2}
(G\left| {ij} \right.)}  \notag\\
&=& [ {\sigma _l (G) + \hat a_{nn} \sigma
_{l-1} (G) - \sum\limits_{i \in G} {\hat a_{in} ^2 \sigma _{l - 2}
(G\left| i\right.)} }]\sum\limits_{m \in B} {\hat a_{mm, \alpha \alpha} } \notag\\
&&+\sigma _l(G)\sum\limits_{m \in G} {\hat a_{mm,\alpha \a } \frac{1}
{{\hat a_{mm} }}[ {\hat a_{nn}  - \sum\limits_{\scriptstyle i \in G
\hfill\atop\scriptstyle i \ne m \hfill}  {\frac{{\hat a_{in} ^2}} {{\hat a_{ii} }}} } ]}  \notag \\
&& + \sigma _l (G)[ {\hat a_{nn,\alpha \alpha}  - 2\sum\limits_{i \in
G} {\frac{{\hat a_{in} }} {{\hat a_{ii}}}\hat a_{in,\alpha \alpha } +
\sum\limits_{\scriptstyle i,j \in G \hfill \atop \scriptstyle i \ne j
\hfill}  {\frac{{\hat a_{in} }} {{\hat a_{ii} }}\frac{{\hat a_{jn} }} {{\hat a_{jj}
}}\hat a_{ij,\alpha \alpha} } } } ] \notag \\
&\sim& [ {\sigma _l (G) + \hat a_{nn} \sigma
_{l-1} (G) - \sum\limits_{i \in G} {\hat a_{in} ^2 \sigma _{l - 2}
(G\left| i\right.)} }]\sum\limits_{m \in B} {\hat a_{mm, \alpha \alpha} } \notag\\
&&+ \sigma_l(G)[ {\hat a_{nn,\alpha \alpha}  - 2\sum\limits_{i \in
G} {\frac{{\hat a_{in} }} {{\hat a_{ii}}}\hat a_{in,\alpha \alpha }+\sum\limits_{i,j
\in G} {\frac{{\hat a_{in} }} {{\hat a_{ii} }}\frac{{\hat a_{jn} }} {{\hat a_{jj}
}}\hat a_{ij,\alpha \alpha} } } } ].
\end{eqnarray}

For $II_{\alpha}$, it follows that
\begin{align}\label{6.54}
II_{\alpha}=&\sum_{i\neq j}\sigma_{l - 1}(G|ij)(\hat a_{ii,\alpha }\hat a_{jj,\a}-\hat a_{ij,\alpha }^2 )
+\hat a_{nn}\sum_{i\neq j}\sigma_{l - 2}(G|ij)(\hat a_{ii,\alpha }\hat a_{jj,\a}-\hat a_{ij,\alpha }^2 )\notag\\
&-\sum_i\sum_{\stackrel{j,k}{j\ne i,k \ne i,j\ne k}}\sigma_{l - 3}(G|ijk)
\hat a_{ni}^2(\hat a_{jj,\alpha }\hat a_{kk,\a}-\hat a_{jk,\alpha }^2)\notag\\
=&\hat a_{nn}\sum_{\stackrel{i,j\in G}{i\ne j}}\sigma_{l - 2}(G|ij)(\hat a_{ii,\alpha }\hat a_{jj,\a}-\hat a_{ij,\alpha }^2 )
-\sum_{\stackrel{i,j,k\in G }{i\ne j,i\ne k,j\ne k}}\sigma_{l - 3}(G|ijk)
\hat a_{ni}^2(\hat a_{jj,\alpha }\hat a_{kk,\a}-\hat a_{jk,\alpha }^2)\notag\\
=&\hat a_{nn}\sigma _l(G)\sum_{\stackrel{i,j\in G}{i\ne j}}\frac{\hat a_{ii,\alpha }\hat a_{jj,\a}-\hat a_{ij,\alpha }^2}{\hat a_{ii}\hat a_{jj}}
-\sigma_l(G)\sum_{\stackrel{i,j,k\in G }{i\ne j,i\ne k,j\ne k}}\frac{\hat a_{ni}^2 }{\hat a_{ii}}
\frac{\hat a_{jj,\alpha }\hat a_{kk,\a}-\hat a_{jk,\alpha }^2}{\hat a_{jj}\hat a_{kk}}\notag\\
=&\sigma _l(G)\sum_{\stackrel{i,j\in G}{i\ne j}}(\hat a_{nn}-\sum_{\stackrel{k \in G}{k \ne i,k \ne j}}\frac{\hat a_{nk}^2 }{\hat a_{kk}})
\frac{\hat a_{ii,\alpha }\hat a_{jj,\a}-\hat a_{ij,\alpha }^2}{\hat a_{ii}\hat a_{jj}}\notag\\
\sim&\sigma_l(G)\sum_{\stackrel{i,j\in G}{i\ne j}}[\frac{\hat a_{ni}^2 }{\hat a_{ii}}+\frac{\hat a_{nj}^2 }{\hat a_{jj}}]
\frac{\hat a_{ii,\alpha }\hat a_{jj,\a}-\hat a_{ij,\alpha }^2}{\hat a_{ii}\hat a_{jj}}\notag\\
=&2\sigma_l(G)\sum_{i,j\in G}\frac{\hat a_{ni}^2 }{\hat a_{ii}^2}
\hat a_{ii,\alpha }\frac{\hat a_{jj,\a}}{\hat a_{jj}}
-2\sigma_l(G)\sum_{i,j\in G}\frac{\hat a_{ni}^2 }{\hat a_{ii}^2}\frac{\hat a_{ij,\alpha }^2}{\hat a_{jj}}.
\end{align}

For $III_{\alpha}$, it follows that
\begin{align} \label{6.55}
III_{\alpha}=&2\hat a_{nn,\alpha}
[\sum_{i\in G}\sigma_{l-1}(G|i)\hat a_{ii,\a}+\sigma_{l-1}(G)\sum_{i\in B}\hat a_{ii,\a}] \notag \\
&-4\sum_{i\in G}\hat a_{ni}\hat a_{ni,\alpha}[\sum_{\stackrel{j\in G}{j\ne i}}\sigma_{l-2}(G|ij)\hat a_{jj,\a}
+\sigma_{l-2}(G|i)\sum_{j\in B}\hat a_{jj,\a}] \notag \\
&+2\sum_{\stackrel{i,j\in G}{i\ne j}}\hat a_{ni}\hat a_{nj}\hat a_{ij,\alpha}
[\sum_{\stackrel{k\in G}{k\ne i,k\ne j}}\sigma_{l-3}(G|ijk)\hat a_{kk,\a}+\sigma_{l-3}(G|ij)\sum_{k\in B}\hat a_{kk,\a}] \notag\\
=&2\hat a_{nn,\alpha} \sum_{i\in G}\sigma_{l-1}(G|i)\hat a_{ii,\a}
-4\sum_{i\in G}\hat a_{ni}\hat a_{ni,\alpha}\sum_{\stackrel{j\in G}{j\ne i}}\sigma_{l-2}(G|ij)\hat a_{jj,\a}
\notag \\
&+2\sum_{\stackrel{i,j\in G}{i\ne j}}\hat a_{ni}\hat a_{nj}\hat a_{ij,\alpha}
\sum_{\stackrel{k\in G}{k\ne i,k\ne j}}\sigma_{l-3}(G|ijk)\hat a_{kk,\a}\notag\\
=&2\sigma_l(G)\hat a_{nn,\alpha}\sum_{i\in G}\frac{\hat a_{ii,\a}}{\hat a_{ii}}
-4\sigma_l(G)\sum_{\stackrel{i,j\in G}{i\ne j}}\frac{\hat a_{ni}}{\hat a_{ii}}\hat a_{ni,\alpha}
\frac{\hat a_{jj,\a}}{\hat a_{jj}}\notag\\
&+2\sigma_{l}(G)\sum_{\stackrel{i,j,k\in G}{i\ne j,i\ne k,j\ne k}}
\frac{\hat a_{ni}}{\hat a_{ii}}\frac{\hat a_{nj}}{\hat a_{jj}}\frac{\hat a_{kk,\a}}{\hat a_{kk}}\hat a_{ij,\alpha}.
\end{align}
Now we need to expand the sum of the third term in \eqref{6.55}
\begin{align}\label{6.56}
\sum_{\stackrel{i,j,k\in G}{i\ne j,i\ne k,j\ne k}}=\sum_{i,j,k\in G}-\sum_{\stackrel{i,j,k\in G}{i=j}}
-\sum_{\stackrel{i,j,k\in G}{i=k}}-\sum_{\stackrel{i,j,k\in G}{j=k}}
+2\sum_{\stackrel{i,j,k\in G}{i=j=k}},
\end{align}
so the third term in \eqref{6.55} becomes
\begin{align}\label{6.57}
\sigma_{l}(G)\sum_{\stackrel{i,j,k\in G}{i\ne j,i\ne k,j\ne k}}
\frac{\hat a_{ni}}{\hat a_{ii}}\frac{\hat a_{nj}}{\hat a_{jj}}\frac{\hat a_{kk,\a}}{\hat a_{kk}}\hat a_{ij,\alpha}
=&\sigma_{l}(G)\sum_{i,j,k\in G}\frac{\hat a_{ni}}{\hat a_{ii}}\frac{\hat a_{nj}}{\hat a_{jj}}\frac{\hat a_{kk,\a}}{\hat a_{kk}}\hat a_{ij,\alpha}
-\sigma_{l}(G)\sum_{i,j\in G}\frac{\hat a_{ni}^2}{\hat a_{ii}^2}\frac{\hat a_{jj,\a}}{\hat a_{jj}}\hat a_{ii,\alpha}\notag\\
&-2\sigma_{l}(G)\sum_{i,j\in G}\frac{\hat a_{ni}}{\hat a_{ii}}\frac{\hat a_{nj}}{\hat a_{jj}}\frac{\hat a_{ii,\a}}{\hat a_{ii}}\hat a_{ij,\alpha}
+2\sigma_{l}(G)\sum_{i\in G}\frac{\hat a_{ni}^2}{\hat a_{ii}^2}\frac{\hat a_{ii,\a}}{\hat a_{ii}}\hat a_{ii,\alpha}.
\end{align}
Combining \eqref{6.55} and \eqref{6.57}, it yields that
\begin{align}\label{6.58}
III_{\alpha}=&2\sigma_l(G)\sum_{i\in G}\frac{\hat a_{ii,\a}}{\hat a_{ii}}(\hat a_{nn,\alpha}
-2\sum_{\stackrel{j\in G}{j\ne i}}\frac{\hat a_{nj}}{\hat a_{jj}}\hat a_{nj,\alpha}
+\sum_{j,k\in G}\frac{\hat a_{nj}}{\hat a_{jj}}\frac{\hat a_{nk}}{\hat a_{kk}}\hat a_{jk,\alpha})\notag\\
&-2\sigma_{l}(G)\sum_{i,j\in G}\frac{\hat a_{ni}^2}{\hat a_{ii}^2}\frac{\hat a_{jj,\a}}{\hat a_{jj}}\hat a_{ii,\alpha}
-4\sigma_{l}(G)\sum_{i,j\in G}\frac{\hat a_{ni}}{\hat a_{ii}}\frac{\hat a_{nj}}{\hat a_{jj}}\frac{\hat a_{ii,\a}}{\hat a_{ii}}\hat a_{ij,\alpha}
+4\sigma_{l}(G)\sum_{i\in G}\frac{\hat a_{ni}^2}{\hat a_{ii}^2}\frac{\hat a_{ii,\a}}{\hat a_{ii}}\hat a_{ii,\alpha}\notag\\
\sim&4\sigma_l(G)\sum_{i\in G}\frac{\hat a_{ni}}{\hat a_{ii}}\frac{\hat a_{ii,\a}}{\hat a_{ii}}[\hat a_{ni,\alpha}
-\sum_{j\in G}\frac{\hat a_{nj}}{\hat a_{jj}}\hat a_{ij,\alpha} ]\notag\\
&-2\sigma_{l}(G)\sum_{i,j\in G}\frac{\hat a_{ni}^2}{\hat a_{ii}^2}\frac{\hat a_{jj,\a}}{\hat a_{jj}}\hat a_{ii,\alpha}
+4\sigma_{l}(G)\sum_{i\in G}\frac{\hat a_{ni}^2}{\hat a_{ii}^2}\frac{\hat a_{ii,\a}}{\hat a_{ii}}\hat a_{ii,\alpha}.
\end{align}

At last, we deal with $IV_{\alpha}$,
\begin{align*}
IV_{\alpha}=&-2[\sum_{i\in G}\sigma _{l - 1} (G|i)\hat a_{ni,\alpha }\hat a_{ni,\a}
+\sigma_{l-1}(G)\sum_{i\in B}\hat a_{ni,\alpha }\hat a_{ni,\a}]\\
&+4\sum_{j\in G}\hat a_{nj}[\sum_{\stackrel{i\in G}{i\ne j}}\sigma_{l-2}(G|ij)\hat a_{ni,\alpha}\hat a_{ij,\a}
+\sigma_{l-2}(G|j)\sum_{i\in B}\hat a_{ni,\alpha}\hat a_{ij,\a}]\\
&-2\sum_{\stackrel{i,j\in G}{i\ne j}}\hat a_{ni}\hat a_{nj}[\sum_{\stackrel{k\in G}{k\ne i,k\ne j}}
\sigma_{l-3}(G|ijk)\hat a_{ki,\alpha}\hat a_{kj,\a}+\sigma_{l-3}(G|ij)\sum_{k\in B}\hat a_{ki,\alpha}\hat a_{kj,\a}]\\
=&-2\sum_{i\in G}\sigma_{l-1}(G|i)\hat a_{ni,\alpha }\hat a_{ni,\a}
+4\sum_{\stackrel{i,j\in G}{i\ne j}}\sigma_{l-2}(G|ij)\hat a_{nj}\hat a_{ni,\alpha}\hat a_{ij,\a}\notag\\
&-2\sum_{\stackrel{i,j,k\in G}{i\ne j,i\ne k,j\ne k}}\sigma_{l-3}(G|ijk)\hat a_{ni}\hat a_{nj}\hat a_{ki,\alpha}\hat a_{kj,\a}.\notag
\end{align*}
By the decomposition \eqref{6.56}, we have
\begin{align}\label{6.59}
IV_{\alpha}=&-2\sigma_{l}(G)\sum_{i\in G}\frac{\hat a_{ni,\alpha }^2}{\hat a_{ii}}
+4\sigma_{l}(G)\sum_{\stackrel{i,j\in G}{i\ne j}}\frac{\hat a_{nj}}{\hat a_{jj}}\frac{\hat a_{ij,\a}}{\hat a_{ii}}a_{ni,\alpha}
-2\sigma_{l}(G)\sum_{\stackrel{i,j,k\in G}{i\ne j,i\ne k,j\ne k}}\frac{\hat a_{ni}}{\hat a_{ii}}\frac{\hat a_{nj}}{\hat a_{jj}}
\frac{\hat a_{ki,\alpha}\hat a_{kj,\a}}{\hat a_{kk}}\notag\\
=&-2\sigma_{l}(G)\sum_{i\in G}\frac{\hat a_{ni,\alpha }^2}{\hat a_{ii}}+4\sigma_{l}(G)\sum_{i,j\in G}\frac{\hat a_{nj}}{\hat a_{jj}}\frac{\hat a_{ij,\a}}{\hat a_{ii}}a_{ni,\alpha}
-4\sigma_{l}(G)\sum_{i\in G}\frac{\hat a_{ni}}{\hat a_{ii}}\frac{\hat a_{ii,\a}}{\hat a_{ii}}a_{ni,\alpha} \notag\\
&-2\sigma_{l}(G)\sum_{i,j,k\in G}\frac{\hat a_{ni}}{\hat a_{ii}}\frac{\hat a_{nj}}{\hat a_{jj}}
\frac{\hat a_{ki,\alpha}\hat a_{kj,\a}}{\hat a_{kk}}
+2\sigma_{l}(G)\sum_{i,j\in G}\frac{\hat a_{ni}^2}{\hat a_{ii}^2}\frac{\hat a_{ij,\alpha}\hat a_{ij,\a}}{\hat a_{jj}}
\notag\\
&+2\sigma_{l}(G)\sum_{i,j\in G}\frac{\hat a_{ni}}{\hat a_{ii}}\frac{\hat a_{nj}}{\hat a_{jj}}
\frac{\hat a_{ii,\alpha}\hat a_{ij,\a}}{\hat a_{ii}}+2\sigma_{l}(G)\sum_{i,j\in G}\frac{\hat a_{ni}}{\hat a_{ii}}\frac{\hat a_{nj}}{\hat a_{jj}}
\frac{\hat a_{ij,\alpha}\hat a_{jj,\a}}{\hat a_{jj}}\notag\\
&-4\sigma_{l}(G)\sum_{i\in G}\frac{\hat a_{ni}^2}{\hat a_{ii}^2}\frac{\hat a_{ii,\alpha}\hat a_{ii,\a}}{\hat a_{ii}}.
\end{align}
Combining the terms, it follows
\begin{align}\label{6.60}
IV_{\alpha}=&-2\sigma_l(G)\sum_{i\in G}\frac{1}{\hat a_{ii}}[\hat a_{ni,\alpha }-\sum_{j\in G}\frac{\hat a_{nj}}{\hat a_{jj}}\hat a_{ij,\a}]^2
-4\sigma_l(G)\sum_{i\in G}\frac{\hat a_{ni}}{\hat a_{ii}}\frac{\hat a_{ii,\a}}{\hat a_{ii}}
[\hat a_{ni,\alpha }-\sum_{j\in G}\frac{\hat a_{nj}}{\hat a_{jj}}\hat a_{ij,\a}] \notag \\
&+2\sigma_{l}(G)\sum_{i,j\in G}\frac{\hat a_{ni}^2}{\hat a_{ii}^2}\frac{\hat a_{ij,\alpha}\hat a_{ij,\a}}{\hat a_{jj}}
-4\sigma_{l}(G)\sum_{i\in G}\frac{\hat a_{ni}^2}{\hat a_{ii}^2}\frac{\hat a_{ii,\alpha}\hat a_{ii,\a}}{\hat a_{ii}}.
\end{align}

So from  \eqref{6.54}, \eqref{6.58} and \eqref{6.60}, we get
\begin{align}\label{6.61}
II_{\alpha}+III_{\alpha}+IV_{\alpha}\sim - 2\sigma _l (G)\sum\limits_{i \in G} {\frac{1} {{\hat a _{ii}
}}\left[ {\hat a_{in,\alpha }  - \sum\limits_{j \in G} {\frac{{\hat a_{jn} }}
{{\hat a_{jj} }}\hat a_{ij,\alpha } } } \right]^2}.
\end{align}
The latter completes the proof of Lemma \ref{lem6.5}, jointly with\eqref{6.53}.
\end{proof}

\begin{lemma} \label{lem6.6}
Under the above assumptions and notations, for any $(x,t) \in \mathcal {O} \times (t_0-\delta, t_0]$ with the coordinate \eqref{6.7}, we have
\begin{eqnarray}\label{6.62}
\phi_{\alpha \alpha} \sim &&\sigma _l
(G)\left( {1 + \sum\limits_{i \in G} {\frac{{\hat a_{in} ^2 }} {{\hat a
_{ii}^2 }}} } \right)\Big(-\frac{|u_t|}{|D u|{u_t}^3}\Big)\sum\limits_{m \in B} {\hat h_{mm, \alpha \alpha} }  \notag \\
&&+ \sigma_l(G) \Big(-\frac{|u_t|}{|D u|{u_t}^3}\Big) \frac{1}{\hat W ^2}\left[ {\hat h_{nn,\alpha \alpha}  - 2\sum\limits_{i \in
G} {\frac{{\hat h_{in} }} {{\hat h_{ii}}}\hat h_{in,\alpha \alpha }+\sum\limits_{i,j
\in G} {\frac{{\hat h_{in} }} {{\hat h_{ii} }}\frac{{\hat h_{jn} }} {{\hat h_{jj}
}}\hat h_{ij,\alpha \alpha} } } } \right] \notag\\
&&- 2\sigma _l (G) \Big(-\frac{|u_t|}{|D u|{u_t}^3}\Big) \frac{1}{\hat W ^2}\sum\limits_{i \in G} {\frac{1} {{\hat h _{ii}
}}\left[ {\hat h_{in,\alpha }  - \sum\limits_{j \in G} {\frac{{\hat h_{jn} }}
{{\hat h_{jj} }}\hat h_{ij,\alpha } } } \right]^2}.
\end{eqnarray}

\end{lemma}

\begin{proof}
From \eqref{6.7} and Lemma \ref{lem6.2}, we have
\begin{align*}
&u_i \hat h_{in}=(u_i \hat h_{in})_{\a}=(u_i \hat h_{in})_{\a\a}=0, \quad i \in B, \a =1,2,\cdots, n-1;\\
&u_i \hat h_{in}=(u_i \hat h_{in})_{n}=0, (u_i \hat h_{in})_{nn}=2u_{in} \hat h_{in,n}=0, \quad i \in B.
\end{align*}
So from \eqref{6.1} and \eqref{6.2}, we have
\begin{align}\label{6.63}
\hat a_{mm, \alpha \alpha} = \Big(-\frac{|u_t|}{|D u|{u_t}^3}\Big) \hat A_{mm, \alpha \alpha}
= \Big(-\frac{|u_t|}{|D u|{u_t}^3}\Big) \hat h_{mm, \alpha \alpha}, \quad m \in B.
\end{align}

From \eqref{6.1} - \eqref{6.4}, \eqref{6.20} and \eqref{6.25}, we have
\begin{eqnarray}
&&\hat a_{nn,\alpha \alpha}  - 2\sum\limits_{i \in
G} {\frac{{\hat a_{in} }} {{\hat a_{ii}}}\hat a_{in,\alpha \alpha }+\sum\limits_{i,j
\in G} {\frac{{\hat a_{in} }} {{\hat a_{ii} }}\frac{{\hat a_{jn} }} {{\hat a_{jj}
}}\hat a_{ij,\alpha \alpha} } } \notag \\
&=& \Big(-\frac{|u_t|}{|D u|{u_t}^3}\Big) \left[ {\hat A_{nn,\alpha \alpha}  - 2\sum\limits_{i \in
G} {\frac{{\hat A_{in} }} {{\hat A_{ii}}}\hat A_{in,\alpha \alpha }+\sum\limits_{i,j
\in G} {\frac{{\hat A_{in} }} {{\hat A_{ii} }}\frac{{\hat A_{jn} }} {{\hat A_{jj}
}}\hat A_{ij,\alpha \alpha} } } } \right] \notag \\
&&+2 \Big(-\frac{|u_t|}{|D u|{u_t}^3}\Big)_{\alpha} \left[ {\hat A_{nn,\alpha}  - 2\sum\limits_{i \in
G} {\frac{{\hat A_{in} }} {{\hat A_{ii}}}\hat A_{in,\alpha }+\sum\limits_{i,j
\in G} {\frac{{\hat A_{in} }} {{\hat A_{ii} }}\frac{{\hat A_{jn} }} {{\hat A_{jj}
}}\hat A_{ij,\alpha} } } } \right] \notag \\
&&+\Big(-\frac{|u_t|}{|D u|{u_t}^3}\Big)_{\alpha \alpha} \left[ {\hat A_{nn}  - 2\sum\limits_{i \in
G} {\frac{{\hat A_{in} }} {{\hat A_{ii}}}\hat A_{in }+\sum\limits_{i,j
\in G} {\frac{{\hat A_{in} }} {{\hat A_{ii} }}\frac{{\hat A_{jn} }} {{\hat A_{jj}
}}\hat A_{ij} } } } \right] \notag \\
&\sim& \Big(-\frac{|u_t|}{|D u|{u_t}^3}\Big) \left[ {\hat A_{nn,\alpha \alpha}  - 2\sum\limits_{i \in
G} {\frac{{\hat A_{in} }} {{\hat A_{ii}}}\hat A_{in,\alpha \alpha }+\sum\limits_{i,j
\in G} {\frac{{\hat A_{in} }} {{\hat A_{ii} }}\frac{{\hat A_{jn} }} {{\hat A_{jj}
}}\hat A_{ij,\alpha \alpha} } } } \right],\notag
\end{eqnarray}
and
\begin{eqnarray}
&&- 2\sigma _l (G) \sum\limits_{i \in G} {\frac{1} {{\hat a _{ii}
}}\left[ {\hat a_{in,\alpha }  - \sum\limits_{j \in G} {\frac{{\hat a_{jn} }}
{{\hat a_{jj} }}\hat a_{ij,\alpha } } } \right]^2}   \notag \\
&=&- 2\sigma _l (G) \sum\limits_{i \in G} {\frac{1} {{\hat a _{ii}
}}\left[\Big(-\frac{|u_t|}{|D u|{u_t}^3}\Big)\Big( {\hat A_{in,\alpha }  - \sum\limits_{j \in G} {\frac{{\hat A_{jn} }}
{{\hat A_{jj} }}\hat A_{ij,\alpha } } }\Big)+\Big(-\frac{|u_t|}{|D u|{u_t}^3}\Big)_{\alpha }\Big( {\hat A_{in}  - \sum\limits_{j \in G} {\frac{{\hat A_{jn} }}
{{\hat A_{jj} }}\hat A_{ij} } }\Big) \right]^2}   \notag \\
&=&- 2\sigma _l (G) \Big(-\frac{|u_t|}{|D u|{u_t}^3}\Big) \sum\limits_{i \in G} {\frac{1} {{\hat A _{ii}
}}\left[ {\hat A_{in,\alpha }  - \sum\limits_{j \in G} {\frac{{\hat A_{jn} }}
{{\hat A_{jj} }}\hat A_{ij,\alpha } } } \right]^2}. \notag
\end{eqnarray}
So we get
\begin{align} \label{6.64}
&\sigma_l(G)\left[ {\hat a_{nn,\alpha \alpha}  - 2\sum\limits_{i \in
G} {\frac{{\hat a_{in} }} {{\hat a_{ii}}}\hat a_{in,\alpha \alpha }+\sum\limits_{ij
\in G} {\frac{{\hat a_{in} }} {{\hat a_{ii} }}\frac{{\hat a_{jn} }} {{\hat a_{jj}
}}\hat a_{ij,\alpha \alpha} } } } \right]- 2\sigma _l (G) \sum\limits_{i \in G} {\frac{1} {{\hat a _{ii}
}}\left[ {\hat a_{in,\alpha }  - \sum\limits_{j \in G} {\frac{{\hat a_{jn} }}
{{\hat a_{jj} }}\hat a_{ij,\alpha } } } \right]^2}   \notag \\
\sim&\sigma_l(G) \Big(-\frac{|u_t|}{|D u|{u_t}^3}\Big) \left[ {\hat A_{nn,\alpha \alpha}  - 2\sum\limits_{i \in
G} {\frac{{\hat A_{in} }} {{\hat A_{ii}}}\hat A_{in,\alpha \alpha }+\sum\limits_{i,j
\in G} {\frac{{\hat A_{in} }} {{\hat A_{ii} }}\frac{{\hat A_{jn} }} {{\hat A_{jj}
}}\hat A_{ij,\alpha \alpha} } } } \right]\notag \\
&- 2\sigma _l (G) \Big(-\frac{|u_t|}{|D u|{u_t}^3}\Big) \sum\limits_{i \in G} {\frac{1} {{\hat A _{ii}
}}\left[ {\hat A_{in,\alpha }  - \sum\limits_{j \in G} {\frac{{\hat A_{jn} }}
{{\hat A_{jj} }}\hat A_{ij,\alpha } } } \right]^2}.
\end{align}
In the following, we will prove
\begin{eqnarray} \label{6.65}
&& \left[ {\hat A_{nn,\alpha \alpha}  - 2\sum\limits_{i \in
G} {\frac{{\hat A_{in} }} {{\hat A_{ii}}}\hat A_{in,\alpha \alpha }+\sum\limits_{i,j
\in G} {\frac{{\hat A_{in} }} {{\hat A_{ii} }}\frac{{\hat A_{jn} }} {{\hat A_{jj}
}}\hat A_{ij,\alpha \alpha} } } } \right]
- 2\sum\limits_{i \in G} {\frac{1} {{\hat A _{ii}
}}\left[ {\hat A_{in,\alpha }  - \sum\limits_{j \in G} {\frac{{\hat A_{jn} }}
{{\hat A_{jj} }}\hat A_{ij,\alpha } } } \right]^2}  \notag  \\
&\sim& \frac{1}{\hat W ^2}\left[ {\hat h_{nn,\alpha \alpha}  - 2\sum\limits_{i \in
G} {\frac{{\hat h_{in} }} {{\hat h_{ii}}}\hat h_{in,\alpha \alpha }+\sum\limits_{i,j
\in G} {\frac{{\hat h_{in} }} {{\hat h_{ii} }}\frac{{\hat h_{jn} }} {{\hat h_{jj}
}}\hat h_{ij,\alpha \alpha} } } } \right]\notag \\
&&- \frac{2}{\hat W ^2}\sum\limits_{i \in G} {\frac{1} {{\hat h _{ii}
}}\left[ {\hat h_{in,\alpha }  - \sum\limits_{j \in G} {\frac{{\hat h_{jn} }}
{{\hat h_{jj} }}\hat h_{ij,\alpha } } } \right]^2}.
\end{eqnarray}
If \eqref{6.65} holds, \eqref{6.62} holds from \eqref{6.51}, \eqref{6.52}.

From \eqref{6.2} and \eqref{6.3}, taking the first derivatives of $\hat A_{\a \b}$, we get
\begin{align}
\hat A_{in,\a}
=&\Big(\frac{1}{\hat W}\Big)_{\a}\hat h_{in}+\frac{1}{\hat W}\hat h_{in,\a}
-\frac{u_{i\a}u_n \hat h_{nn}}{\hat W^2(1+\hat W)u_t^2} -\frac{u_n \sum_{l=1}^{n-1}u_{l\a}\hat h_{il}}{\hat W(1+\hat W)u_t^2} \notag \\
=&\Big(\frac{1}{\hat W}\Big)_{\a}\hat h_{in}+\frac{1}{\hat W}\hat h_{in,\a}
-\frac{u_{i\a}u_n \hat h_{nn}}{\hat W^2(1+\hat W)u_t^2} -\frac{u_n u_{i\a}\hat h_{ii}}{\hat W(1+\hat W)u_t^2}, \notag
\end{align}
and
\begin{align}
- \sum\limits_{j \in G} {\frac{{\hat A_{jn} }}
{{\hat A_{jj} }}\hat A_{ij,\alpha } } =& - \frac{1}{\hat W}\sum\limits_{j \in G} {\frac{{\hat h_{jn} }}
{{\hat h_{jj} }}[\hat h_{ij,\a}-\frac{u_{i\a}u_n \hat h_{jn}}{\hat W(1+\hat W)u_t^2} -\frac{u_{j\a}u_n \hat h_{in}}{\hat W(1+ \hat W)u_t^2}] } \notag \\
=& - \frac{1}{\hat W}\sum\limits_{j \in G} {\frac{{\hat h_{jn} }}
{{\hat h_{jj} }}\hat h_{ij,\a}}+\frac{u_{i\a}u_n \hat h_{nn}}{\hat W^2(1+\hat W)u_t^2}
+\frac{u_n \hat h_{in}}{\hat W^2(1+ \hat W)u_t^2}\sum\limits_{j \in G} {\frac{{u_{j\a} \hat h_{jn}}}
{{\hat h_{jj} }} }, \notag
\end{align}
so
\begin{align}\label{6.66}
&- 2\sum\limits_{i \in G} {\frac{1} {{\hat A _{ii}
}}\left[ {\hat A_{in,\alpha }  - \sum\limits_{j \in G} {\frac{{\hat A_{jn} }}
{{\hat A_{jj} }}\hat A_{ij,\alpha } } } \right]^2}  \notag  \\
=&- 2\sum\limits_{i \in G} {\frac{1} {{\hat h _{ii}
}}\left[\frac{1}{\hat W}\Big(\hat h_{in,\a}-\sum\limits_{j \in G} {\frac{{\hat h_{jn} }}
{{\hat h_{jj} }}\hat h_{ij,\a}}\Big)+\Big(\frac{1}{\hat W}\Big)_{\a}\hat h_{in}
-\frac{u_n u_{i\a}\hat h_{ii}}{\hat W(1+\hat W)u_t^2}+\frac{u_n \hat h_{in}}{\hat W^2(1+ \hat W)u_t^2}\sum\limits_{j \in G} {\frac{{u_{j\a} \hat h_{jn}}}
{{\hat h_{jj} }} } \right]^2 }\notag  \\
=&- \frac{2}{\hat W^2}\sum\limits_{i \in G} {\frac{1} {{\hat h _{ii}
}}\Big(\hat h_{in,\a}-\sum\limits_{j \in G} {\frac{{\hat h_{jn} }}
{{\hat h_{jj} }}\hat h_{ij,\a}}\Big)^2 }- 4\frac{1}{\hat W}\Big(\frac{1}{\hat W}\Big)_{\a}\sum\limits_{i \in G} {\frac{\hat h_{in}} {{\hat h _{ii}
}}\Big(\hat h_{in,\a}-\sum\limits_{j \in G} {\frac{{\hat h_{jn} }}
{{\hat h_{jj} }}\hat h_{ij,\a}}\Big) }\notag  \\
&+ \frac{4u_n} {\hat W^2(1+\hat W)u_t^2}\sum\limits_{i \in G} {
\Big(\hat h_{in,\a}-\sum\limits_{j \in G} {\frac{{\hat h_{jn} }}
{{\hat h_{jj} }}\hat h_{ij,\a}}\Big)u_{i\a }} \notag \\
&- \frac{4u_n} {\hat W^3(1+\hat W)u_t^2}\sum\limits_{i \in G} {\frac{\hat h_{in}} {{\hat h _{ii}}}
\Big(\hat h_{in,\a}-\sum\limits_{j \in G} {\frac{{\hat h_{jn} }}
{{\hat h_{jj} }}\hat h_{ij,\a}}\Big)\sum\limits_{k \in G} {\frac{{u_{k\a} \hat h_{kn}}}
{{\hat h_{kk} }} }} \notag \\
&- 2\sum\limits_{i \in G} {\frac{1} {{\hat h _{ii}
}}\left[\Big(\frac{1}{\hat W}\Big)_{\a}\hat h_{in}
-\frac{u_n u_{i\a}\hat h_{ii}}{\hat W(1+\hat W)u_t^2}+\frac{u_n \hat h_{in}}{\hat W^2(1+ \hat W)u_t^2}\sum\limits_{j \in G} {\frac{{u_{j\a} \hat h_{jn}}}
{{\hat h_{jj} }} } \right]^2 }.
\end{align}

Taking the second derivatives of $\hat A_{\a \b}$, we get
\begin{align} \label{6.67}
\hat A_{nn,\a \a} =& \Big(\frac{1}{\hat W^2}\Big)_{\a\a} \hat h_{nn}
+2\Big(\frac{1}{\hat W^2}\Big)_{\a} \hat h_{nn,\a} +\frac{1}{\hat W^2} \hat h_{nn,\a\a} \notag \\
&-2[\frac{u_n\sum_{l=1}^{n-1}u_{l\a\a}\hat h_{nl}}{\hat W^2(1+\hat W)u_t^2}
+2\sum_{l=1}^{n-1}u_{l\a}\Big(\frac{u_n\hat h_{nl}}{\hat W^2(1+\hat W)u_t^2}\Big)_{\a}]\notag \\
&+4\sum_{l=1}^{n-1} u_{l\a}^2\frac{\hat h_{nn}}{\hat W^2(1+\hat W)u_t^2}
+ \frac{ 2\sum_{k,l=1}^{n-1} u_{k\a}u_{l\a} \hat h_{kl}}{\hat W(1+\hat W)u_t^2}[1- \frac{1}{\hat W}] \notag \\
=& \Big(\frac{1}{\hat W^2}\Big)_{\a\a} \hat h_{nn}
+2\Big(\frac{1}{\hat W^2}\Big)_{\a} \hat h_{nn,\a} +\frac{1}{\hat W^2} \hat h_{nn,\a\a}
-2\frac{u_n\sum\limits_{l\in G}u_{l\a\a}\hat h_{nl}}{\hat W^2(1+\hat W)u_t^2}\notag \\
&-4\sum\limits_{l \in G}u_{l\a}[\Big(\frac{u_n}{\hat W(1+\hat W)u_t^2}\Big)_{\a}\frac{\hat h_{nl}}{\hat W}
+\frac{u_n}{\hat W(1+\hat W)u_t^2}\Big(\frac{\hat h_{nl}}{\hat W}\Big)_{\a}] \\
&+\frac{4\sum_{l=1}^{n-1} u_{l\a}^2\hat h_{nn}}{\hat W^2(1+\hat W)u_t^2}
+ \frac{ 2\sum_{l=1}^{n-1} u_{l\a}^2 \hat h_{ll}}{\hat W(1+\hat W)u_t^2}[1- \frac{1}{\hat W}].\notag
\end{align}

\begin{align}\label{6.68}
\hat A_{in,\a\a} =& \Big(\frac{1}{\hat W}\Big)_{\a\a}\hat h_{in}+2\Big(\frac{1}{\hat W}\Big)_{\a}\hat h_{in,\a}+\frac{1}{\hat W}\hat h_{in,\a\a} \notag \\
&-[u_{i\a\a}\frac{u_n \hat h_{nn}}{\hat W^2(1+\hat W)u_t^2} +2u_{i\a}\Big(\frac{u_n \hat h_{nn}}{\hat W^2(1+\hat W)u_t^2}\Big)_{\a}]\notag \\
&-\sum_{l=1}^{n-1}[u_{l\a\a}\frac{u_n \hat h_{il}}{\hat W(1+\hat W)u_t^2} +2u_{l\a}\Big(\frac{u_n \hat h_{il}}{\hat W(1+\hat W)u_t^2}\Big)_{\a} ]\notag \\
&+\frac{2 \sum_{l=1}^{n-1} u_{l\a}^2\hat h_{in}}{\hat W(1+\hat W)u_t^2}
+ \frac{2u_{i\a} \sum_{l=1}^{n-1}u_{l\a}\hat h_{nl}}{\hat W(1+\hat W)u_t^2}[1- \frac{2}{\hat W}] \notag\\
=& \Big(\frac{1}{\hat W}\Big)_{\a\a}\hat h_{in}
+2\Big(\frac{1}{\hat W}\Big)_{\a}\hat h_{in,\a}+\frac{1}{\hat W}\hat h_{in,\a\a}
-u_{i\a\a}\frac{u_n \hat h_{nn}}{\hat W^2(1+\hat W)u_t^2}- u_{i\a\a}\frac{u_n \hat h_{ii}}{\hat W(1+\hat W)u_t^2} \notag \\
&-2u_{i\a}[\Big(\frac{u_n }{\hat W^2(1+\hat W)u_t^2}\Big)_{\a} \hat h_{nn}
+\frac{u_n }{\hat W^2(1+\hat W)u_t^2}\hat h_{nn,\a} ]\notag \\
&-2[u_{i\a}\Big(\frac{u_n }{\hat W(1+\hat W)u_t^2}\Big)_{\a}\hat h_{ii}
+\sum_{l\in G}u_{l\a}\Big(\frac{u_n }{\hat W(1+\hat W)u_t^2}\Big)\hat h_{il,\a}] \\
&+\frac{2 \sum_{l=1}^{n-1} u_{l\a}^2\hat h_{in}}{\hat W(1+\hat W)u_t^2}
+ \frac{2u_{i\a} \sum_{l=1}^{n-1}u_{l\a}\hat h_{nl}}{\hat W(1+\hat W)u_t^2}[1- \frac{2}{\hat W}],\notag
\end{align}
and
\begin{align} \label{6.69}
\hat A_{ij,\a\a} =& \hat h_{ij,\a\a}
-[u_{i\a\a}\frac{u_n \hat h_{jn}}{\hat W(1+\hat W)u_t^2}+2u_{i\a}\Big(\frac{u_n \hat h_{jn}}{\hat W(1+\hat W)u_t^2}\Big)_{\a} ]\notag \\
&-[u_{j\a\a}\frac{u_n \hat h_{in}}{\hat W(1+ \hat W)u_t^2}+2u_{j\a}\Big(\frac{u_n \hat h_{in}}{\hat W(1+ \hat W)u_t^2}\Big)_{\a}] \notag \\
&-2\frac{u_{i \a}\sum_{l=1}^{n-1}u_{l\a} \hat h_{jl}}{\hat W(1+\hat W)u_t^2} -2\frac{u_{j\a} \sum_{l=1}^{n-1} u_{l\a} \hat h_{il}}{\hat W(1+ \hat W)u_t^2}
+ 2\frac{u_{i\a}u_{j\a}u_n^2 \hat h_{nn}}{\hat W^2(1+\hat W)^2u_t^4} \notag \\
=& \hat h_{ij,\a\a}-[u_{i\a\a}\frac{u_n \hat h_{jn}}{\hat W(1+\hat W)u_t^2}
+u_{j\a\a}\frac{u_n \hat h_{in}}{\hat W(1+ \hat W)u_t^2} ]\notag \\
&-2u_{i\a}[\Big(\frac{u_n }{\hat W^2(1+\hat W)u_t^2}\Big)_{\a} \hat W \hat h_{jn}
+\frac{u_n }{\hat W^2(1+\hat W)u_t^2}\Big( \hat W \hat h_{jn}\Big)_{\a}] \notag \\
&-2u_{j\a}[\Big(\frac{u_n }{\hat W^2(1+ \hat W)u_t^2}\Big)_{\a}\hat W \hat h_{in}
+\frac{u_n }{\hat W^2(1+ \hat W)u_t^2}\Big(\hat W \hat h_{in}\Big)_{\a}]  \\
&-2\frac{u_{i\a} u_{j\a} \hat h_{jj}}{\hat W(1+\hat W)u_t^2} -2\frac{u_{j\a}  u_{i\a} \hat h_{ii}}{\hat W(1+ \hat W)u_t^2}
+ 2\frac{u_{i\a}u_{j\a} \hat h_{nn}}{\hat W(1+\hat W)u_t^2} [1-\frac{1}{\hat W}].\notag
\end{align}
Then
\begin{eqnarray}\label{6.70}
\hat A_{nn,\alpha \alpha}  - 2\sum\limits_{i \in
G} {\frac{{\hat A_{in} }} {{\hat A_{ii}}}\hat A_{in,\alpha \alpha }+\sum\limits_{i,j
\in G} {\frac{{\hat A_{in} }} {{\hat A_{ii} }}\frac{{\hat A_{jn} }} {{\hat A_{jj}
}}\hat A_{ij,\alpha \alpha} } }  = I_{\a} + II_{\a} + III_{\a} +IV_{\a},
\end{eqnarray}
where
\begin{align}\label{6.71}
I_{\a}=&\Big(\frac{1}{\hat W^2}\Big)_{\a\a} \hat h_{nn}
+2\Big(\frac{1}{\hat W^2}\Big)_{\a} \hat h_{nn,\a} +\frac{1}{\hat W^2} \hat h_{nn,\a\a} \notag \\
& - 2\frac{1}{\hat W}\sum\limits_{i \in G} {\frac{{\hat h_{in} }} {{\hat h_{ii}}}[\Big(\frac{1}{\hat W}\Big)_{\a\a}\hat h_{in}
+2\Big(\frac{1}{\hat W}\Big)_{\a}\hat h_{in,\a}+\frac{1}{\hat W}\hat h_{in,\a\a}]} \notag \\
&+\frac{1}{\hat W^2}\sum\limits_{i,j\in G} {\frac{{\hat h_{in} }} {{\hat h_{ii} }}\frac{{\hat h_{jn} }} {{\hat h_{jj}
}}\hat h_{ij,\a \a} } \notag \\
\sim&\frac{1}{\hat W^2} [\hat h_{nn,\a\a}
- 2\sum\limits_{i \in G} {\frac{{\hat h_{in} }} {{\hat h_{ii}}}\hat h_{in,\a\a}}
+\sum\limits_{i,j\in G} {\frac{{\hat h_{in} }} {{\hat h_{ii} }}\frac{{\hat h_{jn} }} {{\hat h_{jj}
}}\hat h_{ij,\a \a} }]\notag \\
& + 2\Big(\frac{1}{\hat W}\Big)_{\a}^2\hat h_{nn}
+2\Big(\frac{1}{\hat W^2}\Big)_{\a} [\hat h_{nn,\a}-\sum\limits_{i \in G} {\frac{{\hat h_{in} }} {{\hat h_{ii}}}\hat h_{in,\a}} ]\notag \\
\sim&\frac{1}{\hat W^2} [\hat h_{nn,\a\a}
- 2\sum\limits_{i \in G} {\frac{{\hat h_{in} }} {{\hat h_{ii}}}\hat h_{in,\a\a}}
+\sum\limits_{i,j\in G} {\frac{{\hat h_{in} }} {{\hat h_{ii} }}\frac{{\hat h_{jn} }} {{\hat h_{jj}
}}\hat h_{ij,\a \a} }]\notag \\
& + 2\sum\limits_{i \in G} \frac{{1 }} {{\hat h_{ii}}}[\Big(\frac{1}{\hat W}\Big)_{\a}\hat h_{in}]^2
+4 \sum\limits_{i \in G} \frac{{1 }} {{\hat h_{ii}}}[\Big(\frac{1}{\hat W}\Big)_{\a}\hat h_{in}]
[\frac{1}{\hat W}(\hat h_{in,\a}-\sum\limits_{j \in G} {\frac{{\hat h_{jn} }} {{\hat h_{jj}}}\hat h_{ij,\a}}) ],
\end{align}
and
\begin{eqnarray}\label{6.72}
II_{\a}&=&-2\frac{u_n\sum\limits_{l\in G}u_{l\a\a}\hat h_{nl}}{\hat W^2(1+\hat W)u_t^2}  - 2\frac{1}{\hat W}\sum\limits_{i \in
G} {\frac{{\hat h_{in} }} {{\hat h_{ii}}}[-u_{i\a\a}\frac{u_n \hat h_{nn}}{\hat W^2(1+\hat W)u_t^2}- u_{i\a\a}\frac{u_n \hat h_{ii}}{\hat W(1+\hat W)u_t^2}]} \notag \\
&&+\frac{1}{\hat W^2}\sum\limits_{i,j
\in G} {\frac{{\hat h_{in} }} {{\hat h_{ii} }}\frac{{\hat h_{jn} }} {{\hat h_{jj}
}}[-u_{i\a\a}\frac{u_n \hat h_{jn}}{\hat W(1+\hat W)u_t^2}
-u_{j\a\a}\frac{u_n \hat h_{in}}{\hat W(1+ \hat W)u_t^2} ]}  \notag \\
&\sim&  0.
\end{eqnarray}
For the term $III_{\a}$,
\begin{eqnarray*}\label{6.73}
III_{\a}&=&-4\sum\limits_{l \in G}u_{l\a}[\Big(\frac{u_n}{\hat W(1+\hat W)u_t^2}\Big)_{\a}\frac{\hat h_{nl}}{\hat W}
+\frac{u_n}{\hat W(1+\hat W)u_t^2}\Big(\frac{\hat h_{nl}}{\hat W}\Big)_{\a}] \notag \\
&& - 2\frac{1}{\hat W}\sum\limits_{i \in
G} {\frac{{\hat h_{in} }} {{\hat h_{ii}}}\left\{-2u_{i\a}[\Big(\frac{u_n }{\hat W^2(1+\hat W)u_t^2}\Big)_{\a} \hat h_{nn}
+\frac{u_n }{\hat W^2(1+\hat W)u_t^2}\hat h_{nn,\a} ] \right.}\notag \\
&& \left.\qquad \qquad  \quad -2[u_{i\a}\Big(\frac{u_n }{\hat W(1+\hat W)u_t^2}\Big)_{\a}\hat h_{ii}
+\sum_{l\in G}u_{l\a}\Big(\frac{u_n }{\hat W(1+\hat W)u_t^2}\Big)\hat h_{il,\a}] \right\}  \notag \\
&&+\frac{1}{\hat W^2}\sum\limits_{i,j
\in G} {\frac{{\hat h_{in} }} {{\hat h_{ii} }}\frac{{\hat h_{jn} }} {{\hat h_{jj}
}} \left\{-2u_{i\a}[\Big(\frac{u_n }{\hat W^2(1+\hat W)u_t^2}\Big)_{\a} \hat W \hat h_{jn}
+\frac{u_n }{\hat W^2(1+\hat W)u_t^2}\Big( \hat W \hat h_{jn}\Big)_{\a}]  \right.} \notag  \\
&& \left.\qquad \qquad \qquad -2u_{j\a}[\Big(\frac{u_n }{\hat W^2(1+ \hat W)u_t^2}\Big)_{\a}\hat W \hat h_{in}
+\frac{u_n }{\hat W^2(1+ \hat W)u_t^2}\Big(\hat W \hat h_{in}\Big)_{\a}] \right\}  \notag \\
&\sim&-4\sum\limits_{l \in G}u_{l\a}[\frac{u_n}{\hat W(1+\hat W)u_t^2}\Big(\frac{\hat h_{nl}}{\hat W}\Big)_{\a}] \notag \\
&&-2\frac{1}{\hat W}\sum\limits_{i \in
G} {\frac{{\hat h_{in} }} {{\hat h_{ii}}}\left\{-2 u_{i\a}[\frac{u_n }{\hat W^2(1+\hat W)u_t^2}\hat h_{nn,\a} ]
-2[\sum_{l\in G}u_{l\a}\Big(\frac{u_n }{\hat W(1+\hat W)u_t^2}\Big)\hat h_{il,\a}] \right\} } \notag \\
&&+\frac{1}{\hat W^2}\sum\limits_{i,j
\in G} {\frac{{\hat h_{in} }} {{\hat h_{ii} }}\frac{{\hat h_{jn} }} {{\hat h_{jj}
}} \left\{-2u_{i\a}[\frac{u_n }{\hat W^2(1+\hat W)u_t^2}\Big( \hat W \hat h_{jn}\Big)_{\a}]
-2u_{j\a}[\frac{u_n }{\hat W^2(1+ \hat W)u_t^2}\Big(\hat W \hat h_{in}\Big)_{\a}] \right\}},
\end{eqnarray*}
it follows that
\begin{eqnarray}\label{6.73}
III_{\a}&=& -4\sum\limits_{l \in G}\frac{u_nu_{l\a}}{\hat W(1+\hat W)u_t^2} \Big[\Big(\frac{1}{\hat W}\Big)_{\a}\hat h_{nl}
+\frac{1}{\hat W}\hat h_{nl, \a}\Big] \notag \\
&&+4\sum\limits_{i \in
G} {\frac{{\hat h_{in} }} {{\hat h_{ii}}} \Big[\frac{ u_n u_{i\a}}{\hat W^3(1+\hat W)u_t^2}\hat h_{nn,\a}
+\sum_{l\in G}\frac{ u_n u_{l\a}}{\hat W^2 (1+\hat W)u_t^2}\hat h_{il,\a} \Big] } \notag \\
&&-4\sum\limits_{i,j
\in G} {\frac{{\hat h_{in} }} {{\hat h_{ii} }}\frac{{\hat h_{jn} }} {{\hat h_{jj}
}} \left\{\frac{u_n u_{i\a}}{\hat W^4(1+\hat W)u_t^2} \Big[\hat W_\a \hat h_{jn} +\hat W \hat h_{jn, \a}\Big] \right\}}  \notag \\
&=&4\sum\limits_{i \in G} {\frac{1} {{\hat h _{ii}
}}\Big(\frac{1}{\hat W}\Big)_{\a}\hat h_{in}
\left[-\frac{u_n u_{i\a}\hat h_{ii}}{\hat W(1+\hat W)u_t^2}+\frac{u_n \hat h_{in}}{\hat W^2(1+ \hat W)u_t^2}\sum\limits_{j \in G} {\frac{{u_{j\a} \hat h_{jn}}}
{{\hat h_{jj} }} } \right] } \notag \\
&&+ 4\sum\limits_{i \in G} {\frac{1} {{\hat h _{ii}
}}\frac{1}{\hat W}\Big(\hat h_{in,\a}-\sum\limits_{j \in G} {\frac{{\hat h_{jn} }}
{{\hat h_{jj} }}\hat h_{ij,\a}}\Big)\left[ -\frac{u_n u_{i\a}\hat h_{ii}}{\hat W(1+\hat W)u_t^2}\right] }\notag \\
&&+4\frac{1}{\hat W}\Big(\hat h_{nn,\a}-\sum\limits_{j \in G} {\frac{{\hat h_{jn} }}
{{\hat h_{jj} }}\hat h_{jn,\a}}\Big)\left[\frac{u_n }{\hat W^2(1+ \hat W)u_t^2}\sum\limits_{i \in G} {\frac{{u_{i\a} \hat h_{in}}}
{{\hat h_{ii} }} } \right]   \notag \\
&\sim&4\sum\limits_{i \in G} {\frac{1} {{\hat h _{ii}
}}\Big(\frac{1}{\hat W}\Big)_{\a}\hat h_{in}
\left[-\frac{u_n u_{i\a}\hat h_{ii}}{\hat W(1+\hat W)u_t^2}+\frac{u_n \hat h_{in}}{\hat W^2(1+ \hat W)u_t^2}\sum\limits_{j \in G} {\frac{{u_{j\a} \hat h_{jn}}}
{{\hat h_{jj} }} } \right] } \\
&&+ 4\sum\limits_{i \in G} {\frac{1} {{\hat h _{ii}
}}\frac{1}{\hat W}\Big(\hat h_{in,\a}-\sum\limits_{j \in G} {\frac{{\hat h_{jn} }}
{{\hat h_{jj} }}\hat h_{ij,\a}}\Big)\left[ -\frac{u_n u_{i\a}\hat h_{ii}}{\hat W(1+\hat W)u_t^2}
+ \frac{u_n \hat h_{in}}{\hat W^2(1+ \hat W)u_t^2}\sum\limits_{j \in G} {\frac{{u_{j\a} \hat h_{jn}}}
{{\hat h_{jj} }} }\right] }.\notag
\end{eqnarray}
For the term $IV_{\a}$
\begin{eqnarray*}
IV_{\a}&=&\frac{4\sum_{l=1}^{n-1} u_{l\a}^2\hat h_{nn}}{\hat W^2(1+\hat W)u_t^2}
+ \frac{ 2\sum_{l=1}^{n-1} u_{l\a}^2 \hat h_{ll}}{\hat W(1+\hat W)u_t^2}(1- \frac{1}{\hat W}) \notag \\
&&- 2\frac{1}{\hat W}\sum\limits_{i \in
G} {\frac{{\hat h_{in} }} {{\hat h_{ii}}}[\frac{2 \sum_{l=1}^{n-1} u_{l\a}^2\hat h_{in}}{\hat W(1+\hat W)u_t^2}
+ \frac{2u_{i\a} \sum_{l=1}^{n-1}u_{l\a}\hat h_{nl}}{\hat W(1+\hat W)u_t^2}(1- \frac{2}{\hat W})] }\notag \\
&&+\frac{1}{\hat W^2}\sum\limits_{i,j
\in G} {\frac{{\hat h_{in} }} {{\hat h_{ii} }}\frac{{\hat h_{jn} }} {{\hat h_{jj}
}} [-2\frac{u_{i\a} u_{j\a} \hat h_{jj}}{\hat W(1+\hat W)u_t^2} -2\frac{u_{j\a}  u_{i\a} \hat h_{ii}}{\hat W(1+ \hat W)u_t^2}
+ 2\frac{u_{i\a}u_{j\a} \hat h_{nn}}{\hat W(1+\hat W)u_t^2} (1-\frac{1}{\hat W})] }  \notag \\
&\sim& \frac{ 2\sum\limits_{l\in G} u_{l\a}^2 \hat h_{ll}}{\hat W(1+\hat W)u_t^2}(1- \frac{1}{\hat W}) - 2\frac{1}{\hat W}\sum\limits_{i \in
G} {\frac{{\hat h_{in} }} {{\hat h_{ii}}}[ \frac{2u_{i\a} \sum\limits_{l \in G}u_{l\a}\hat h_{nl}}{\hat W(1+\hat W)u_t^2}(1- \frac{2}{\hat W})] }\notag \\
&&+\frac{1}{\hat W^2}\sum\limits_{i,j
\in G} {\frac{{\hat h_{in} }} {{\hat h_{ii} }}\frac{{\hat h_{jn} }} {{\hat h_{jj}
}} [-4\frac{u_{i\a} u_{j\a} \hat h_{jj}}{\hat W(1+\hat W)u_t^2} + 2\frac{u_{i\a}u_{j\a} \hat h_{nn}}{\hat W(1+\hat W)u_t^2} (1-\frac{1}{\hat W})] },
\end{eqnarray*}
then we have
\begin{eqnarray}\label{6.74}
IV_{\a}
&=& \frac{ 2\sum\limits_{l\in G} u_{l\a}^2 \hat h_{ll}}{\hat W(1+\hat W)u_t^2}(1- \frac{1}{\hat W}) - 4\sum\limits_{i \in
G} {\frac{{\hat h_{in} u_{i\a}}} {{\hat h_{ii}}}[ \frac{ \sum\limits_{l \in G}u_{l\a}\hat h_{nl}}{\hat W^2 (1+\hat W)u_t^2}(1- \frac{1}{\hat W})] }\notag \\
&&+2\sum\limits_{i,j
\in G} {\frac{{\hat h_{in} }} {{\hat h_{ii} }}\frac{{\hat h_{jn} }} {{\hat h_{jj}
}} [\frac{u_{i\a}u_{j\a} \hat h_{nn}}{\hat W^3 (1+\hat W)u_t^2} (1-\frac{1}{\hat W})] } \notag \\
&=& \frac{ 2\sum\limits_{l\in G} u_{l\a}^2 \hat h_{ll}}{\hat W(1+\hat W)u_t^2}(\frac{u_n^2}{\hat W(1+\hat W)u_t^2}) - 4\sum\limits_{i \in
G} {\frac{{\hat h_{in} u_{i\a}}} {{\hat h_{ii}}}[ \frac{ \sum\limits_{l \in G}u_{l\a}\hat h_{nl}}{\hat W^2 (1+\hat W)u_t^2}(\frac{u_n^2}{\hat W(1+\hat W)u_t^2})] }\notag \\
&&+2\sum\limits_{i,j
\in G} {\frac{{\hat h_{in} }} {{\hat h_{ii} }}\frac{{\hat h_{jn} }} {{\hat h_{jj}
}} [\frac{u_{i\a}u_{j\a} \hat h_{nn}}{\hat W^3 (1+\hat W)u_t^2} (\frac{u_n^2}{\hat W(1+\hat W)u_t^2})] } \notag \\
&=&2\sum\limits_{i \in G} {\frac{1} {{\hat h _{ii}
}}\left[-\frac{u_n u_{i\a}\hat h_{ii}}{\hat W(1+\hat W)u_t^2}+\frac{u_n \hat h_{in}}{\hat W^2(1+ \hat W)u_t^2}\sum\limits_{j \in G} {\frac{{u_{j\a} \hat h_{jn}}}{{\hat h_{jj} }} } \right]^2 }.
\end{eqnarray}

By \eqref{6.66}, \eqref{6.70} and \eqref{6.71}-\eqref{6.74}, we can get \eqref{6.65}.
So the lemma holds.
\end{proof}

\subsection{Step 3: proof of Theorem \ref{th1.3}}

\begin{theorem} \label{th6.7}
Under the assumptions of Theorem \ref{th1.1} and the above notations, we
have
\begin{equation}\label{6.75}
\Delta \phi-\phi_t  \le C( \phi+ |\nabla \phi| ), \quad (x,t) \in \mathcal {O}\times (t_0-\delta, t_0].
\end{equation}
So by the strong maximum principle and the method of continuity, Theorem \ref{th1.3} holds.
\end{theorem}

\begin{proof} In fact, if $t_0 = T$ and $(x, t) \in \mathcal {O}\times \{ t_0\}$, we only have \eqref{2.30} instead of \eqref{2.29} from Lemma \ref{lem2.9} ( see Remark \ref{rem2.10} ). So in order to use \eqref{2.29}, we just prove \eqref{6.75} holds for any $(x, t) \in \mathcal {O}\times (t_0-\delta, t_0]$, with a constant $C$ independent of $ dist(\mathcal {O}\times (t_0-\delta,
t_0], \partial(\Omega \times (0, T]) )$ and then by approximation, \eqref{6.75} holds for $t = t_0$.
Then by the strong maximum principle and the method of continuity, we can prove Theorem \ref{th1.3} under CASE 2.

From Lemma \ref{lem6.3} and Lemma \ref{lem6.6}, we have
\begin{align}\label{6.76}
& \Delta \phi-\phi_t \notag\\
\sim &\sigma _l(G)\left( {1 + \sum\limits_{i \in G} {\frac{{\hat a_{in} ^2 }} {{\hat a
_{ii}^2 }}} } \right)\Big(-\frac{|u_t|}{|D u|{u_t}^3}\Big)\sum\limits_{i \in B} {\Delta \hat h_{ii} } \notag \\
&+ \sigma_l(G) \Big(-\frac{|u_t|}{|D u|{u_t}^3}\Big) \frac{1}{\hat W ^2}\left[ {\Big( \Delta \hat h_{nn}-\hat h_{nn,t} \Big)
- 2\sum\limits_{i \in G} {\frac{{\hat h_{in} }} {{\hat h_{ii}}}\Big( \Delta \hat h_{in}-\hat h_{in,t} \Big)+\sum\limits_{i,j
\in G} {\frac{{\hat h_{in} }} {{\hat h_{ii} }}\frac{{\hat h_{jn} }} {{\hat h_{jj}
}}\Big(\Delta \hat h_{ij} -\hat h_{ij,t} \Big)} } } \right] \notag\\
&- 2\sigma _l (G) \Big(-\frac{|u_t|}{|D u|{u_t}^3}\Big) \frac{1}{\hat W ^2}\sum\limits_{i \in G} {\frac{1} {{\hat h _{ii}
}}\sum_{\a=1}^n\left[ {\hat h_{in,\alpha }  - \sum\limits_{j \in G} {\frac{{\hat h_{jn} }}
{{\hat h_{jj} }}\hat h_{ij,\alpha } } } \right]^2}.
\end{align}

First, we have for $i \in B$, $u_{in}=u_{it}= u_{iit} =0$,
\begin{align}\label{6.77}
\Delta \hat h_{ii}  =& u_t ^2 \Delta u_{ii}  + 4u_t u_{\alpha t} u_{ii\alpha }  + 2u_{i\alpha } ^2 u_{tt} \notag \\
&- 2u_t u_{it} \Delta u_i  - 4[u_{i\alpha } u_{t\alpha } u_{it}  + u_{i\alpha } u_t u_{it\alpha } ] \notag\\
=& u_t ^2 \Delta u_{ii}   \notag \\
=&u_t ^2u_{iit}  =0,
\end{align}
then we get
\begin{align}\label{6.78}
\sum\limits_{i \in B} {\Delta \hat h_{ii} }=0.
\end{align}

By \eqref{6.5}, we get
\begin{align*}
 \hat h_{nn,t}  =& 2u_t u_{tt} u_{nn}  + u_t ^2 u_{nnt}  + 2u_n u_{nt} u_{tt}  + u_n ^2 u_{ttt}  \\
  &- 2u_t u_{nt} ^2  - 2u_n u_{tt} u_{nt}  - 2u_n u_t u_{ntt}  \\
  =& u_t ^2 u_{nnt}  + u_n ^2 u_{ttt}  + 2u_t u_{tt} u_{nn}  - 2u_t u_{nt} ^2  - 2u_n u_t u_{ntt},  \\
  \end{align*}
  and
  \begin{align*}
 \Delta \hat h_{nn}  =& u_t ^2 \Delta u_{nn}  + 4u_t u_{\alpha t} u_{nn\alpha }  + 2[u_t \Delta u_t  + u_{\alpha t} ^2 ]u_{nn}  \\
  &+ u_n ^2 \Delta u_{tt}  + 4u_n u_{\alpha n} u_{tt\alpha }  + 2[u_n \Delta u_n  + u_{\alpha n} ^2 ]u_{tt}  \\
  &- 2\Delta u_n u_t u_{nt}  - 2u_n \Delta u_t u_{nt}  - 2u_n u_t \Delta u_{nt}  \\
  &- 4[u_{n\alpha } u_{t\alpha } u_{nt}  + u_{n\alpha } u_t u_{nt\alpha }  + u_n u_{t\alpha } u_{nt\alpha } ] \\
  =& u_t ^2 \Delta u_{nn}  + u_n ^2 \Delta u_{tt}  + 2u_t u_{nn} \Delta u_t  - 2u_n u_t \Delta u_{nt}  - 2u_t u_{nt} \Delta u_n  \\
  &+ 4u_t u_{\alpha t} u_{nn\alpha }  + 4u_n u_{\alpha n} u_{tt\alpha }  + 2u_n u_{tt} \Delta u_n  - 2u_n u_{nt} \Delta u_t  - 4[u_{n\alpha } u_t u_{nt\alpha }  + u_n u_{t\alpha } u_{nt\alpha } ] \\
  &+ 2u_{\alpha t} ^2 u_{nn}  + 2u_{\alpha n} ^2 u_{tt}  - 4u_{n\alpha } u_{t\alpha } u_{nt},
\end{align*}
so
\begin{align} \label{6.79}
\Delta \hat h_{nn}  - \hat h_{nn,t}
=& u_t ^2 [\Delta u_{nn}  - u_{nnt} ] + u_n ^2 [\Delta u_{tt}  - u_{ttt} ] + 2u_t u_{nn} [\Delta u_t  - u_{tt} ]\notag \\
&- 2u_n u_t [\Delta u_{nt}  - u_{ntt} ] - 2u_t u_{nt} [\Delta u_n  - u_{nt} ] \notag\\
&+ 4u_t u_{\alpha t} u_{nn\alpha }  + 4u_n u_{\alpha n} u_{tt\alpha }  + 2u_n u_{tt} \Delta u_n
- 2u_n u_{nt} \Delta u_t  \notag\\
&- 4[u_{n\alpha } u_t u_{nt\alpha }  + u_n u_{t\alpha } u_{nt\alpha } ] \notag\\
&+ 2u_{\alpha t} ^2 u_{nn}  + 2u_{\alpha n} ^2 u_{tt}  - 4u_{n\alpha } u_{t\alpha } u_{nt}  \notag\\
=& 4u_t u_{\alpha t} u_{nn\alpha }  + 4u_n u_{\alpha n} u_{tt\alpha }  + 2u_n u_{tt} u_{nt}
- 2u_n u_{nt} u_{tt}- 4[u_{n\alpha } u_t  + u_n u_{t\alpha } ]u_{nt\alpha } \notag \\
&+ 2u_{\alpha t} ^2 u_{nn}  + 2u_{\alpha n} ^2 u_{tt}  - 4u_{n\alpha } u_{t\alpha } u_{nt} \notag \\
=& 4 u_t u_{\alpha t} u_{nn\alpha }  + 4 u_n u_{\alpha n} u_{tt\alpha }  - 4[u_{n\alpha } u_t  + u_n u_{t\alpha } ]u_{nt\alpha }  \notag\\
&+2 u_{\alpha t} ^2 u_{nn}  + 2 u_{\alpha n} ^2 u_{tt}  - 4u_{n\alpha } u_{t\alpha } u_{nt} \notag \\
=& 4u_t u_{\alpha t} u_{nn\alpha }  + 4u_n u_{\alpha n} u_{tt\alpha }  - 4[u_{n\alpha } u_t  + u_n u_{t\alpha } ]u_{nt\alpha }  \\
&+ 2u_{\alpha t} ^2 u_{nn}  + 2u_{\alpha n} ^2 u_{tt}  - 4u_{n\alpha } u_{t\alpha } u_{nt}. \notag
\end{align}

By \eqref{6.5}, we get for $i\in G$,
\begin{align*}
\hat h_{in,t}  =& 2u_t u_{tt} u_{in}  + u_t ^2 u_{int}  + u_n u_{it} u_{tt}  \\
  &- u_{nt} u_t u_{it}  - u_n u_{tt} u_{it}  - u_n u_t u_{itt}  - u_{it} u_t u_{nt}  \\
  =& 2u_t u_{tt} u_{in}  + u_t ^2 u_{int}  - u_n u_t u_{itt}  - 2u_t u_{nt} u_{it},  \\
  \end{align*}
  and
  \begin{align*}
   \Delta \hat h_{in}  =& u_t ^2 \Delta u_{in}  + 4u_t u_{\alpha t} u_{in\alpha }  + 2[u_t \Delta u_t  + u_{\alpha t} ^2 ]u_{in}  \\
  &+ 2u_{i\alpha } [u_n u_{tt\alpha }  + u_{n\alpha } u_{tt} ] + u_n u_{tt} \Delta u_i  \\
  &- 2u_{i\alpha } [u_t u_{nt\alpha }  + u_{t\alpha } u_{nt} ] - u_t u_{nt} \Delta u_i  \\
  &- \Delta u_n u_t u_{it}  - u_n \Delta u_t u_{it}  - u_n u_t \Delta u_{it}
  - 2[u_{n\alpha } u_{t\alpha } u_{it}  + u_{n\alpha } u_t u_{it\alpha }  + u_n u_{t\alpha } u_{it\alpha } ] \\
  =& u_t ^2 \Delta u_{in}  + 2u_t u_{in} \Delta u_t  - u_n u_t \Delta u_{it}  - u_t u_{nt} \Delta u_i  - u_t u_{it} \Delta u_n  \\
  &+ 4u_t u_{\alpha t} u_{in\alpha }  + 2u_n u_{i\alpha } u_{tt\alpha }  + u_n u_{tt} \Delta u_i  - 2u_t u_{i\alpha } u_{nt\alpha }
   - u_n u_{it} \Delta u_t \notag\\
& - 2[u_{n\alpha } u_t u_{it\alpha }  + u_n u_{t\alpha } u_{it\alpha } ] \\
  &+ 2u_{\alpha t} ^2 u_{in}  + 2u_{i\alpha } u_{n\alpha } u_{tt}  - 2u_{i\alpha } u_{t\alpha } u_{nt}  - 2u_{n\alpha } u_{t\alpha } u_{it},
\end{align*}
so we have
\begin{align} \label{6.80}
 \Delta \hat h_{in}  - \hat h_{in,t}  =& 2u_t u_{in} [\Delta u_t  - u_{tt} ] - u_n u_t [\Delta u_{it}  - u_{itt} ] \notag\\
  &+ 4u_t u_{\alpha t} u_{in\alpha }  + 2u_n u_{i\alpha } u_{tt\alpha }  + u_n u_{tt} [\Delta u_i ] - 2u_t u_{i\alpha } u_{nt\alpha }  \notag\\
  &- u_n u_{it} [\Delta u_t ] - 2[u_{n\alpha } u_t u_{it\alpha }  + u_n u_{t\alpha } u_{it\alpha } ] \notag\\
  &+ 2u_{\alpha t} ^2 u_{in}  + 2u_{i\alpha } u_{n\alpha } u_{tt}  - 2u_{i\alpha } u_{t\alpha } u_{nt}  - 2u_{n\alpha } u_{t\alpha } u_{it}  \notag\\
  =&4u_t u_{\alpha t} u_{in\alpha }  + 2u_n u_{i\alpha } u_{tt\alpha }  + u_n u_{tt} u_{it}  - 2u_t u_{i\alpha } u_{nt\alpha }  \notag\\
  &- u_n u_{it} u_{tt} - 2[u_{n\alpha } u_t u_{it\alpha }  + u_n u_{t\alpha } u_{it\alpha } ] \notag\\
  &+ 2u_{\alpha t} ^2 u_{in}  + 2u_{i\alpha } u_{n\alpha } u_{tt}  - 2u_{i\alpha } u_{t\alpha } u_{nt}  - 2u_{n\alpha } u_{t\alpha } u_{it}  \notag\\
  =&4u_t u_{\alpha t} u_{in\alpha }  + 2u_n u_{i\alpha } u_{tt\alpha }  - 2u_t u_{i\alpha } u_{nt\alpha }
  - 2[u_{n\alpha } u_t  + u_n u_{t\alpha } ]u_{it\alpha }  \notag\\
  &+ 2u_{\alpha t} ^2 u_{in}  + 2u_{i\alpha } u_{n\alpha } u_{tt}  - 2u_{i\alpha } u_{t\alpha } u_{nt}  - 2u_{n\alpha } u_{t\alpha } u_{it}  \notag\\
  =& 4u_t u_{\alpha t} u_{in\alpha }  + 2u_n u_{i\alpha } u_{tt\alpha }  - 2u_t u_{i\alpha } u_{nt\alpha }
  - 2[u_{n\alpha } u_t  + u_n u_{t\alpha } ]u_{it\alpha }  \\
  &+ 2u_{\alpha t} ^2 u_{in}  + 2u_{i\alpha } u_{n\alpha } u_{tt}  - 2u_{i\alpha } u_{t\alpha } u_{nt}  - 2u_{n\alpha } u_{t\alpha } u_{it}. \notag
\end{align}

By \eqref{6.5}, we get for $i, j \in G$,
\begin{align*}
\hat h_{ij,t}  =& 2u_t u_{tt} u_{ij}  + u_t ^2 u_{ijt}  - 2u_t u_{it} u_{jt},  \\
\Delta \hat h_{ij}  =& u_t ^2 \Delta u_{ij}  + 4u_t u_{\alpha t} u_{ij\alpha }  + 2[u_t \Delta u_t  + u_{\alpha t} ^2 ]u_{ij}  \\
&+ 2u_{i\alpha } u_{j\alpha } u_{tt}  - \Delta u_i u_t u_{jt}  - 2[u_{i\alpha } u_{t\alpha } u_{jt}  + u_{i\alpha } u_t u_{jt\alpha } ] \\
&- \Delta u_j u_t u_{it}  - 2[u_{j\alpha } u_{t\alpha } u_{it}  + u_{j\alpha } u_t u_{it\alpha } ] \\
=& u_t ^2 \Delta u_{ij}  + 2u_t u_{ij} \Delta u_t  - u_t u_{jt} \Delta u_i  - u_t u_{it} \Delta u_j  \\
&+ 4u_t u_{\alpha t} u_{ij\alpha }  - 2u_{i\alpha } u_t u_{jt\alpha }  - 2u_{j\alpha } u_t u_{it\alpha }  \\
&+ 2u_{\alpha t} ^2 u_{ij}  + 2u_{i\alpha } u_{j\alpha } u_{tt}  - 2u_{i\alpha } u_{t\alpha } u_{jt}  - 2u_{j\alpha } u_{t\alpha } u_{it},
\end{align*}
so
\begin{align}  \label{6.81}
\Delta \hat h_{ij}  - \hat h_{ij,t}  =& 2u_t u_{ij} [\Delta u_t  - u_{tt} ] \notag\\
&+ 4u_t u_{\alpha t} u_{ij\alpha }  - 2u_{i\alpha } u_t u_{jt\alpha }  - 2u_{j\alpha } u_t u_{it\alpha }  \notag\\
&+ 2u_{\alpha t} ^2 u_{ij}  + 2u_{i\alpha } u_{j\alpha } u_{tt}  - 2u_{t\alpha } [u_{i\alpha } u_{jt}  + u_{j\alpha } u_{it} ] \notag\\
=& 4u_t u_{\alpha t} u_{ij\alpha }  - 2u_{i\alpha } u_t u_{jt\alpha }  - 2u_{j\alpha } u_t u_{it\alpha }  \notag\\
&+ 2u_{\alpha t} ^2 u_{ij}  + 2u_{i\alpha } u_{j\alpha } u_{tt}  - 2u_{t\alpha } [u_{i\alpha } u_{jt}  + u_{j\alpha } u_{it} ] \notag\\
=& 4u_t u_{\alpha t} u_{ij\alpha }  - 2u_{i\alpha } u_t u_{jt\alpha }  - 2u_{j\alpha } u_t u_{it\alpha }  \\
&+ 2u_{\alpha t} ^2 u_{ij}  + 2u_{i\alpha } u_{j\alpha } u_{tt}  -2u_{t\alpha } [u_{i\alpha } u_{jt}  + u_{j\alpha } u_{it} ].\notag
\end{align}

Then we get from \eqref{6.79} - \eqref{6.81}
\begin{align} \label{6.82}
\Big( \Delta \hat h_{nn}-\hat h_{nn,t} \Big)
- 2\sum\limits_{i \in G} {\frac{{\hat h_{in} }} {{\hat h_{ii}}}\Big( \Delta \hat h_{in}-\hat h_{in,t} \Big) }+\sum\limits_{i,j
\in G} {\frac{{\hat h_{in} }} {{\hat h_{ii} }}\frac{{\hat h_{jn} }} {{\hat h_{jj}
}}\Big( \Delta \hat h_{ij} -\hat h_{ij,t} \Big)} =: 2[I+II],
\end{align}
where
\begin{align}\label{6.83}
I=&\Big( 2u_t u_{\alpha t} u_{nn\alpha }  + 2u_n u_{\alpha n} u_{tt\alpha }  - 2[u_{n\alpha } u_t  + u_n u_{t\alpha } ]u_{nt\alpha }   \Big)  \notag \\
&- 2\sum\limits_{i \in G} {\frac{{\hat h_{in} }} {{\hat h_{ii}}}\Big(  2u_t u_{\alpha t} u_{in\alpha }  +u_n u_{i\alpha } u_{tt\alpha }  - u_t u_{i\alpha } u_{nt\alpha }  - [u_{n\alpha } u_t  + u_n u_{t\alpha } ]u_{it\alpha }    \Big) }   \notag  \\
&+\sum\limits_{i,j\in G} {\frac{{\hat h_{in} }} {{\hat h_{ii} }}\frac{{\hat h_{jn} }} {{\hat h_{jj}
}}\Big( 2u_t u_{\alpha t} u_{ij\alpha }  - u_{i\alpha } u_t u_{jt\alpha }  - u_{j\alpha } u_t u_{it\alpha }    \Big)},
\end{align}
and
\begin{align}\label{6.85}
II=&\Big( u_{\alpha t} ^2 u_{nn}  + u_{\alpha n} ^2 u_{tt}  - 2u_{n\alpha } u_{t\alpha } u_{nt}    \Big)  \notag \\
&- 2\sum\limits_{i \in G} {\frac{{\hat h_{in} }} {{\hat h_{ii}}}\Big(  u_{\alpha t} ^2 u_{in}  + u_{i\alpha } u_{n\alpha } u_{tt}  - u_{i\alpha } u_{t\alpha } u_{nt}  -u_{n\alpha } u_{t\alpha } u_{it} \Big) }   \notag  \\
&+\sum\limits_{i,j\in G} {\frac{{\hat h_{in} }} {{\hat h_{ii} }}\frac{{\hat h_{jn} }} {{\hat h_{jj}
}}\Big( u_{\alpha t} ^2 u_{ij}  + u_{i\alpha } u_{j\alpha } u_{tt}  - u_{t\alpha } [u_{i\alpha } u_{jt}  + u_{j\alpha } u_{it} ]   \Big)}.
\end{align}

In the following, we will deal with $I$ and $II$. For $I$, we have
\begin{align*}
I=&u_t^2 u_{nnn} [2 \frac{u_{nt}}{u_t} ]  + \sum_{i \in G} u_t^2 u_{nni} [2 \frac{u_{it}}{u_t}-4\frac{{\hat h_{in} }}{{\hat h_{ii} }}\frac{{u_{nt} }}{{u_t }} ]  +u_n^2 u_{ttn} [2 \frac{u_{nn}}{u_n}-2 \sum_{i \in G} \frac{{\hat h_{in} }}{{\hat h_{ii} }}\frac{{u_{in} }}{{u_n }} ] \notag \\
&+\sum_{ i \in G} u_n^2 u_{tti} [2 \frac{u_{ni}}{u_n}-2  \frac{{\hat h_{in} }}{{\hat h_{ii} }}\frac{{u_{ii} }}{{u_n }} ] + u_n u_t u_{nnt} [-2(\frac{{u_{nn} }}{{u_n }}+\frac{{u_{nt} }}{{u_t }})+2\sum_{i \in G} \frac{{\hat h_{in} }}{{\hat h_{ii} }}\frac{{u_{in} }}{{u_n }}  ]   \notag \\
&+\sum_{i \in G} u_n u_t u_{int} [-2(\frac{{u_{in} }}{{u_n }}+\frac{{u_{it} }}{{u_t }})+2\frac{{\hat h_{in} }}{{\hat h_{ii} }}(\frac{{u_{nn} }}{{u_n }}+\frac{{u_{nt} }}{{u_t }})+2 \frac{{\hat h_{in} }}{{\hat h_{ii} }}\frac{{u_{ii} }}{{u_n }} -2\frac{{\hat h_{in} }}{{\hat h_{ii} }}\sum_{j \in G} \frac{{\hat h_{jn} }}{{\hat h_{jj} }}\frac{{u_{jn} }}{{u_n }}  ]   \notag \\
&+ \sum_{i,j\in G} u_t^2 u_{ijn} [-4\frac{{\hat h_{in} }}{{\hat h_{ii} }}\frac{{u_{jt} }}{{u_t }} +2\frac{{\hat h_{in} }}{{\hat h_{ii} }} \frac{{\hat h_{jn} }}{{\hat h_{jj} }}\frac{{u_{nt} }}{{u_t }}] + \sum_{i,j\in G} u_n u_t u_{ijt} [2\frac{{\hat h_{in} }}{{\hat h_{ii} }}(\frac{{u_{jn} }}{{u_n }}+\frac{{u_{jt} }}{{u_t }}) -  2 \frac{{\hat h_{in} }}{{\hat h_{ii} }} \frac{{\hat h_{jn} }}{{\hat h_{jj} }}\frac{{u_{ii} }}{{u_n }}]   \notag  \\
&+2\sum\limits_{i,j,k\in G} {u_t^2 u_{ijk } \frac{{\hat h_{in} }} {{\hat h_{ii} }}\frac{{\hat h_{jn} }} {{\hat h_{jj}
}}\frac{u_{kt} }{u_t}},
\end{align*}
it follows that
\begin{align}\label{6.86}
I=&u_t^2 u_{nnn} [2 \frac{u_{nt}}{u_t} ]  + \sum_{i \in G} u_t^2 u_{nni} [2 \frac{u_{it}}{u_t}-4\frac{{\hat h_{in} }}{{\hat h_{ii} }}\frac{{u_{nt} }}{{u_t }} ]  +u_n^2 u_{ttn} [2 \frac{u_{nn}}{u_n}-2 \sum_{i \in G} \frac{{\hat h_{in} }}{{\hat h_{ii} }}\frac{{u_{in} }}{{u_n }} ] \notag \\
&+\sum_{ i \in G} u_n^2 u_{tti} [2 \frac{u_{it}}{u_t}] + u_n u_t u_{nnt} [-2(\frac{{u_{nn} }}{{u_n }}+\frac{{u_{nt} }}{{u_t }})+2\sum_{i \in G} \frac{{\hat h_{in} }}{{\hat h_{ii} }}\frac{{u_{in} }}{{u_n }}  ]   \notag \\
&+\sum_{i \in G} u_n u_t u_{int} [-4\frac{{u_{it} }}{{u_t }}+2\frac{{\hat h_{in} }}{{\hat h_{ii} }}(\frac{{u_{nn} }}{{u_n }}+\frac{{u_{nt} }}{{u_t }}) -2\frac{{\hat h_{in} }}{{\hat h_{ii} }}\sum_{j \in G} \frac{{\hat h_{jn} }}{{\hat h_{jj} }}\frac{{u_{jn} }}{{u_n }}  ]   \notag \\
&+ \sum_{i,j\in G} u_t^2 u_{ijn} [-4\frac{{\hat h_{in} }}{{\hat h_{ii} }}\frac{{u_{jt} }}{{u_t }} +2\frac{{\hat h_{in} }}{{\hat h_{ii} }} \frac{{\hat h_{jn} }}{{\hat h_{jj} }}\frac{{u_{nt} }}{{u_t }}] + \sum_{i,j\in G} u_n u_t u_{ijt} [4\frac{{\hat h_{in} }}{{\hat h_{ii} }}\frac{{u_{jt} }}{{u_t }}+2\frac{{\hat h_{in} }}{{\hat h_{ii} }}\frac{{\hat h_{jn} }}{{u_n u_t^2}} -  2 \frac{{\hat h_{in} }}{{u_n u_t^2 }} \frac{{\hat h_{jn} }}{{\hat h_{jj} }}]   \notag  \\
&+2\sum\limits_{i,j,k\in G} {u_t^2 u_{ijk } \frac{{\hat h_{in} }} {{\hat h_{ii} }}\frac{{\hat h_{jn} }} {{\hat h_{jj}
}}\frac{u_{kt} }{u_t}}  \notag \\
=:& A +B,
\end{align}
where $A$ is the terms of $\hat h_{ij,k}$, and $B$ is the sum of the terms of $h_{ij}$. That is
\begin{align*}
A=&(-\sum_{i \in G} \hat h_{ii,n} )[2 \frac{u_{nt}}{u_t} ]  + \sum_{i \in G} (-\sum_{k \in G} \hat h_{kk,i})  [2 \frac{u_{it}}{u_t}-4\frac{{\hat h_{in} }}{{\hat h_{ii} }}\frac{{u_{nt} }}{{u_t }} ]  \notag \\
&+(\hat h_{nn,n}+2\sum_{k \in G} \hat h_{kn,k} -\sum_{k \in G} \hat h_{kk,n} ) [2 \frac{u_{nn}}{u_n}-2 \sum_{i \in G} \frac{{\hat h_{in} }}{{\hat h_{ii} }}\frac{{u_{in} }}{{u_n }} ] \notag \\
&+\sum_{ i \in G} (\hat h_{nn,i} - 2 \hat h_{in,n} -\sum_{k \in G} \hat h_{kk,i} ) [2 \frac{u_{it}}{u_t}] + (\sum_{k \in G} [\hat h_{kn,k} - \hat h_{kk,n}] ) [-2(\frac{{u_{nn} }}{{u_n }}+\frac{{u_{nt} }}{{u_t }})+2\sum_{i \in G} \frac{{\hat h_{in} }}{{\hat h_{ii} }}\frac{{u_{in} }}{{u_n }}  ]   \notag \\
&+\sum_{i \in G} (-\hat h_{in,n}-\sum_{k \in G} \hat h_{kk,i} ) [-4\frac{{u_{it} }}{{u_t }}+2\frac{{\hat h_{in} }}{{\hat h_{ii} }}(\frac{{u_{nn} }}{{u_n }}+\frac{{u_{nt} }}{{u_t }}) -2\frac{{\hat h_{in} }}{{\hat h_{ii} }}\sum_{j \in G} \frac{{\hat h_{jn} }}{{\hat h_{jj} }}\frac{{u_{jn} }}{{u_n }}  ]   \notag \\
&+ \sum_{i,j\in G} \hat h_{ij,n} [-4\frac{{\hat h_{in} }}{{\hat h_{ii} }}\frac{{u_{jt} }}{{u_t }} +2\frac{{\hat h_{in} }}{{\hat h_{ii} }} \frac{{\hat h_{jn} }}{{\hat h_{jj} }}\frac{{u_{nt} }}{{u_t }}] + \sum_{i,j\in G} (- \hat h_{in,j }+\hat h_{ij,n}) [4\frac{{\hat h_{in} }}{{\hat h_{ii} }}\frac{{u_{jt} }}{{u_t }}+2\frac{{\hat h_{in} }}{{\hat h_{ii} }}\frac{{\hat h_{jn} }}{{u_n u_t^2}} -  2 \frac{{\hat h_{in} }}{{u_n u_t^2 }} \frac{{\hat h_{jn} }}{{\hat h_{jj} }}]   \notag  \\
&+2\sum\limits_{i,j,k\in G} {\hat h_{ij,k } \frac{{\hat h_{in} }} {{\hat h_{ii} }}\frac{{\hat h_{jn} }} {{\hat h_{jj}
}}\frac{u_{kt} }{u_t}},
\end{align*}
so we have
\begin{align*}\label{6.87}
A
=& 2\frac{u_{nt}}{u_t}[ \sum_{i \in G}(\frac{{\hat h_{in} }}{{\hat h_{ii} }}\sum_{k \in G} \hat h_{kk,i}- \hat h_{in,i} )+\sum_{i \in G}\frac{{\hat h_{in} }}{{\hat h_{ii} }}(\sum_{j \in G}(\frac{{\hat h_{jn} }}{{\hat h_{jj} }} h_{ij,n}- \hat h_{in,n}) ]  \notag \\
&+(2 \frac{u_{nn}}{u_n}-2 \sum_{i \in G} \frac{{\hat h_{in} }}{{\hat h_{ii} }}\frac{{u_{in} }}{{u_n }})[ \sum_{i \in G}(\hat h_{in,i}- \frac{{\hat h_{in} }}{{\hat h_{ii} }}\sum_{k \in G} \hat h_{kk,i} )+(\hat h_{nn,n}-\sum_{i \in G}\frac{{\hat h_{in} }}{{\hat h_{ii} }} h_{in,n} ) ]  \notag  \\
&+ \sum_{i,j\in G} (- \hat h_{in,j }+\hat h_{ij,n}) [2\frac{{\hat h_{in} }}{{\hat h_{ii} }}\frac{{\hat h_{jn} }}{{u_n u_t^2}} -  2 \frac{{\hat h_{in} }}{{u_n u_t^2 }} \frac{{\hat h_{jn} }}{{\hat h_{jj} }}].
\end{align*}
With some computations, we obtain
\begin{eqnarray}\label{6.87}
A=&& 2[\sum_{i \in G}\frac{{\hat h_{in} }}{{\hat h_{ii} }}(\hat h_{in,n}-\sum_{j \in G}\frac{{\hat h_{jn} }}{{\hat h_{jj} }} h_{ij,n} )]( \frac{u_{nn}}{u_n}- \sum_{i \in G} \frac{{\hat h_{in} }}{{\hat h_{ii} }}\frac{{u_{in} }}{{u_n }}-\frac{u_{nt}}{u_t})  \notag \\
&&+2[ \sum_{i \in G}(\hat h_{in,i}- \frac{{\hat h_{in} }}{{\hat h_{ii} }}\sum_{k \in G} \hat h_{kk,i} ) ] ( \frac{u_{nn}}{u_n}- \sum_{i \in G} \frac{{\hat h_{in} }}{{\hat h_{ii} }}\frac{{u_{in} }}{{u_n }}-\frac{u_{nt}}{u_t})\notag  \\
&&+ \sum_{i,j\in G} 2\frac{{\hat h_{in} }}{{\hat h_{ii} }}\frac{{\hat h_{jn} }}{{u_n u_t^2}}(- \hat h_{in,j }+\hat h_{jn,i}) \notag  \\
=&& 2[\sum_{i \in G}\frac{{\hat h_{in} }}{{\hat h_{ii} }}(\hat h_{in,n}-\sum_{j \in G}\frac{{\hat h_{jn} }}{{\hat h_{jj} }} h_{ij,n} )]( \frac{u_{nn}}{u_n}- \sum_{i \in G} \frac{{\hat h_{in} }}{{\hat h_{ii} }}\frac{{u_{in} }}{{u_n }}-\frac{u_{nt}}{u_t})  \notag \\
&&+2[ \sum_{i \in G}(\hat h_{in,i}- \sum_{k \in G}\frac{{\hat h_{kn} }}{{\hat h_{kk} }} \hat h_{ik,i} ) ] ( \frac{u_{nn}}{u_n}- \sum_{i \in G} \frac{{\hat h_{in} }}{{\hat h_{ii} }}\frac{{u_{in} }}{{u_n }}-\frac{u_{nt}}{u_t})\notag  \\
&&+2[ \sum_{i \in G}(\sum_{k \in G}\frac{{\hat h_{kn} }}{{\hat h_{kk} }} \hat h_{ik,i}- \frac{{\hat h_{in} }}{{\hat h_{ii} }}\sum_{k \in G} \hat h_{kk,i} ) ] ( \frac{u_{nn}}{u_n}- \sum_{i \in G} \frac{{\hat h_{in} }}{{\hat h_{ii} }}\frac{{u_{in} }}{{u_n }}-\frac{u_{nt}}{u_t})\notag  \\
&&+ 2\sum_{i,j\in G} \frac{{\hat h_{in} }}{{\hat h_{ii} }}\frac{{\hat h_{jn} }}{{u_n u_t^2}}(- \hat h_{in,j }+\hat h_{jn,i}).\notag
\end{eqnarray}
At last we get
\begin{align}
A=& 2[\sum_{i \in G}\frac{{\hat h_{in} }}{{\hat h_{ii} }}(\hat h_{in,n}-\sum_{j \in G}\frac{{\hat h_{jn} }}{{\hat h_{jj} }} h_{ij,n} )]( \frac{u_{nn}}{u_n}- \sum_{i \in G} \frac{{\hat h_{in} }}{{\hat h_{ii} }}\frac{{u_{in} }}{{u_n }}-\frac{u_{nt}}{u_t})  \notag \\
&+2[ \sum_{i \in G}(\hat h_{in,i}- \sum_{k \in G}\frac{{\hat h_{kn} }}{{\hat h_{kk} }} \hat h_{ik,i} ) ] ( \frac{u_{nn}}{u_n}- \sum_{i \in G} \frac{{\hat h_{in} }}{{\hat h_{ii} }}\frac{{u_{in} }}{{u_n }}-\frac{u_{nt}}{u_t})\notag  \\
&+2[ \sum_{i,k \in G}\frac{{\hat h_{in} }}{{\hat h_{ii} }}(\hat h_{ki,k} -\hat h_{kk,i} ) ] ( \frac{u_{nn}}{u_n}- \sum_{i \in G} \frac{{\hat h_{in} }}{{\hat h_{ii} }}\frac{{u_{in} }}{{u_n }}-\frac{u_{nt}}{u_t})\notag  \\
&+2 \sum_{i,j\in G} \frac{{\hat h_{in} }}{{\hat h_{ii} }}\frac{{\hat h_{jn} }}{{u_n u_t^2}}(- \hat h_{in,j }+\hat h_{jn,i}).
\end{align}
Since
\begin{align*}
&\hat h_{ki,k} =  u_t^2 u_{kik}+ 2u_t u_{tk} u_{ki} -u_{kk} u_t u_{ti}- u_{ik} u_t u_{tk},  \\
&\hat h_{kk,i}= u_t^2 u_{kki}+ 2u_t u_{t i} u_{kk} -2u_{ki} u_t u_{tk},
\end{align*}
we can get
\begin{align}\label{6.88}
2\sum_{i,k \in G}\frac{{\hat h_{in} }}{{\hat h_{ii} }}(\hat h_{ki,k} -\hat h_{kk,i} ) = &2\sum_{i,k \in G}\frac{{\hat h_{in} }}{{\hat h_{ii} }}(3 u_t u_{tk} u_{ki} - 3u_{kk} u_t u_{ti})  \notag \\
 =&6\sum_{i \in G}\hat h_{in}\frac{{u_{ti} }}{u_t} -6\sum_{k \in G} \hat h_{kk} \sum_{i \in G}\frac{{\hat h_{in} }}{{\hat h_{ii} }}\frac{{u_{ti} }}{u_t}.
\end{align}
From Lemma \ref{lem6.4}, we have
\begin{align*}
&u_{n}u_t u_{ii t} =-\hat h_{in,i}+\hat h_{ii,n}+3\frac{{u_{it} }}{u_t}\hat h_{in} -3\frac{{u_{nt} }}{u_t}\hat h_{ii}+2u_n u_{it} ^2+ u_{ii}u_n u_{tt}, \quad  i \in G;  \\
 &u_{n}u_t u_{ij t} =-\hat h_{in,j}+\hat h_{ij,n}+3\frac{{u_{jt} }}{u_t}\hat h_{in} +2u_n u_{it} u_{jt}, \quad  i \in G, j \in G, i \ne j;
\end{align*}
then
\begin{align} \label{6.89}
2 \sum_{i,j\in G} \frac{{\hat h_{in} }}{{\hat h_{ii} }}\frac{{\hat h_{jn} }}{{u_n u_t^2}}(- \hat h_{in,j }+\hat h_{jn,i})=&2 \sum_{i,j\in G} \frac{{\hat h_{in} }}{{\hat h_{ii} }}\frac{{\hat h_{jn} }}{{u_n u_t^2}}(- 3\frac{{u_{jt} }}{u_t}\hat h_{in}+3\frac{{u_{it} }}{u_t}\hat h_{jn})  \notag \\
=& -6 \frac{{\hat h_{nn} }}{{u_n u_t^2}}\sum_{i \in G}\hat h_{in}\frac{{u_{ti} }}{u_t}
+6 \sum_{i\in G} \frac{{\hat h_{in} }}{{\hat h_{ii} }} \frac{{u_{ti} }}{u_t} \sum_{j \in G}\frac{{\hat h_{jn}^2 }}{{u_n u_t^2}}.
\end{align}
So by \eqref{6.87} - \eqref{6.89}, we have
\begin{align} \label{6.90}
A =& 2[\sum_{i \in G}\frac{{\hat h_{in} }}{{\hat h_{ii} }}(\hat h_{in,n}-\sum_{j \in G}\frac{{\hat h_{jn} }}{{\hat h_{jj} }} h_{ij,n} )]( \frac{u_{nn}}{u_n}- \sum_{i \in G} \frac{{\hat h_{in} }}{{\hat h_{ii} }}\frac{{u_{in} }}{{u_n }}-\frac{u_{nt}}{u_t})  \notag \\
&+2[ \sum_{i \in G}(\hat h_{in,i}- \sum_{k \in G}\frac{{\hat h_{kn} }}{{\hat h_{kk} }} \hat h_{ik,i} ) ] ( \frac{u_{nn}}{u_n}- \sum_{i \in G} \frac{{\hat h_{in} }}{{\hat h_{ii} }}\frac{{u_{in} }}{{u_n }}-\frac{u_{nt}}{u_t})\notag  \\
&+[ 6\sum_{i \in G}\hat h_{in}\frac{{u_{ti} }}{u_t} -6\sum_{k \in G} \hat h_{kk} \sum_{i \in G}\frac{{\hat h_{in} }}{{\hat h_{ii} }}\frac{{u_{ti} }}{u_t} ] ( \frac{u_{nn}}{u_n}- \sum_{i \in G} \frac{{\hat h_{in} }}{{\hat h_{ii} }}\frac{{u_{in} }}{{u_n }}-\frac{u_{nt}}{u_t})\notag  \\
&-6 \frac{{\hat h_{nn} }}{{u_n u_t^2}}\sum_{i \in G}\hat h_{in}\frac{{u_{ti} }}{u_t}
+6 \sum_{i\in G} \frac{{\hat h_{in} }}{{\hat h_{ii} }} \frac{{u_{ti} }}{u_t} \sum_{j \in G}\frac{{\hat h_{jn}^2 }}{{u_n u_t^2}}.
\end{align}
Now for $B$, we have
\begin{align}
B=&2 \frac{u_{nt}}{u_t} [u_t^2 u_{nt}
+2\frac{{u_{nt} }}{u_t}\sum\limits_{k \in G} {\hat h_{kk}}-2\sum\limits_{k \in G} {\frac{{u_{kt} }}{u_t}\hat h_{kn}}-2\sum\limits_{k \in G} {u_n u_{kt} ^2}]  \notag \\
&+ \sum_{i \in G}[2 \frac{u_{it}}{u_t}-4\frac{{\hat h_{in} }}{{\hat h_{ii} }}\frac{{u_{nt} }}{{u_t }} ] [u_t^2 u_{it} +2\frac{{u_{it} }}{u_t}\sum\limits_{k \in G, k\ne i} {\hat h_{kk}}] \notag \\
&+2 [ \frac{u_{nn}}{u_n}- \sum_{i \in G} \frac{{\hat h_{in} }}{{\hat h_{ii} }}\frac{{u_{in} }}{{u_n }} ] [-u_t^2 u_{nt}-4\sum\limits_{i \in G} \frac{{u_{it} }}{u_t}\hat h_{in}+ 4 \frac{{u_{nt} }}{u_t}\sum\limits_{i \in G}\hat h_{ii} - 2\sum\limits_{i \in G} u_n  u_{it}^2+ 2u_{n}u_{tn}^2]\notag \\
&+2\sum_{ i \in G} \frac{u_{it}}{u_t} [-u_t^2 u_{it} +4\frac{{u_{it} }}{u_t}\sum\limits_{k \in G,k \ne i} \hat h_{kk}+2\frac{{u_{it} }}{u_t}\hat h_{ii} +2\frac{{u_{nt} }}{u_t}\hat h_{in}- 2 u_t u_{it} u_{nn} + 2u_{t}u_{ni} u_{tn}+ 2u_{n}u_{ti} u_{tn}] \notag \\
&+ [-2(\frac{{u_{nn} }}{{u_n }}+\frac{{u_{nt} }}{{u_t }})+2\sum_{i \in G} \frac{{\hat h_{in} }}{{\hat h_{ii} }}\frac{{u_{in} }}{{u_n }}  ] [u_{n}u_t u_{tt}-\sum\limits_{i \in G} (3\frac{{u_{it} }}{u_t}\hat h_{in} -3\frac{{u_{nt} }}{u_t}\hat h_{ii}+2u_n u_{it} ^2+ u_{ii}u_n u_{tt})]  \notag \\
&+\sum_{i \in G} [-4\frac{{u_{it} }}{{u_t }}+2\frac{{\hat h_{in} }}{{\hat h_{ii} }}(\frac{{u_{nn} }}{{u_n }}+\frac{{u_{nt} }}{{u_t }}) -2\frac{{\hat h_{in} }}{{\hat h_{ii} }}\sum_{j \in G} \frac{{\hat h_{jn} }}{{\hat h_{jj} }}\frac{{u_{jn} }}{{u_n }}  ] [3\frac{{u_{it} }}{u_t}\sum\limits_{k \in G,k \ne i} \hat h_{kk} +\frac{{u_{it} }}{u_t}\hat h_{ii}+\frac{{u_{nt} }}{u_t}\hat h_{in}+ u_{in}u_n u_{tt}]  \notag \\
&+ \sum_{i,j\in G}  [-4\frac{{\hat h_{in} }}{{\hat h_{ii} }}\frac{{u_{jt} }}{{u_t }} +2\frac{{\hat h_{in} }}{{\hat h_{ii} }} \frac{{\hat h_{jn} }}{{\hat h_{jj} }}\frac{{u_{nt} }}{{u_t }}] [\frac{u_{i t}}{u_t}  \hat h_{jn}+\frac{u_{j t}}{u_t} \hat h_{in}+ 2 u_n u_{i t}u_{j t}] \notag  \\
&+ \sum_{i \in G}  [-4\frac{{\hat h_{in} }}{{\hat h_{ii} }}\frac{{u_{it} }}{{u_t }} +2\frac{{\hat h_{in}^2 }}{{\hat h_{ii}^2 }} \frac{{u_{nt} }}{{u_t }}] [-2\frac{u_{n t}}{u_t}  \hat h_{ii}]\notag \\
&+ \sum_{i,j\in G} [4\frac{{\hat h_{in} }}{{\hat h_{ii} }}\frac{{u_{jt} }}{{u_t }}+2\frac{{\hat h_{in} }}{{\hat h_{ii} }}\frac{{\hat h_{jn} }}{{u_n u_t^2}} -  2 \frac{{\hat h_{in} }}{{u_n u_t^2 }} \frac{{\hat h_{jn} }}{{\hat h_{jj} }}] [3\frac{{u_{jt} }}{u_t}\hat h_{in} +2u_n u_{it} u_{jt}] \notag  \\
& + 4\sum_{i \in G} \frac{{\hat h_{in} }}{{\hat h_{ii} }}\frac{{u_{it} }}{{u_t }} [-3\frac{{u_{nt} }}{u_t}\hat h_{ii}+ u_{ii}u_n u_{tt}]\notag  \\
&+2\sum\limits_{i,j,k\in G} {u_t^2 u_{ijk } \frac{{\hat h_{in} }} {{\hat h_{ii} }}\frac{{\hat h_{jn} }} {{\hat h_{jj}
}}\frac{u_{kt} }{u_t}}   - 2\sum\limits_{i,j,k\in G} {\hat h_{ij,k } \frac{{\hat h_{in} }} {{\hat h_{ii} }}\frac{{\hat h_{jn} }} {{\hat h_{jj}
}}\frac{u_{kt} }{u_t}},
\end{align}
and we let
\begin{align*}
B=2 \frac{u_{nt}}{u_t}  B_1 + 2[ \frac{u_{nn}}{u_n}- \sum_{i \in G} \frac{{\hat h_{in} }}{{\hat h_{ii} }}\frac{{u_{in} }}{{u_n }} ]B_2+ B_3+B_4+B_5, \notag
\end{align*}
where
\begin{align} \label{6.91}
B_1=&[u_t^2 u_{nt}
+2\frac{{u_{nt} }}{u_t}\sum\limits_{k \in G} {\hat h_{kk}}-2\sum\limits_{k \in G} {\frac{{u_{kt} }}{u_t}\hat h_{kn}}-2\sum\limits_{k \in G} {u_n u_{kt} ^2}]  \notag \\
&-2\sum_{i \in G}\frac{{\hat h_{in} }}{{\hat h_{ii} }} [u_t^2 u_{it} +2\frac{{u_{it} }}{u_t}\sum\limits_{k \in G, k\ne i} {\hat h_{kk}}] \notag \\
&- [u_{n}u_t u_{tt}-\sum\limits_{i \in G} (3\frac{{u_{it} }}{u_t}\hat h_{in} -3\frac{{u_{nt} }}{u_t}\hat h_{ii}+2u_n u_{it} ^2+ u_{ii}u_n u_{tt})]  \notag \\
&+\sum_{i \in G} \frac{{\hat h_{in} }}{{\hat h_{ii} }} [3\frac{{u_{it} }}{u_t}\sum\limits_{k \in G,k \ne i} \hat h_{kk} +\frac{{u_{it} }}{u_t}\hat h_{ii}+\frac{{u_{nt} }}{u_t}\hat h_{in}+ u_{in}u_n u_{tt}]  \notag \\
&+ \sum_{i,j\in G}\frac{{\hat h_{in} }}{{\hat h_{ii} }} \frac{{\hat h_{jn} }}{{\hat h_{jj} }} [\frac{u_{i t}}{u_t}  \hat h_{jn}+\frac{u_{j t}}{u_t} \hat h_{in}+ 2 u_n u_{i t}u_{j t}] + \sum_{i \in G}  \frac{{\hat h_{in}^2 }}{{\hat h_{ii}^2 }}  [-2\frac{u_{n t}}{u_t}  \hat h_{ii}]\notag \\
=&\frac{{u_{nt}}}{{u_t}}u_t^2 [\sum_{i \in G} u_{ii} + u_{nn} ]
+2\frac{{u_{nt} }}{u_t}\sum\limits_{k \in G} {\hat h_{kk}}-2\sum\limits_{k \in G} {\hat h_{kn}\frac{{u_{kt} }}{u_t}}-2\sum\limits_{k \in G} {u_n u_{kt} ^2}  \notag \\
&-2\sum_{i \in G}\frac{{\hat h_{in} }}{{\hat h_{ii} }} \frac{{u_{it}}}{{u_t}}u_t^2 [\sum_{k \in G} u_{kk} + u_{nn} ]-4\sum_{i \in G}\frac{{\hat h_{in} }}{{\hat h_{ii} }} \frac{{u_{it}}}{{u_t}}\sum\limits_{k \in G} {\hat h_{kk}} +4\sum\limits_{i \in G} {\hat h_{in}\frac{{u_{it} }}{u_t}} \notag \\
&- u_{n}u_{tt}[\sum_{i \in G} u_{ii} + u_{nn} ] +3\sum\limits_{i \in G} \hat h_{in}\frac{{u_{it} }}{u_t} -3\frac{{u_{nt} }}{u_t}\sum\limits_{i \in G} \hat h_{ii}+2u_n \sum\limits_{i \in G} u_{it} ^2+ u_n u_{tt}\sum\limits_{i \in G} u_{ii}  \notag \\
&+3\sum_{i \in G}\frac{{\hat h_{in} }}{{\hat h_{ii} }} \frac{{u_{it}}}{{u_t}}\sum\limits_{k \in G} {\hat h_{kk}} -2\sum\limits_{i \in G} {\hat h_{in}\frac{{u_{it} }}{u_t}}+\frac{{u_{nt} }}{u_t}\sum_{i \in G} \frac{{\hat h_{in}^2 }}{{\hat h_{ii} }} + u_n^2 u_{tt}\sum_{i \in G} \frac{{\hat h_{in} }}{{\hat h_{ii} }} \frac{ u_{in}}{u_t}  \notag \\
&+ 2\hat h_{nn}\sum_{i\in G}\frac{{\hat h_{in} }}{{\hat h_{ii} }} \frac{u_{i t}}{u_t} + 2u_n u_t^2[\sum_{i\in G}\frac{{\hat h_{in} }}{{\hat h_{ii} }} \frac{u_{i t}}{u_t}]^2 -2\frac{u_{n t}}{u_t}\hat h_{nn},
\end{align}
so we get
\begin{align} \label{6.91}
B_1
\sim&\frac{{u_{nt}}}{{u_t}}u_t^2u_{nn} +3\sum\limits_{k \in G} {\hat h_{kn}\frac{{u_{kt} }}{u_t}}  -2\sum_{i \in G}\frac{{\hat h_{in} }}{{\hat h_{ii} }} \frac{{u_{it}}}{{u_t}}u_t^2 u_{nn}-3\sum_{i \in G}\frac{{\hat h_{in} }}{{\hat h_{ii} }} \frac{{u_{it}}}{{u_t}}\sum\limits_{k \in G} {\hat h_{kk}}  \notag \\
&- u_{n}^2u_{tt} \frac{u_{nn}}{u_n}  + u_n^2 u_{tt} \sum_{i \in G} \frac{{\hat h_{in} }}{{\hat h_{ii} }}\frac{ u_{in}}{u_t}  \notag \\
&+ 2\hat h_{nn}\sum_{i\in G}\frac{{\hat h_{in} }}{{\hat h_{ii} }} \frac{u_{i t}}{u_t} + 2u_n u_t^2[\sum_{i\in G}\frac{{\hat h_{in} }}{{\hat h_{ii} }} \frac{u_{i t}}{u_t}]^2 -\frac{u_{n t}}{u_t}\hat h_{nn}  \notag \\
=&\frac{{u_{nt}}}{{u_t}}[-u_n^2u_{tt} +2u_n u_t u_{nt}]-2\sum_{i \in G}\frac{{\hat h_{in} }}{{\hat h_{ii} }} \frac{{u_{it}}}{{u_t}}[-u_n^2u_{tt} +2u_n u_t u_{nt}]\notag \\
& +3\sum\limits_{k \in G} {\hat h_{kn}\frac{{u_{kt} }}{u_t}}  -3\sum_{i \in G}\frac{{\hat h_{in} }}{{\hat h_{ii} }} \frac{{u_{it}}}{{u_t}}\sum\limits_{k \in G} {\hat h_{kk}}  \notag \\
&+ u_n^2 u_{tt} [- \frac{u_{nn}}{u_n}+\sum_{i \in G} \frac{{\hat h_{in} }}{{\hat h_{ii} }}\frac{ u_{in}}{u_t} ] + 2u_n u_t^2[\sum_{i\in G}\frac{{\hat h_{in} }}{{\hat h_{ii} }} \frac{u_{i t}}{u_t}]^2  \notag \\
=& [2 \sum_{i \in G}  \frac{{u_{it} }}{{u_t }}\frac{{\hat h_{in} }}{{\hat h_{ii} }}- \frac{{u_{nt} }}{{u_t }} - (\frac{{u_{nn} }}{{u_n }} - \sum_{i \in G} \frac{{\hat h_{in} }}{{\hat h_{ii} }}\frac{{u_{in} }}{{u_n }})]u_n ^2 u_{tt}  \notag\\
&+ 2u_n u_t^2 [\frac{{u_{nt} }}{{u_t }} -\sum_{i \in G} \frac{{\hat h_{in} }}{{\hat h_{ii} }}\frac{{u_{it} }}{{u_t }}]^2 +3\sum_{i \in G}\hat h_{in}\frac{{u_{ti} }}{u_t} -3\sum_{k \in G} \hat h_{kk} \sum_{i \in G}\frac{{\hat h_{in} }}{{\hat h_{ii} }}\frac{{u_{ti} }}{u_t}.
\end{align}

Now for the $B_2$,
\begin{align}\label{6.92a}
B_2=& [-u_t^2 u_{nt} -4\sum\limits_{i \in G} \frac{{u_{it} }}{u_t}\hat h_{in}+ 4 \frac{{u_{nt} }}{u_t}\sum\limits_{i \in G}\hat h_{ii} - 2\sum\limits_{i \in G} u_n  u_{it}^2+ 2u_{n}u_{tn}^2]\notag \\
&- [u_{n}u_t u_{tt}-\sum\limits_{i \in G} (3\frac{{u_{it} }}{u_t}\hat h_{in} -3\frac{{u_{nt} }}{u_t}\hat h_{ii}+2u_n u_{it} ^2+ u_{ii}u_n u_{tt})]  \notag \\
&+\sum_{i \in G}\frac{{\hat h_{in} }}{{\hat h_{ii} }} [3\frac{{u_{it} }}{u_t}\sum\limits_{k \in G,k \ne i} \hat h_{kk} +\frac{{u_{it} }}{u_t}\hat h_{ii}+\frac{{u_{nt} }}{u_t}\hat h_{in}+ u_{in}u_n u_{tt}]  \notag \\
=& -\frac{u_{nt}}{u_t} u_t^2 [\sum_{i \in G} u_{ii} + u_{nn} ] -4\sum\limits_{i \in G}\hat h_{in} \frac{{u_{it} }}{u_t}+ 4 \frac{{u_{nt} }}{u_t}\sum\limits_{i \in G}\hat h_{ii} - 2 u_n \sum\limits_{i \in G} u_{it}^2+ 2u_{n}u_{tn}^2\notag \\
&- u_{n}u_{tt}[\sum_{i \in G} u_{ii} + u_{nn} ] +3\sum\limits_{i \in G} \hat h_{in}\frac{{u_{it} }}{u_t} -3\frac{{u_{nt} }}{u_t}\sum\limits_{i \in G} \hat h_{ii}+2u_n \sum\limits_{i \in G} u_{it} ^2+ u_n u_{tt}\sum\limits_{i \in G} u_{ii}  \notag \\
&+3\sum_{i \in G}\frac{{\hat h_{in} }}{{\hat h_{ii} }} \frac{{u_{it}}}{{u_t}}\sum\limits_{k \in G} {\hat h_{kk}} -2\sum\limits_{i \in G} {\hat h_{in}\frac{{u_{it} }}{u_t}}+\frac{{u_{nt} }}{u_t}\sum_{i \in G} \frac{{\hat h_{in}^2 }}{{\hat h_{ii} }} + u_n^2 u_{tt}\sum_{i \in G} \frac{{\hat h_{in} }}{{\hat h_{ii} }} \frac{ u_{in}}{u_t},
\end{align}
so we get
\begin{align}\label{6.92}
B_2\sim& -\frac{u_{nt}}{u_t} [u_t^2u_{nn} -\hat h_{nn}]-3 \sum\limits_{i \in G}\hat h_{in} \frac{{u_{it} }}{u_t}+ 2u_{n}u_{tn}^2\notag \\
&- u_{n}^2 u_{tt} [\frac{u_{nn}}{u_n} -\sum_{i \in G} \frac{{\hat h_{in} }}{{\hat h_{ii} }} \frac{ u_{in}}{u_t} ]+3\sum_{i \in G}\frac{{\hat h_{in} }}{{\hat h_{ii} }} \frac{{u_{it}}}{{u_t}}\sum\limits_{k \in G} {\hat h_{kk}}  \notag \\
=& - [\frac{{u_{nn} }}{{u_n }} - \frac{{u_{nt} }}{{u_t }}- \sum_{i \in G} \frac{{\hat h_{in} }}{{\hat h_{ii} }}\frac{{u_{in} }}{{u_n }}]u_n ^2 u_{tt}  -3\sum_{i \in G}\hat h_{in}\frac{{u_{ti} }}{u_t} +3\sum_{k \in G} \hat h_{kk} \sum_{i \in G}\frac{{\hat h_{in} }}{{\hat h_{ii} }}\frac{{u_{ti} }}{u_t}.
\end{align}
Now for the $B_3$, we have
\begin{align}\label{6.93*}
B_3=& 2\sum_{i \in G} \frac{u_{it}}{u_t} [u_t^2u_{it} +2\frac{{u_{it} }}{u_t}\sum\limits_{k \in G, k\ne i} {\hat h_{kk}}] \notag \\
&+2\sum_{ i \in G} \frac{u_{it}}{u_t} [-u_t^2 u_{it} +4\frac{{u_{it} }}{u_t}\sum\limits_{k \in G,k \ne i} \hat h_{kk}+2\frac{{u_{it} }}{u_t}\hat h_{ii} +2\frac{{u_{nt} }}{u_t}\hat h_{in}- 2 u_t u_{it} u_{nn} + 2u_{t}u_{ni} u_{tn}+ 2u_{n}u_{ti} u_{tn}] \notag \\
&-4\sum_{i \in G} \frac{{u_{it} }}{{u_t }} [3\frac{{u_{it} }}{u_t}\sum\limits_{k \in G,k \ne i} \hat h_{kk} +\frac{{u_{it} }}{u_t}\hat h_{ii}+\frac{{u_{nt} }}{u_t}\hat h_{in}+ u_{in}u_n u_{tt}]  \notag \\
&-4\sum_{i,j\in G}  \frac{{\hat h_{in} }}{{\hat h_{ii} }}\frac{{u_{jt} }}{{u_t }}  [\frac{u_{i t}}{u_t}  \hat h_{jn}+\frac{u_{j t}}{u_t} \hat h_{in}+ 2 u_n u_{i t}u_{j t}] -4 \sum_{i \in G}  \frac{{\hat h_{in} }}{{\hat h_{ii} }}\frac{{u_{it} }}{{u_t }} [-2\frac{u_{n t}}{u_t}  \hat h_{ii}]\notag \\
&+4 \sum_{i,j\in G} \frac{{\hat h_{in} }}{{\hat h_{ii} }}\frac{{u_{jt} }}{{u_t }} [3\frac{{u_{jt} }}{u_t}\hat h_{in} +2u_n u_{it} u_{jt}]  + 4\sum_{i \in G} \frac{{\hat h_{in} }}{{\hat h_{ii} }}\frac{{u_{it} }}{{u_t }} [-3\frac{{u_{nt} }}{u_t}\hat h_{ii}+ u_{ii}u_n u_{tt}]\notag  \\
\sim&  4\sum_{i \in G} (\frac{u_{it}}{u_t})^2 \sum\limits_{k \in G} {\hat h_{kk}}- 4\sum_{i \in G} (\frac{u_{it}}{u_t})^2 \hat h_{ii} + 8 \sum_{i \in G} (\frac{u_{it}}{u_t})^2 \sum\limits_{k \in G} {\hat h_{kk}} - 4\sum_{i \in G} (\frac{u_{it}}{u_t})^2 \hat h_{ii} \notag \\
&+4\frac{{u_{nt} }}{u_t}\sum_{ i \in G} \frac{u_{it}}{u_t}\hat h_{in}- 4 u_t^2  u_{nn} \sum_{ i \in G} (\frac{u_{it}}{u_t})^2+ 4u_n u_{t}u_{tn}\sum_{ i \in G} \frac{u_{it}}{u_t}[ \frac{\hat h_{in}}{u_n u_t^2}+\frac{u_{it}}{u_t} ]\notag \\
&+ 4u_{n}u_{t} u_{tn}\sum_{ i \in G} (\frac{u_{it}}{u_t})^2-12 \sum_{i \in G} (\frac{u_{it}}{u_t})^2 \sum\limits_{k \in G} {\hat h_{kk}}+8 \sum_{i \in G} (\frac{u_{it}}{u_t})^2 \hat h_{ii}
\notag \\
&-4\frac{{u_{nt} }}{u_t}\sum_{i \in G} \hat h_{in}\frac{{u_{it} }}{{u_t }} -4u_n^2 u_{tt}\sum_{i \in G} \frac{{u_{it} }}{{u_t }}[ \frac{\hat h_{in}}{u_n u_t^2}+\frac{u_{it}}{u_t} ]-4\sum\limits_{i\in G} \hat h_{in}\frac{u_{it} }{u_t} \sum\limits_{j\in G} \frac{{\hat h_{jn} }} {{\hat h_{jj}}}\frac{u_{j t}}{u_t} -4\hat h_{nn}\sum\limits_{i \in G}  (\frac{u_{it}}{u_t})^2 \notag \\
&-8 u_n u_t^2 \sum_{i\in G}  \frac{{\hat h_{in} }}{{\hat h_{ii} }}\frac{{u_{it} }}{{u_t }}  \sum_{j \in G}(\frac{ u_{j t}}{u_{t}})^2 +8\frac{u_{n t}}{u_t}  \sum_{i \in G} \hat h_{in}\frac{{u_{it} }}{{u_t }}+12\hat h_{nn}\sum\limits_{i \in G}  (\frac{u_{it}}{u_t})^2\notag \\
&+8 u_n u_t^2 \sum_{i\in G}  \frac{{\hat h_{in} }}{{\hat h_{ii} }}\frac{{u_{it} }}{{u_t }}  \sum_{j \in G}(\frac{ u_{j t}}{u_{t}})^2 -12\frac{u_{n t}}{u_t}  \sum_{i \in G} \hat h_{in}\frac{{u_{it} }}{{u_t }}+ 4\sum_{i \in G}\hat h_{in} \frac{{u_{it} }}{{u_t }} \frac{u_n^2u_{tt}}{u_nu_t^2},
\end{align}
so we have
\begin{align}\label{6.93}
B_3
=&- 4 u_t^2  u_{nn} \sum_{ i \in G} (\frac{u_{it}}{u_t})^2+ 8u_{n}u_{t} u_{tn}\sum_{ i \in G} (\frac{u_{it}}{u_t})^2 -4u_n^2 u_{tt}\sum_{ i \in G} (\frac{u_{it}}{u_t})^2   \notag \\
&-4\sum\limits_{i\in G} \hat h_{in}\frac{u_{it} }{u_t} \sum\limits_{j\in G} \frac{{\hat h_{jn} }} {{\hat h_{jj}}}\frac{u_{j t}}{u_t} +8\hat h_{nn}\sum\limits_{i \in G}  (\frac{u_{it}}{u_t})^2 \notag \\
=&4\hat h_{nn}\sum\limits_{i \in G}  (\frac{u_{it}}{u_t})^2 -4 \sum\limits_{i\in G} \hat h_{in}\frac{u_{it} }{u_t} \sum\limits_{j\in G} \frac{{\hat h_{jn} }} {{\hat h_{jj}}}\frac{u_{j t}}{u_t}.
\end{align}
For the term $B_4$,
\begin{align} \label{6.94}
B_4=& \sum_{i,j\in G} [2\frac{{\hat h_{in} }}{{\hat h_{ii} }}\frac{{\hat h_{jn} }}{{u_n u_t^2}} -  2 \frac{{\hat h_{in} }}{{u_n u_t^2 }} \frac{{\hat h_{jn} }}{{\hat h_{jj} }}] [3\frac{{u_{jt} }}{u_t}\hat h_{in} +2u_n u_{it} u_{jt}] \notag \\
\sim& 6 \frac{{\hat h_{nn} }}{{u_n u_t^2}}\sum_{i \in G}\hat h_{in}\frac{{u_{ti} }}{u_t}
-6 \sum_{i\in G} \frac{{\hat h_{in} }}{{\hat h_{ii} }} \frac{{u_{ti} }}{u_t} \sum_{j \in G}\frac{{\hat h_{jn}^2 }}{{u_n u_t^2}},
\end{align}
and
\begin{align} \label{6.95}
B_5=& 2\sum\limits_{i,j,k\in G} {u_t^2 u_{ijk } \frac{{\hat h_{in} }} {{\hat h_{ii} }}\frac{{\hat h_{jn} }} {{\hat h_{jj}
}}\frac{u_{kt} }{u_t}}   - 2\sum\limits_{i,j,k\in G} {\hat h_{ij,k } \frac{{\hat h_{in} }} {{\hat h_{ii} }}\frac{{\hat h_{jn} }} {{\hat h_{jj}
}}\frac{u_{kt} }{u_t}} \notag  \\
=& 2\sum\limits_{i,k\in G,i \ne k}  \frac{{\hat h_{in} ^2}} {{\hat h_{ii}^2 }}\frac{u_{kt} }{u_t}[-2\frac{u_{k t}}{u_t}  \hat h_{ii}]
+2 \sum\limits_{i,j\in G,i \ne j}  \frac{{\hat h_{in} }} {{\hat h_{ii} }}\frac{{\hat h_{jn} }} {{\hat h_{jj}}}\frac{u_{it} }{u_t} [\frac{u_{j t}}{u_t}  \hat h_{ii}] \notag  \\
&+2 \sum\limits_{i,j\in G,i \ne j}  \frac{{\hat h_{in} }} {{\hat h_{ii} }}\frac{{\hat h_{jn} }} {{\hat h_{jj}}}\frac{u_{jt} }{u_t} [\frac{u_{i t}}{u_t}  \hat h_{jj}]\notag  \\
=&2\sum\limits_{i,k\in G,i \ne k}  \frac{{\hat h_{in} ^2}} {{\hat h_{ii}^2 }}\frac{u_{kt} }{u_t}[-2\frac{u_{k t}}{u_t}  \hat h_{ii}]
+2 \sum\limits_{i,j\in G,i \ne j}  \frac{{\hat h_{in} }} {{\hat h_{ii} }}\frac{{\hat h_{jn} }} {{\hat h_{jj}}}\frac{u_{it} }{u_t}] [\frac{u_{j t}}{u_t}  \hat h_{ii}] \notag  \\
&+2 \sum\limits_{i,j\in G,i \ne j}  \frac{{\hat h_{in} }} {{\hat h_{ii} }}\frac{{\hat h_{jn} }} {{\hat h_{jj}}}\frac{u_{jt} }{u_t}] [\frac{u_{i t}}{u_t}  \hat h_{jj}]\notag \\
\sim&-4\hat h_{nn}\sum\limits_{j\in G}  (\frac{u_{jt}}{u_t})^2 +4 \sum\limits_{i\in G} \hat h_{in}\frac{u_{it} }{u_t} \sum\limits_{j\in G} \frac{{\hat h_{jn} }} {{\hat h_{jj}}}\frac{u_{j t}}{u_t}.
\end{align}
So by \eqref{6.93} and \eqref{6.95}, it is easy to see
\begin{align}\label{6.96}
B_3+B_5 \sim 0.
\end{align}

For the $II$, we can get
\begin{align}
II =&\frac{{u_{\alpha t} ^2 }}{{u_t ^2 }}[ u_{ t} ^2 u_{nn}  - 2\sum_{i \in G} {\frac{{\hat h_{in} }} {{\hat h_{ii}}}u_{ t} ^2 u_{in}}+\sum_{i,j\in G} {\frac{{\hat h_{in} }} {{\hat h_{ii} }} \frac{{\hat h_{jn} }} {{\hat h_{jj}}} u_{ t} ^2 u_{ij}} ] -2u_{n \a}u_{t \a} u_{nt} \notag \\
&+ u_{n} ^2 u_{tt}[\frac{{u_{\alpha n} ^2 }}{{u_n ^2 }}- 2\sum\limits_{i \in G} {\frac{{\hat h_{in} }} {{\hat h_{ii}}}\frac{{u_{i\alpha} }}{{u_n }}\frac{{u_{\alpha n} }}{{u_n }}} +\sum\limits_{i,j\in G} {\frac{{\hat h_{in} }} {{\hat h_{ii} }} \frac{{\hat h_{jn} }} {{\hat h_{jj}}} \frac{{u_{i\alpha} }}{{u_n }}\frac{{u_{j \alpha} }}{{u_n }}}  ] \notag \\
&+ 2\sum\limits_{i \in G} {\frac{{\hat h_{in} }} {{\hat h_{ii}}}[ u_{i\alpha } u_{t\alpha } u_{nt}  + u_{n\alpha } u_{t\alpha } u_{it} ]}
-2\sum\limits_{i,j\in G} {\frac{{\hat h_{in} }} {{\hat h_{ii} }}\frac{{\hat h_{jn} }} {{\hat h_{jj}
}} u_{t\alpha } u_{i\alpha } u_{jt}  } \notag\\
\sim&\frac{{u_{\alpha t} ^2 }}{{u_t ^2 }}[ u_{ t} ^2 u_{nn}- \hat h_{nn}  - 2\sum\limits_{i \in G} {\frac{{\hat h_{in} }} {{\hat h_{ii}}} u_n u_{ t} u_{it}} ] -2u_{n \a}u_{t \a} u_{nt}\notag \\
&+ u_{n} ^2 u_{tt}[(\frac{{u_{n n} }}{{u_n }}- \sum\limits_{i \in G} {\frac{{\hat h_{in} }} {{\hat h_{ii}}}\frac{{u_{in} }}{{u_n }}})^2 +\sum_{i\in G } (\frac{{u_{i n} }}{{u_n }}- \frac{{\hat h_{in} }} {{\hat h_{ii}}}\frac{{u_{ii} }}{{u_n }})^2]\notag \\
&+ 2\sum\limits_{i \in G} {\frac{{\hat h_{in} }} {{\hat h_{ii}}}[ u_{ii} u_{ti} u_{nt}  +u_{in } u_{tn } u_{nt}  + u_{nn } u_{tn } u_{it}+\sum_{j \in G} u_{nj } u_{tj} u_{it} ]}\notag \\
&-2\sum\limits_{i,j\in G} {\frac{{\hat h_{in} }} {{\hat h_{ii} }}\frac{{\hat h_{jn} }} {{\hat h_{jj}
}} [u_{tn} u_{in} u_{jt}+u_{ti } u_{ii } u_{jt}  ]  },
\end{align}
then we have
\begin{align}
II\sim&\frac{{u_{n t} ^2 }}{{u_t ^2 }}[- u_{ n} ^2 u_{tt} + 2u_n u_t u_{nt} - 2\sum\limits_{i \in G} {\frac{{\hat h_{in} }} {{\hat h_{ii}}} u_n u_{ t} u_{it}} ]+\sum_{j \in G} \frac{{u_{j t} ^2 }}{{u_t ^2 }}[ - u_{ n} ^2 u_{tt} + 2u_n u_t u_{nt}  - 2\sum\limits_{i \in G} {\frac{{\hat h_{in} }} {{\hat h_{ii}}} u_n u_{ t} u_{it}} ] \notag \\
&-2u_{nn}u_{tn} u_{nt}-2 \sum_{j \in G} u_{nj}u_{tj} u_{nt}+ u_{n} ^2 u_{tt}[(\frac{{u_{n n} }}{{u_n }}- \sum\limits_{i \in G} {\frac{{\hat h_{in} }} {{\hat h_{ii}}}\frac{{u_{in} }}{{u_n }}})^2 +\sum_{i\in G } (\frac{{u_{i t} }}{{u_t }})^2]\notag \\
&+ 2 \frac{u_{nt}}{u_t} \sum\limits_{i \in G} {\hat h_{in} \frac{{ u_{ti} }} {{u_t}}}+ 2 \frac{{u_{n t} ^2 }}{{u_t ^2 }} \sum\limits_{i \in G} {\frac{{\hat h_{in} }} {{\hat h_{ii}}} u_t^2 u_{in}} +2 \frac{u_{nt}}{u_t}\sum\limits_{i \in G} {\frac{{\hat h_{in} }} {{\hat h_{ii}}}\frac{u_{it}}{u_t}} u_t^2 u_{nn}
+ 2\sum\limits_{i \in G} {\frac{{\hat h_{in} }} {{\hat h_{ii}}}[ \sum_{j \in G} u_{nj } u_{tj} u_{it} ]}\notag \\
&-2 \frac{u_{nt}}{u_t} \sum\limits_{i,j\in G} {\frac{{\hat h_{in} }} {{\hat h_{ii} }}\frac{{\hat h_{jn} }} {{\hat h_{jj}
}} \frac{ u_{jt}}{u_t} u_t^2 u_{in} }
-2\sum\limits_{i,j\in G} {\hat h_{in}\frac{{\hat h_{jn} }} {{\hat h_{jj}
}}\frac{ u_{it}}{u_t}\frac{ u_{jt}}{u_t} } \notag\\
=&\frac{{u_{n t} ^2 }}{{u_t ^2 }}[-2u_t^2 u_{nn}- u_{ n} ^2 u_{tt} + 2u_n u_t u_{nt} - 2\sum\limits_{i \in G} {\frac{{\hat h_{in} }} {{\hat h_{ii}}} u_n u_{ t} u_{it}} ]+\sum_{j \in G} \frac{{u_{j t} ^2 }}{{u_t ^2 }}[ 2u_n u_t u_{nt}  - 2\sum\limits_{i \in G} {\frac{{\hat h_{in} }} {{\hat h_{ii}}} u_n u_{ t} u_{it}} ] \notag \\
& +u_{n} ^2 u_{tt}[(\frac{{u_{n n} }}{{u_n }}- \sum\limits_{i \in G} {\frac{{\hat h_{in} }} {{\hat h_{ii}}}\frac{{u_{in} }}{{u_n }}})^2 ]+ 2 \frac{u_{nt}}{u_t} \sum\limits_{i \in G} {[\hat h_{in}-u_t^2 u_{in}] \frac{{ u_{ti} }} {{u_t}}}\notag \\
&+ 2 \frac{{u_{n t} ^2 }}{{u_t ^2 }} \sum\limits_{i \in G} {\frac{{\hat h_{in} }} {{\hat h_{ii}}} [\hat h_{in} +u_n u_t u_{it}]} +2 \frac{u_{nt}}{u_t}\sum\limits_{i \in G} {\frac{{\hat h_{in} }} {{\hat h_{ii}}}\frac{u_{it}}{u_t}} u_t^2 u_{nn}
+ 2\sum\limits_{i \in G} {\frac{{\hat h_{in} }} {{\hat h_{ii}}} \frac{u_{it}}{u_t}\sum_{j \in G} [ \hat h_{jn} +u_nu_t u_{jt}]\frac{u_{tj}}{u_t} }\notag \\
&-2 \frac{u_{nt}}{u_t} \sum\limits_{i,j\in G} {\frac{{\hat h_{in} }} {{\hat h_{ii} }}\frac{{\hat h_{jn} }} {{\hat h_{jj}
}} \frac{ u_{jt}}{u_t}[\hat h_{in} +u_n u_t u_{it}] }-2\sum\limits_{i,j\in G} {\hat h_{in}\frac{{\hat h_{jn} }} {{\hat h_{jj}
}}\frac{ u_{it}}{u_t}\frac{ u_{jt}}{u_t} }, \notag
\end{align}
so
\begin{align}\label{6.97}
II \sim&\frac{{u_{n t} ^2 }}{{u_t ^2 }}[ u_{ n} ^2 u_{tt} - 2u_n u_t u_{nt}] +u_{n} ^2 u_{tt}[(\frac{{u_{n n} }}{{u_n }}- \sum\limits_{i \in G} {\frac{{\hat h_{in} }} {{\hat h_{ii}}}\frac{{u_{in} }}{{u_n }}})^2 ]\notag \\
&+ 2 \frac{u_{nt}}{u_t}\sum\limits_{i \in G} {\frac{{\hat h_{in} }} {{\hat h_{ii}}}\frac{u_{it}}{u_t}}[ u_t^2 u_{nn}- \hat h_{nn}]-2 \frac{u_{nt}}{u_t} \sum\limits_{i,j\in G} {\frac{{\hat h_{in} }} {{\hat h_{ii} }}\frac{{\hat h_{jn} }} {{\hat h_{jj}}} \frac{ u_{jt}}{u_t} u_n u_t u_{it} }\notag\\
=& [\frac{{u_{nt} ^2 }}{{u_t ^2 }} - 2\frac{{u_{nt} }}{{u_t }} \sum_{i \in G}  \frac{{u_{it} }}{{u_t }}\frac{{\hat h_{in} }}{{\hat h_{ii} }} + (\frac{{u_{nn} }}{{u_n }} - \sum_{i \in G} \frac{{\hat h_{in} }}{{\hat h_{ii} }}\frac{{u_{in} }}{{u_n }})^2 ]u_n ^2 u_{tt}  \notag\\
&- 2u_n u_t u_{nt} [\frac{{u_{nt} }}{{u_t }} -\sum_{i \in G} \frac{{u_{it} }}{{u_t }}\frac{{\hat h_{in} }}{{\hat h_{ii} }}]^2.
\end{align}

Then we get from \eqref{6.90} - \eqref{6.92}, \eqref{6.94} and \eqref{6.97}
\begin{align}  \label{6.98}
&A+2 \frac{u_{nt}}{u_t}  B_1 + 2[ \frac{u_{nn}}{u_n}- \sum_{i \in G} \frac{{\hat h_{in} }}{{\hat h_{ii} }}\frac{{u_{in} }}{{u_n }} ]B_2+ B_4 +II \notag \\
=& 2[\sum_{i \in G}\frac{{\hat h_{in} }}{{\hat h_{ii} }}(\hat h_{in,n}-\sum_{j \in G}\frac{{\hat h_{jn} }}{{\hat h_{jj} }} h_{ij,n} )]( \frac{u_{nn}}{u_n}- \sum_{i \in G} \frac{{\hat h_{in} }}{{\hat h_{ii} }}\frac{{u_{in} }}{{u_n }}-\frac{u_{nt}}{u_t})  \notag \\
&+2[ \sum_{i \in G}(\hat h_{in,i}- \sum_{k \in G}\frac{{\hat h_{kn} }}{{\hat h_{kk} }} \hat h_{ik,i} ) ] ( \frac{u_{nn}}{u_n}- \sum_{i \in G} \frac{{\hat h_{in} }}{{\hat h_{ii} }}\frac{{u_{in} }}{{u_n }}-\frac{u_{nt}}{u_t})\notag  \\
&+2u_n u_t u_{nt}[\frac{{u_{nt} }}{{u_t }} -\sum_{i \in G} \frac{{\hat h_{in} }}{{\hat h_{ii} }}\frac{{u_{it} }}{{u_t }}]^2  +u_n ^2 u_{tt} C,
\end{align}
where
\begin{align}  \label{6.99}
C=& 2 \frac{u_{nt}}{u_t} [2 \sum_{i \in G}  \frac{{u_{it} }}{{u_t }}\frac{{\hat h_{in} }}{{\hat h_{ii} }}- \frac{{u_{nt} }}{{u_t }} - (\frac{{u_{nn} }}{{u_n }} - \sum_{i \in G} \frac{{\hat h_{in} }}{{\hat h_{ii} }}\frac{{u_{in} }}{{u_n }})]  \notag\\
&-  2[ \frac{u_{nn}}{u_n}- \sum_{i \in G} \frac{{\hat h_{in} }}{{\hat h_{ii} }}\frac{{u_{in} }}{{u_n }} ][\frac{{u_{nn} }}{{u_n }} - \frac{{u_{nt} }}{{u_t }}- \sum_{i \in G} \frac{{\hat h_{in} }}{{\hat h_{ii} }}\frac{{u_{in} }}{{u_n }}]  \notag\\
& + [\frac{{u_{nt} ^2 }}{{u_t ^2 }} - 2\frac{{u_{nt} }}{{u_t }} \sum_{i \in G}  \frac{{u_{it} }}{{u_t }}\frac{{\hat h_{in} }}{{\hat h_{ii} }} + (\frac{{u_{nn} }}{{u_n }} - \sum_{i \in G} \frac{{\hat h_{in} }}{{\hat h_{ii} }}\frac{{u_{in} }}{{u_n }})^2 ]  \notag\\
=& -\frac{{u_{nt} ^2 }}{{u_t ^2 }}-(\frac{{u_{nn} }}{{u_n }} - \sum_{i \in G} \frac{{\hat h_{in} }}{{\hat h_{ii} }}\frac{{u_{in} }}{{u_n }})^2
+2\frac{u_{nt}}{u_t} \sum_{i \in G}  \frac{{u_{it} }}{{u_t }}\frac{{\hat h_{in} }}{{\hat h_{ii} }} \notag\\
=& -(\frac{{u_{nn} }}{{u_n }} -\frac{{u_{nt} }}{{u_t  }}- \sum_{i \in G} \frac{{\hat h_{in} }}{{\hat h_{ii} }}\frac{{u_{in} }}{{u_n }})^2
-2\frac{{u_{nt} }}{{u_t  }}(\frac{{u_{nn} }}{{u_n }} - \sum_{i \in G} \frac{{\hat h_{in} }}{{\hat h_{ii} }}\frac{{u_{in} }}{{u_n }})+2\frac{u_{nt}}{u_t} \sum_{i \in G}  \frac{{u_{it} }}{{u_t }}\frac{{\hat h_{in} }}{{\hat h_{ii} }},
\end{align}
and
\begin{align}  \label{6.100}
2u_n u_t u_{nt}[\frac{{u_{nt} }}{{u_t }} -\sum_{i \in G} \frac{{\hat h_{in} }}{{\hat h_{ii} }}\frac{{u_{it} }}{{u_t }}]^2
=&2u_n u_t u_{nt}[\frac{{u_{nt} }}{{u_t }} -\sum_{i \in G} \frac{{\hat h_{in} }}{{\hat h_{ii} }}(\frac{{u_{in} }}{{u_n }}-\frac{{\hat h_{in} }}{{u_n u_t^2}})]^2  \notag \\
=&2u_n u_t u_{nt}[\frac{{u_{nt} }}{{u_t }} +\frac{{\hat h_{nn} }}{{u_n u_t^2}}-\sum_{i \in G} \frac{{\hat h_{in} }}{{\hat h_{ii} }}\frac{{u_{in} }}{{u_n }}]^2  \notag \\
=&2u_n u_t u_{nt}[\frac{{u_{nn} }}{{u_n}}-\frac{{u_{nt} }}{{u_t }} +\frac{{u_n^2 u_{tt} }}{{u_n u_t^2}}-\sum_{i \in G} \frac{{\hat h_{in} }}{{\hat h_{ii} }}\frac{{u_{in} }}{{u_n }}]^2  \notag \\
=&2u_n u_t u_{nt}[\frac{{u_{nn} }}{{u_n}}-\frac{{u_{nt} }}{{u_t }} -\sum_{i \in G} \frac{{\hat h_{in} }}{{\hat h_{ii} }}\frac{{u_{in} }}{{u_n }}]^2 \\
&+4u_n u_t u_{nt}[\frac{{u_{nn} }}{{u_n}}-\frac{{u_{nt} }}{{u_t }}-\sum_{i \in G} \frac{{\hat h_{in} }}{{\hat h_{ii} }}\frac{{u_{in} }}{{u_n }}]  \frac{{u_n^2 u_{tt} }}{{u_n u_t^2}} +2u_n u_t u_{nt}[\frac{{u_n^2 u_{tt} }}{{u_n u_t^2}}]^2.\notag
\end{align}
So
\begin{align}  \label{6.101}
&2u_n u_t u_{nt}[\frac{{u_{nt} }}{{u_t }} -\sum_{i \in G} \frac{{\hat h_{in} }}{{\hat h_{ii} }}\frac{{u_{it} }}{{u_t }}]^2  +u_n ^2 u_{tt} C  \notag \\
=&2u_n u_t u_{nt}[\frac{{u_{nn} }}{{u_n}}-\frac{{u_{nt} }}{{u_t }} -\sum_{i \in G} \frac{{\hat h_{in} }}{{\hat h_{ii} }}\frac{{u_{in} }}{{u_n }}]^2 +4u_n u_t u_{nt}[\frac{{u_{nn} }}{{u_n}}-\frac{{u_{nt} }}{{u_t }}-\sum_{i \in G} \frac{{\hat h_{in} }}{{\hat h_{ii} }}\frac{{u_{in} }}{{u_n }}]  \frac{{u_n^2 u_{tt} }}{{u_n u_t^2}} \notag \\
&+2u_n u_t u_{nt}[\frac{{u_n^2 u_{tt} }}{{u_n u_t^2}}]^2-(\frac{{u_{nn} }}{{u_n }} -\frac{{u_{nt} }}{{u_t  }}- \sum_{i \in G} \frac{{\hat h_{in} }}{{\hat h_{ii} }}\frac{{u_{in} }}{{u_n }})^2u_n ^2 u_{tt} \notag \\
&-2\frac{{u_{nt} }}{{u_t  }}(\frac{{u_{nn} }}{{u_n }}- \sum_{i \in G} \frac{{\hat h_{in} }}{{\hat h_{ii} }}\frac{{u_{in} }}{{u_n }})u_n ^2 u_{tt}+2\frac{u_{nt}}{u_t} \sum_{i \in G}  \frac{{u_{it} }}{{u_t }}\frac{{\hat h_{in} }}{{\hat h_{ii} }}u_n ^2 u_{tt} \notag \\
=&[\frac{{u_{nn} }}{{u_n}}-\frac{{u_{nt} }}{{u_t }} -\sum_{i \in G} \frac{{\hat h_{in} }}{{\hat h_{ii} }}\frac{{u_{in} }}{{u_n }}]^2(2u_n u_t u_{nt}-u_n ^2 u_{tt}) \notag \\
 &+2u_n u_t u_{nt}[\frac{{u_{nn} }}{{u_n}}-2\frac{{u_{nt} }}{{u_t }}-\sum_{i \in G} \frac{{\hat h_{in} }}{{\hat h_{ii} }}\frac{{u_{in} }}{{u_n }}+\sum_{i \in G}  \frac{{u_{it} }}{{u_t }}\frac{{\hat h_{in} }}{{\hat h_{ii} }}+\frac{{u_n^2 u_{tt} }}{{u_n u_t^2}}]  \frac{{u_n^2 u_{tt} }}{{u_n u_t^2}}  \notag \\
 =&[\frac{{u_{nn} }}{{u_n}}-\frac{{u_{nt} }}{{u_t }} -\sum_{i \in G} \frac{{\hat h_{in} }}{{\hat h_{ii} }}\frac{{u_{in} }}{{u_n }}]^2(2u_n u_t u_{nt}-u_n ^2 u_{tt}) \notag \\
 =&-[\frac{{u_{nn} }}{{u_n}}-\frac{{u_{nt} }}{{u_t }} -\sum_{i \in G} \frac{{\hat h_{in} }}{{\hat h_{ii} }}\frac{{u_{in} }}{{u_n }}]^2(\sum\limits_{i \in G} {\frac{{\hat h_{in} ^2 }}{{\hat h_{ii} }}}
 + \sum\limits_{i \in G} {\hat h_{ii} }  - u_t ^3  ).
\end{align}
Then we have
\begin{align}   \label{6.102}
&A+2 \frac{u_{nt}}{u_t}  B_1 + 2[ \frac{u_{nn}}{u_n}- \sum_{i \in G} \frac{{\hat h_{in} }}{{\hat h_{ii} }}\frac{{u_{in} }}{{u_n }} ]B_2+ B_4 +II \notag \\
\sim& 2[\sum_{i \in G}\frac{{\hat h_{in} }}{{\hat h_{ii} }}(\hat h_{in,n}-\sum_{j \in G}\frac{{\hat h_{jn} }}{{\hat h_{jj} }} h_{ij,n} )]( \frac{u_{nn}}{u_n}- \sum_{i \in G} \frac{{\hat h_{in} }}{{\hat h_{ii} }}\frac{{u_{in} }}{{u_n }}-\frac{u_{nt}}{u_t})  \notag \\
&+2[ \sum_{i \in G}(\hat h_{in,i}- \sum_{k \in G}\frac{{\hat h_{kn} }}{{\hat h_{kk} }} \hat h_{ik,i} ) ] ( \frac{u_{nn}}{u_n}- \sum_{i \in G} \frac{{\hat h_{in} }}{{\hat h_{ii} }}\frac{{u_{in} }}{{u_n }}-\frac{u_{nt}}{u_t})\notag  \\
&+[ \frac{u_{nn}}{u_n}-\frac{u_{nt}}{u_t}- \sum_{i \in G} \frac{{\hat h_{in} }}{{\hat h_{ii} }}\frac{{u_{in} }}{{u_n }} ]^2 u_{t}^3\notag \\
&-[\frac{{u_{nn} }}{{u_n}}-\frac{{u_{nt} }}{{u_t }} -\sum_{i \in G} \frac{{\hat h_{in} }}{{\hat h_{ii} }}\frac{{u_{in} }}{{u_n }}]^2(\sum\limits_{i \in G} {\frac{{\hat h_{in} ^2 }}{{\hat h_{ii} }}} + \sum\limits_{i \in G} {\hat h_{ii} }   ).
\end{align}

So
\begin{align}
&\frac{1}{2}\left[ {\Big( \Delta \hat h_{nn}-\hat h_{nn,t} \Big)
- 2\sum\limits_{i \in G} {\frac{{\hat h_{in} }} {{\hat h_{ii}}}\Big( \Delta \hat h_{in}-\hat h_{in,t} \Big)+\sum\limits_{i,j
\in G} {\frac{{\hat h_{in} }} {{\hat h_{ii} }}\frac{{\hat h_{jn} }} {{\hat h_{jj}
}}\Big( \Delta \hat h_{ij} -\hat h_{ij,t} \Big)} } } \right] \notag\\
&- \sum\limits_{i \in G} {\frac{1} {{\hat h _{ii}
}}\sum_{\a=1}^n\left[ {\hat h_{in,\alpha }  - \sum\limits_{j \in G} {\frac{{\hat h_{jn} }}
{{\hat h_{jj} }}\hat h_{ij,\alpha } } } \right]^2} \notag  \\
 =& [ I + II ] - \sum\limits_{i \in G} {\frac{1} {{\hat h _{ii}
}}\left[ {\hat h_{in,n }  - \sum\limits_{j \in G} {\frac{{\hat h_{jn} }}
{{\hat h_{jj} }}\hat h_{ij,n} } } \right]^2}- \sum\limits_{i \in G} {\frac{1} {{\hat h _{ii}
}}\sum_{\a \in G}\left[ {\hat h_{in,\alpha }  - \sum\limits_{j \in G} {\frac{{\hat h_{jn} }}
{{\hat h_{jj} }}\hat h_{ij,\alpha } } } \right]^2}   \notag  \\
 =& [A+2 \frac{u_{nt}}{u_t}  B_1 + 2( \frac{u_{nn}}{u_n}- \sum_{i \in G} \frac{{\hat h_{in} }}{{\hat h_{ii} }}\frac{{u_{in} }}{{u_n }})B_2+ B_4  +II]+B_3+B_5 \notag \\
&- \sum\limits_{i \in G} {\frac{1} {{\hat h _{ii}
}}\left[ {\hat h_{in,n }  - \sum\limits_{j \in G} {\frac{{\hat h_{jn} }}
{{\hat h_{jj} }}\hat h_{ij,n} } } \right]^2}  - \sum\limits_{i \in G} {\frac{1} {{\hat h _{ii}
}}\left[ {\hat h_{in,i }  - \sum\limits_{j \in G} {\frac{{\hat h_{jn} }}
{{\hat h_{jj} }}\hat h_{ij,i} } } \right]^2}\notag \\
&- \sum\limits_{i \in G} {\frac{1} {{\hat h _{ii}
}}\sum_{\a \in G, \a \ne i}\left[ {\hat h_{in,\a }  - \sum\limits_{j \in G} {\frac{{\hat h_{jn} }}
{{\hat h_{jj} }}\hat h_{ij,\alpha } } } \right]^2}   \notag  \\
\sim&- \sum\limits_{i \in G} {\frac{1} {{\hat h _{ii}
}}\left[ {\hat h_{in,n }  - \sum\limits_{j \in G} {\frac{{\hat h_{jn} }}
{{\hat h_{jj} }}\hat h_{ij,n} } }-\hat h_{in}[ \frac{u_{nn}}{u_n}-\frac{u_{nt}}{u_t}- \sum_{i \in G} \frac{{\hat h_{in} }}{{\hat h_{ii} }}\frac{{u_{in} }}{{u_n }} ] \right]^2}  \notag \\
&- \sum\limits_{i \in G} {\frac{1} {{\hat h _{ii}
}}\left[ {\hat h_{in,i }  - \sum\limits_{j \in G} {\frac{{\hat h_{jn} }}
{{\hat h_{jj} }}\hat h_{ij,i} } }-\hat h_{ii}[ \frac{u_{nn}}{u_n}-\frac{u_{nt}}{u_t}- \sum_{i \in G} \frac{{\hat h_{in} }}{{\hat h_{ii} }}\frac{{u_{in} }}{{u_n }} ]  \right]^2}\notag\\
&- \sum\limits_{i \in G} {\frac{1} {{\hat h _{ii}
}}\sum_{\a \in G, \a \ne i}\left[ {\hat h_{in,\a }  - \sum\limits_{j \in G} {\frac{{\hat h_{jn} }}
{{\hat h_{jj} }}\hat h_{ij,\alpha } } } \right]^2} +[ \frac{u_{nn}}{u_n}-\frac{u_{nt}}{u_t}- \sum_{i \in G} \frac{{\hat h_{in} }}{{\hat h_{ii} }}\frac{{u_{in} }}{{u_n }} ]^2u_{t}^3\notag  \\
\geq& 0.  \notag
\end{align}

So \eqref{6.75} holds, and the proof of the theorem is complete.
\end{proof}

\newpage

\chapter{The Strict Convexity of Space-time Level Sets}

In this chapter, we study the solution of Borell \cite{Bo82} and prove Theorem \ref{th1.4} in Section 4.1, by utilizing the constant rank theorem of spatial level sets and space-time level sets, i.e. Theorem \ref{th3.1} and Theorem \ref{th1.3}. Similarly, we can prove Theorem \ref{th1.2} in Section 4.2.

\section{The strict convexity of space-time level sets of Borell's solution}

In \cite{Bo82}, Borell studied  problem \eqref{1.3}-\eqref{1.4} and proved that the space-time level sets $\partial\Sigma^c_{x,t}$ of the solution $u$ to \eqref{1.3}-\eqref{1.4} are convex for $c\in(0,1)$ (that is Theorem \ref{th1.1}). Here we want to refine the result of Borell by proving Theorem \ref{th1.4}, that is the strict convexity of space-time level sets of Borell's solution.

\text {Step 1}: First we notice that, thanks to \cite{Bo82}, we know the space-time level sets of $u$ are all convex. Then we can use Theorem \ref{th3.1} to get that the second fundamental form of spatial level sets $\partial\Sigma_x^{c,t}=\{x \in \Omega \,:\, u(x,t) = c\}$
has the constant rank property in $\Omega$ for all $c\in (0,1)$,  i.e. if the
rank of $II_{\partial{\Sigma}^{c,t}_x}$ attains its minimum $l_0$ $(0\leq l\leq
n-1)$ at some point $(x_0,t_0) \in \Omega \times (0,T)$, then the rank of
$II_{\partial{\Sigma}^{c}_{x,t}}$ is constant on $\Omega\times (0,t_0]$. On the other hand
Hopf lemma implies \eqref{1.7}, which in turn implies that the spatial level set $\partial\Sigma_x^{c,t}$ is a closed convex  $(n-1)$-dimensional hypersurface whose second fundamental form has positive Gauss curvature (then full rank) at least at one point for any $c \in (0,1)$ and any fixed $t \in (0, +\infty)$. Then we finally get that $\partial\Sigma_x^{c,t}$ has full rank $n-1$ in $\Omega\times(0, +\infty)$. That is, $u$ is spatial strictly quasiconcave.

\text {Step 2}:  Since  the  space-time level sets $\Sigma^c_{x,t} =
\{(x,t) \in \bar{\Omega} \times [0, +\infty)| u(x,t)\ge c\}$ for  $0< c<1 $ are
convex, we can use Theorem \ref{th1.3} to get the second fundamental
form of the space-time level sets of $u$ has the constant rank property,  i.e. if the rank of $II_{\partial{\Sigma}^{c}_{x,t}}$ attains its
minimum rank $l_0$ $(0\leq l\leq n)$ at some point $(x_0,t_0) \in \Omega \times
(0,+\infty)$, then the rank of $II_{\partial{\Sigma}^{c}_{x,t}}$ is constant on $\Omega\times
(0,t_0]$. From Step 1, we know the rank of $II_{\partial{\Sigma}^{c}_{x,t}}$ is at least $n-1$. Hence from Theorem \ref{th1.3}, there exist $T_0 \in [0, +\infty]$, such that
\begin{align*}
&\mathbf{Rank}( II_{\partial\Sigma^{c}_{x,t}} (x,t) ) \equiv n-1, \quad \text{ for any } (x,t) \in \Omega \times (0, T_0];\\
&\mathbf{Rank}( II_{\partial\Sigma^{c}_{x,t}} (x,t) ) \equiv n, \quad \text{ for any } (x,t) \in \Omega \times (T_0, +\infty).
\end{align*}

In the following, we prove $T_0 < +\infty$. Otherwise, $T_0 = + \infty$, and $\mathbf{Rank}( II_{\partial\Sigma^{c}_{x,t}} (x,t) ) \equiv n-1$ for any $(x,t) \in \Omega \times (0, +\infty)$. Then we know its null space is parallel in $(x, t)\in \Omega \times (0, +\infty)$. As in Gabriel \cite{Ga57} and Lewis
\cite{Le77}, suppose further that at a certain point $P_0(x_0, t_0)\in \Omega
\times (0, +\infty)$, there is a tangential direction $v_0$ of the level
surface of $u$ through $P_0$ for which the normal curvature of the level
surface is zero at $P_0$; then the level surfaces of $u$ in $\R^{n}\times \R^{+}$
are all cones with a common vertex lying on the special tangent $v_0$ at $P_0$.

Case I. The tangential direction $v_0$ is not parallel to the time direction
$t$: since the domain $\Omega$ is bounded, the splitting line through $P_0$ with direction $v_0$ should meet
the boundary of the domain, contradicting $0<u(P_0)<1$ and the regularity of $u$ (which is continuous up to the boundary).

Case II. The tangential direction $v_0$ is  parallel to the time direction $t$: we know $u_t >0$ and this case is also impossible.

Then the second fundamental form of every space-time level sets of $u$
has full rank $n$, that is $\partial\Sigma_{x,t}^c$ has everywhere positive Gauss curvature for every $c\in(0,1)$.

\section{Proof of Theorem \ref{th1.2}}

In this section we prove Theorem \ref{th1.2}, by utilizing the constant rank theorem of spatial level sets and space-time level sets, i.e. Theorem \ref{th3.1} and Theorem \ref{th1.3}.

Let $u$ be a space-time quasiconcave solution of problem \eqref{1.5}. Then the space-time level sets of $u$ are all convex. Hence we can use Theorem \ref{th3.1} to get the constant ran theorem of the second fundamental form of spatial level sets $\partial\Sigma_x^{c,t}=\{x \in \Omega \,:\, u(x,t) = c\}$. Similarly to the above section, we know that the spatial level set $\partial\Sigma_x^{c,t}$ is a closed convex  $(n-1)$-dimensional hypersurface whose second fundamental form has positive Gauss curvature (then full rank) at least at one point for any $c \in (0,1)$ and any fixed $t \in (0, +\infty)$. Then we finally get that $\partial\Sigma_x^{c,t}$ has full rank $n-1$ in $\Omega\times(0, +\infty)$. That is, $u$ is spatial strictly quasiconcave.

Then we use Theorem \ref{th1.3} to get the constant rank theorem of the second fundamental
form of the space-time level sets of $u$. So there exist $T_0 \in [0, +\infty]$, such that
\begin{align*}
&\mathbf{Rank}( II_{\partial\Sigma^{c}_{x,t}} (x,t) ) \equiv n-1, \quad \text{ for any } (x,t) \in \Omega \times (0, T_0];\\
&\mathbf{Rank}( II_{\partial\Sigma^{c}_{x,t}} (x,t) ) \equiv n, \quad \text{ for any } (x,t) \in \Omega \times (T_0, +\infty).
\end{align*}
Similarly to the discussions of the above section, we can prove $T_0 < +\infty$, and then the proof of Theorem \ref{th1.2} is complete.

\newpage

\chapter{Appendix: the proof in dimension $n=2$}

In this Appendix, we prove the Constant Rank Theorem \ref{th1.3} in the plane. Then $\hat{a}=
\left( {\begin{array}{*{20}c}
   {\hat a_{11}} & {\hat a_{12}}  \\
   {\hat a_{21}} & {\hat a_{22}}  \\
\end{array}} \right)
$ and let it attain the minimal
rank $l$ at some point $(x_0, t_0) \in \Omega \times (0, T]$. We assume $l \leq 1$, otherwise there is nothing to prove.

In CASE 1, Theorem \ref{th1.3} holds directly from the constant rank property of the spatial
second fundamental form $a=(a_{11})_{1 \times 1}$. It is easy (see Section 3.1).

In the following, we consider CASE 2 in dimension $n=2$. Since $l \leq 1$, we deal with $l=0$ and $l=1$, respectively.

\section{minimal rank $l=0$}

From CASE 2 of Lemma \ref{lem2.8}, if the minimal rank is $l=0$, we have at $(x_0,t_0)$,
\begin{eqnarray*}
\hat a_{11} =0,\quad \hat a_{22}  = 0,  \quad \hat a_{12} = 0.
\end{eqnarray*}
Then there are a neighborhood $\mathcal {O}$ of $x_0$ and $\delta>0$ such that for any fixed point  $(x,t) \in \mathcal {O}\times (t_0-\delta,
t_0]$, we can choose $e_1, e_2$ such that
\begin{equation}\label{5.6}
u_1(x,t)=0, \quad u_2(x,t)= |\n u(x,t)|>0.
\end{equation}
From Theorem $\ref{th3.1}$, the constant rank theorem holds for the spatial
second fundamental form $a=(a_{11})_{1 \times 1}$. So we can get $a_{11}=0$ for any $(x,t) \in \mathcal {O} \times (t_0-\delta, t_0]$.
Furthermore, $\hat a_{11}=0$.

We set
\begin{eqnarray}\label{5.11}
\phi = \hat a_{11}+\hat a_{22},
\end{eqnarray}
Under the above assumptions, we get
\begin{align*}
&u_{11}=0, \quad u_{22}= u_t, \quad u_{12} =0, \quad u_{1t}=0, \\
&u_2^2 u_{tt}\sim 2u_2u_tu_{2t}-u_t^3
\end{align*}
from the constant rank property of $a = (a_{11})$ and CASE 2 of Lemma \ref{lem2.8}.
Furthermore, by the constant rank property of $a = (a_{11})$ and Lemma \ref{lem2.9}, we can obtain
\begin{align*}
&u_{111}=0, \qquad u_{221}=0, \qquad u_{112}=0, \qquad u_{222}=u_{2t}, \\
&u_{11t}=0, \qquad  u_{22t}=u_{tt}, \qquad u_{12t} \sim 0, \qquad  u_{tt1}\sim 0,  \\
&u_2^2u_{tt2} \sim  2 u_2 u_{2t}^2 - u_t^2 u_{2t}.
\end{align*}

So we get
\begin{align} \label{5.45}
&\Delta \phi-\phi_t  \notag \\
&\sim  \Big(-\frac{|u_t|}{|D u|{u_t}^3}\Big) \frac{1}{\hat W ^2}
{\Big( \Delta \hat h_{22}-\hat h_{22,t} \Big)} \notag \\
&=  \Big(-\frac{|u_t|}{|D u|{u_t}^3}\Big) \frac{1}{\hat W ^2}
\Big[ 4\sum_{\a=1}^2  u_t u_{t\a} u_{22 \a}+ 4\sum_{\a=1}^2 u_2 u_{2\a} u_{tt\a}
-4u_t \sum_{\a=1}^2 u_{2\a}u_{2t\a} -4 u_{2}\sum_{\a=1}^2 u_{t\a}u_{2t\a} \notag \\
&\qquad \qquad \qquad + 2\sum_{\a=1}^2 u_{t\a}^2 u_{22}+ 2 \sum_{\a=1}^2u_{tt} u_{2 \a}^2- 2 u_2 \Delta u_tu_{2t}+  2u_{2} u_{tt}\Delta u_2-4 \sum_{\a=1}^2 u_{2\a} u_{t\a}u_{2t}\Big] \notag \\
&=  \Big(-\frac{|u_t|}{|D u|{u_t}^3}\Big) \frac{1}{\hat W ^2}
\Big[ 4 u_t u_{2t} u_{222}+ 4 u_2 u_{22} u_{tt2}
-4u_t  u_{22}u_{22t} -4 u_{2} u_{2t}u_{22t} \notag \\
&\qquad \qquad \qquad \quad + 2 u_{2t}^2 u_{22}+ 2 u_{tt} u_{22}^2- 2 u_2 u_{tt} u_{2t}+  2u_{2} u_{tt} u_{2t}-4  u_{22} u_{2t}u_{2t}\Big] \notag \\
&\sim  \Big(-\frac{|u_t|}{|D u|{u_t}^3}\Big) \frac{1}{\hat W ^2}
\Big[  2u_t u_{2t}^2 -4 \frac{u_t^3}{u_2} u_{2t} +2  \frac{u_t^5}{u_2^2} \Big] \notag \\
&= \Big(-\frac{|u_t|}{|D u|{u_t}^3}\Big) \frac{1}{\hat W ^2} 2 u_t^3[\frac{u_{22}}{u_2}- \frac{u_{2t}}{u_t} ]^2 \leq 0.
\end{align}
that is
\begin{align}\label{5.44}
\Delta \phi -\phi_t \leq C(\phi +|\n \phi|), \quad \forall (x,t) \in \mathcal {O}\times (t_0-\delta,
t_0].
\end{align}
Finally, by the strong maximum principle and the method of continuity, Theorem \ref{th1.3} holds.

\section{minimal rank $l=1$}

From CASE 2 of Lemma \ref{lem2.8}, if the minimal rank is $l=1$, we have at $(x_0,t_0)$,
\begin{eqnarray*}
\hat a_{22}  = \frac{{\hat a_{12} ^2 }} {{\hat a_{11} }}, \hat a_{11} \geq C_0 >0.
\end{eqnarray*}
Then tthere are a neighborhood $\mathcal {O}$ of $x_0$ and $\delta>0$ such that
\begin{equation}\label{5.49}
\hat a_{11} \geq \frac{C_0}{2} >0 \quad \text{ in } \mathcal {O}\times (t_0-\delta, t_0].
\end{equation}
For any point $(x,t) \in \mathcal {O}\times (t_0-\delta, t_0]$, we can choose $e_1, e_2$ such that
\begin{equation}\label{5.50}
u_1(x,t)=0, \quad u_2(x,t)= |\n u(x,t)|>0.
\end{equation}

We set
\begin{eqnarray}\label{5.51}
\phi = \sigma_{2}(\hat a)= \hat a_{11} \hat a_{22}- \hat a_{12} \hat a_{21}.
\end{eqnarray}
Under the above assumptions, we get from CASE 2 of Lemma \ref{lem2.8},
\begin{align*}
&u_{11}  = \frac{{\hat h_{11} }}{{u_t ^2 }}, \quad u_{22}  = u_t  - \frac{{\hat h_{11} }}{{u_t ^2 }},  \quad u_t ^2 u_{12}  = \hat h_{12}  + u_2 u_t u_{1t},  \\
&u_2 ^2 u_{tt}  \sim \frac{{\hat h_{12} ^2 }}{{\hat h_{11} }} + \hat h_{11}  -u_t ^3  + 2u_2 u_t u_{2t}.
\end{align*}
Direct computations yield
\begin{align*}
&u_t ^2 u_{111}  = \hat h_{11,1},  \quad u_t ^2 u_{221}  =  - \hat h_{11,1}  +u_t^2 u_{1t},  \\
&u_t ^2 u_{112}  = \hat h_{11,2}  - 2\frac{{u_{2t} }}{{u_t }}\hat h_{11}  + 2\frac{{u_{1t} }}{{u_t }}\hat h_{12}  + 2u_2 u_{1t} ^2,  \\
&u_t ^2 u_{222}  =  - \hat h_{11,2}  +u_t ^2 u_{2t}  + 2\frac{{u_{2t} }}{{u_t }}\hat h_{11}  - 2\frac{{u_{1t} }}{{u_t }}\hat h_{12}  - 2u_2 u_{1t} ^2, \\
&u_2 u_t u_{11t}  = \hat h_{11,2}  - \hat h_{12,1}  - 3\frac{{u_{2t} }}{{u_t }}\hat h_{11}  + 3\frac{{u_{1t} }}{{u_t }}\hat h_{12}  + 2u_2 u_{1t} ^2  + u_2 u_{11} u_{tt},  \\
&u_2 u_t u_{22t}  = \hat h_{12,1}  - \hat h_{11,2}  + u_2 u_{22} u_{tt}  + 3\frac{{u_{2t} }}{{u_t }}h_{11}  - 3\frac{{u_{1t} }}{{u_t }}\hat h_{12}  - 2u_2 u_{1t} ^2,  \\
&u_2 u_t u_{12t}  =  - \hat h_{11,1}  - \hat h_{12,2}  + \frac{{u_{1t} }}{{u_t }}\hat h_{11}  + \frac{{u_{2t} }}{{u_t }}\hat h_{12}  +u_2 u_{12} u_{tt},
\end{align*}
and
\begin{align*}
u_2 ^2 u_{tt1}  =& \hat h_{22,1}  - \hat h_{11,1}  - 2\hat h_{12,2}  + 2\frac{{u_{1t} }}{{u_t }}\hat h_{11} + 2\frac{{u_{2t} }}{{u_t }}\hat h_{12} \notag \\
&- 2\frac{{u_{1t} }}{{u_t }}u_t ^2 u_{22} -u_t^2 u_{1t}  + 2u_t u_{12} u_{2t} + 2u_2 u_{1t} u_{2t},  \\
u_2 ^2 u_{tt2}  =& \hat h_{22,2}  - \hat h_{11,2}  + 2\hat h_{12,1}  + 4\frac{{u_{2t} }}{{u_t }}\hat h_{11}  - 4\frac{{u_{1t} }}{{u_t }}\hat h_{12}  - u_t ^2 u_{2t}  - 2u_2 u_{1t} ^2  + 2u_2 u_{2t} ^2.
\end{align*}

At last, we get
\begin{align} \label{5.94}
\Delta \phi-\phi_t
\sim & \Big(-\frac{|u_t|}{|D u|{u_t}^3}\Big)^2 \frac{1}{\hat W ^2} \left[\hat h_{11} [\Big( \Delta \hat h_{22}-\hat h_{22,t} \Big)
- 2\frac{{\hat h_{12} }} {{\hat h_{11}}}\Big( \Delta \hat h_{12}-\hat h_{12,t} \Big)+\frac{{\hat h_{12}^2 }} {{\hat h_{11}^2 }}
\Big( \Delta \hat h_{11} -\hat h_{11,t} \Big)] \right. \notag \\
&\qquad \qquad \qquad \qquad \left. - 2[\hat h_{12,\alpha}-\frac{{\hat h_{12} }} {{\hat h_{11} }}\hat h_{11,\alpha }]^2 \right].
\end{align}
and
\begin{align}\label{5.113}
&\frac{1}{2}\hat h_{11} [\Big( \Delta \hat h_{22}-\hat h_{22,t} \Big)
- 2\frac{{\hat h_{12} }} {{\hat h_{11}}}\Big(\Delta \hat h_{12}-\hat h_{12,t} \Big)+\frac{{\hat h_{12}^2 }} {{\hat h_{11}^2 }}
\Big( \Delta \hat h_{11} -\hat h_{11,t} \Big)]- [\hat h_{12,\alpha}-\frac{{\hat h_{12} }} {{\hat h_{11} }}\hat h_{11,\alpha }]^2   \notag \\
=& \frac{1}{2}\hat h_{11} [\Big( \Delta \hat h_{22}-\hat h_{22,t} \Big)
- 2\frac{{\hat h_{12} }} {{\hat h_{11}}}\Big(\Delta \hat h_{12}-\hat h_{12,t} \Big)+\frac{{\hat h_{12}^2 }} {{\hat h_{11}^2 }}
\Big( \Delta \hat h_{11} -\hat h_{11,t} \Big)] \notag   \\
&- [\hat h_{12,1}-\frac{{\hat h_{12} }} {{\hat h_{11} }}\hat h_{11,1}]^2 - [\hat h_{12,2}-\frac{{\hat h_{12} }} {{\hat h_{11} }}\hat h_{11,2 }]^2  \notag\\
\sim& 2(\hat h_{12,1}  - \frac{{\hat h_{12} }}{{\hat h_{11} }}\hat h_{11,1} )\hat h_{11} [\frac{{u_{22} }}{{u_2 }} - \frac{{u_{2t} }}{{u_t }} - \frac{{\hat h_{12} }}{{\hat h_{11} }}\frac{{u_{12} }}{{u_2 }}] \notag\\
&+ 2(\hat h_{12,2}  - \frac{{\hat h_{12} }}{{\hat h_{11} }}\hat h_{11,2} )\hat h_{12}[\frac{{u_{22} }}{{u_2 }} - \frac{{u_{2t} }}{{u_t }} - \frac{{\hat h_{12} }}{{\hat h_{11} }}\frac{{u_{12} }}{{u_2 }}]   \notag \\
&+[\frac{{u_{22} }}{{u_2 }} - \frac{{u_{2t} }}{{u_t }} - \frac{{\hat h_{12} }}{{\hat h_{11} }}\frac{{u_{12} }}{{u_2 }}]^2 u_t^3 \hat h_{11}  - [\frac{{u_{22} }}{{u_2 }} - \frac{{u_{2t} }}{{u_t }} - \frac{{\hat h_{12} }}{{\hat h_{11} }}\frac{{u_{12} }}{{u_2 }}]^2 (\hat h_{11}  + \hat h_{22} )\hat h_{11}  \notag \\
&- [\hat h_{12,1}-\frac{{\hat h_{12} }} {{\hat h_{11} }}\hat h_{11,1}]^2 - [\hat h_{12,2}-\frac{{\hat h_{12} }} {{\hat h_{11} }}\hat h_{11,2 }]^2  \notag\\
\sim& - [\hat h_{12,1}-\frac{{\hat h_{12} }} {{\hat h_{11} }}\hat h_{11,1}- (\frac{{u_{22} }}{{u_2 }} - \frac{{u_{2t} }}{{u_t }} - \frac{{\hat h_{12} }}{{\hat h_{11} }}\frac{{u_{12} }}{{u_2 }}) \hat h_{11}]^2 \notag \\
&- [\hat h_{12,2}-\frac{{\hat h_{12} }} {{\hat h_{11} }}\hat h_{11,2 }-(\frac{{u_{22} }}{{u_2 }} - \frac{{u_{2t} }}{{u_t }} - \frac{{\hat h_{12} }}{{\hat h_{11} }}\frac{{u_{12} }}{{u_2 }}) \hat h_{12}]^2 \notag \\
&+[\frac{{u_{22} }}{{u_2 }} - \frac{{u_{2t} }}{{u_t }} - \frac{{\hat h_{12} }}{{\hat h_{11} }}\frac{{u_{12} }}{{u_2 }}]^2 u_t^3 \hat h_{11}  \notag  \\
\leq& 0.
\end{align}
Hence we arrive to
\begin{align}\label{5.44}
\Delta \phi -\phi_t \leq C(\phi +|\n \phi|), \quad \forall (x,t) \in \mathcal {O}\times (t_0-\delta,
t_0].
\end{align}
Finally, again by the strong maximum principle and the method of continuity, Theorem \ref{th1.3} holds.

\end{document}